\documentclass[a4paper,10pt,oneside]{articleHJ}
\usepackage{amsfonts,amssymb,amsmath,amsthm}
\usepackage[a4paper, total={6in,9in}]{geometry}
\usepackage[utf8]{inputenc}
\usepackage{graphicx}

\usepackage{subcaption}
\usepackage{tikz}
\usepackage{enumerate}
\usepackage{cases}
\usepackage{multicol}
\usepackage{sidecap}
\usepackage{float}
\usepackage{xcolor}
\definecolor{blue2}{rgb}{0.67, 0.9, 0.93}
\usepackage[color=blue2]{todonotes}
\usepackage{cases}
\usepackage{comment}
\normalfont
\usepackage[T1]{fontenc}
\usepackage[utf8]{inputenc}	
\usepackage[normalem]{ulem}
\usepackage[hidelinks=true]{hyperref}
\usepackage{fancyhdr}

\numberwithin{equation}{section}

\newtheorem{thm}{Theorem}[section]
\newtheorem{lemma}[thm]{Lemma}
\newtheorem{proposition}[thm]{Proposition}
\newtheorem{cor}[thm]{Corollary}
	
\theoremstyle{definition}

\newcommand{\N}{{\mathbb N}}
\newcommand{\R}{{\mathbb R}}
\newcommand{\C}{{\mathbb C}}
\newcommand{\Z}{\mathbb{Z}}

\newcommand{\Zs}{{\mathbb{Z}^2_{\times}}}
\newcommand\norm[1]{\left|\left|{#1}\right|\right|}

\usepackage{setspace}

\begin{document}

\begin{frontmatter}
\title{Curvature-driven front propagation through planar lattices in oblique directions} 

\journal{...}
\author[LD1]{M. Juki\'c\corauthref{coraut}},
\author[LD2]{H. J. Hupkes},
\corauth[coraut]{Corresponding author. }
\address[LD1]{
  Mathematisch Instituut - Universiteit Leiden \\
  P.O. Box 9512; 2300 RA Leiden; The Netherlands \\ Email:  {\normalfont{\texttt{m.jukic@math.leidenuniv.nl}}}
}
\address[LD2]{
  Mathematisch Instituut - Universiteit Leiden \\
  P.O. Box 9512; 2300 RA Leiden; The Netherlands \\ Email:  {\normalfont{\texttt{hhupkes@math.leidenuniv.nl}}}
}

\date{\today}

\begin{abstract}
\singlespacing
In this paper we investigate the long-term behaviour of solutions
to the discrete Allen-Cahn equation posed on a two-dimensional lattice. 
We show that front-like initial conditions
evolve towards a planar travelling wave modulated
by a
phaseshift $\gamma_l(t)$ that depends on the coordinate $l$
transverse to the primary direction of propagation.  This direction
is allowed to be general, but rational, generalizing earlier known results
for the horizontal direction. We show that the behaviour of $\gamma$ can be asymptotically linked to the behaviour of a suitably discretized
mean curvature flow. This allows us to show
that travelling waves propagating in rational directions are nonlinearly 
stable with respect to perturbations that are asymptotically periodic in the transverse direction.
\end{abstract}



\end{frontmatter}

\section{Introduction}\label{sec:intro}

The main goal of this paper is to study the behaviour
of curved wavefronts under the dynamics of the Allen-Cahn lattice differential equation (LDE)
\begin{equation}\label{eqn:intro:AC}
    \dot{u}_{i,j} = 
    u_{i+1,j} + u_{i,j+1} + u_{i-1, j} + u_{i,j-1} - 4 u_{i,j} + g(u_{i,j};a), 
\end{equation}
posed on the planar lattice $(i,j)\in \Z^2$. For concreteness, we consider
the standard bistable nonlinearity
\begin{equation*}
    g(u;a) = u(u-a)(1-u),
    \qquad \qquad a \in (0,1),
\end{equation*}
throughout this introduction. We are interested in 
fronts that move in the rational direction
$(\sigma_h, \sigma_v)\in \Z^2$, which motivates the introduction of the
parallel and transverse coordinates 
\begin{equation}\label{eqn:intro:coordinate}
    n = n(i,j) =  i\sigma_h + j \sigma_v, \qquad 
    \qquad
    l = l(i,j) = i\sigma_v - j\sigma_h
\end{equation}
that we use interchangeably with $(i,j)$; see Figure~\ref{fig:intro:coordinate}.

Our main results state that
initial conditions that are `front-like' in the rough sense that
\begin{equation}
\label{eq:int:init:cond:lde:gen}
    u_{i,j}(0) <a - \epsilon  \qquad \hbox{ for } n(i,j) \ll -1, \qquad  \qquad
    u_{i,j}(0) > a + \epsilon  \qquad \hbox{ for } n(i,j) \gg 1
\end{equation}
holds for some $\epsilon > 0$,
evolve towards an interface of the form
\begin{equation}
\label{eq:int:interface:formation:disc}
    u_{i,j}(t) = \Phi\big( n(i,j) - \gamma_{l(i,j)}(t) \big).
\end{equation}
Here the special case $\gamma_l(t) = ct$ represents the well-known planar travelling wave solution 
to \eqref{eqn:intro:AC} that travels in the direction $(\sigma_h, \sigma_v)$ and connects the two stable equilibria
\begin{equation}
\label{eq:int:phi:bc}
    \lim_{\xi \to - \infty}
    \Phi(\xi) = 0,
    \qquad \qquad
    \lim_{\xi \to + \infty} \Phi(\xi) = 1.
\end{equation}
In general however we show that the dynamics of the phase $\gamma_l$ can be well-approximated by a discrete mean-curvature flow.
This generalizes the results from \cite{jukic2019dynamics} where we only considered the horizontal direction
and extends the known basin
of attraction for planar travelling waves beyond the settings considered
in \cite{hoffman2015multi,hoffman2017entire}.
The misalignment of the propagation direction with the underlying lattice causes
several mathematical intricacies that we resolve throughout this work.

\paragraph{Modelling background}
Lattice differential equations arise in numerous problems in which the underlying discrete spatial topology plays an important role. For example,  in~\cite{bell1981some, bell1984threshold, keener1987propagation}, the authors use LDEs to model \textit{saltatory conduction}, which describes the `hopping' behaviour of action potentials propagating through myelinated nerve axons.    In population dynamics, 
two-dimensional LDEs are used  to model the strong Allee effect on patchy landscapes; see \cite{keitt2001allee, taylor2005allee}. In both of these examples it is necessary to include the spatial heterogeneity of the domain into the model in order to simulate effects such as wave-pinning. 
Lattice models have also been used in many other fields, such as material science, morphology and statistical mechanics \cite{cahn1960theory, cook1969model, merks2007canalization, bates1999discrete}. For a more extensive list of references we refer the reader to the book by Keener and Sneyd~\cite{keener1998mathematical} or the  surveys~\cite{hupkes2018traveling,kevrekidis2011non}.

\paragraph{Motivation}
In order to set the stage, we briefly discuss the continuous counterpart of~\eqref{eqn:intro:AC}.
This is
the well-known Allen-Cahn PDE
\begin{equation}\label{eqn:intro:AC_continuous}
    u_t = \kappa \big[ u_{xx} + u_{yy} \big] + g(u;a),
\end{equation}
where we have included a diffusion constant $\kappa > 0$. Planar travelling front solutions of the form
\begin{equation}
\label{eq:int:pde:trv:wave}
u(x,y,t) = \Phi(x \cos\theta + y \sin \theta - ct)
\end{equation}
play a key role towards understanding the global behaviour of \eqref{eqn:intro:AC_continuous} \cite{aronson1978multidimensional}. They can be found 
\cite{fife2013mathematical} by solving the travelling wave ODE
\begin{equation}\label{eqn:intro:wave_ode}
    -c\Phi'(\xi) = \kappa \Phi''(\xi) + g(\Phi(\xi);a),
\end{equation}
which does not depend on the direction of propagation $(\cos \theta, \sin \theta)$.
In addition, the dependence on the diffusion coefficient $\kappa$ can be eliminated through the spatial rescaling
\begin{equation}
\label{eq:int:resc:kappa}
    \xi \mapsto   \xi / \sqrt{\kappa}, \qquad c \mapsto c/\sqrt{\kappa}.
\end{equation}

This was recently exploited by
 Matano, Mori \& Nara~\cite{matano2019asymptotic}, who studied an anisotropic version
 of \eqref{eqn:intro:AC_continuous} by allowing the diffusion coefficients
 to depend on $\nabla u$. In terms of the travelling wave
 ODE
 \eqref{eqn:intro:wave_ode},
 this effectively introduces a direction-dependence $\kappa = \kappa(\theta)$.
 The spatial rescalings \eqref{eq:int:resc:kappa} subsequently point to a natural anistropic metric that can be used to analyze the long-time evolution of expansion waves. Indeed,
 for initial conditions $u_0$ that satisfy
\begin{equation}
\label{eq:int:ic:comp}
    \min_{|(x,y)|\leq L} u_0(x,y) > a ,
    \qquad
    \qquad
    \limsup_{|(x,y)|\to\infty} u_0(x, y) < a
\end{equation}
for some $L \gg 1$,
the asymptotic behaviour of the level set
$$\Gamma(t):=\left\{(x,y)\in \R^2: u(x, y, t) = a \right\}$$  is well approximated by the boundary of the Wulff shape~\cite{cerf2006wulff, osher1997wulff, wul1901frage}
associated to this metric,
expanding at a speed of $c - [ct]^{-1}$. This latter term can be seen as a correction for curvature-driven effects and also appears in the earlier isotropic 
studies \cite{uchiyama1985asymptotic,jones1983spherically,ROUSSIER2003}. The key point is that the expanding Wulff shape
is a self-similar solution to
an anisotropic mean curvature flow that also 
underpins the large-time behaviour of curved wavefronts.

Returning to our LDE
\eqref{eqn:intro:AC}, we emphasize that anisotropic effects are a natural consequence of the broken rotational symmetry,
but they cannot be readily transformed away by spatial rescalings such as \eqref{eq:int:resc:kappa}. Nevertheless, initial numerical experiments such as those in~\cite{BachBrenda} indicate that the Wulff shape
also 
plays an important role
in the long-term evolution 
of initial conditions such as 
\eqref{eq:int:ic:comp}, but that the behaviour near the corners is rather subtle.
One of our main longer term goals is to gain a detailed understanding of this expansion mechanism. A key intermediate step that we pursue in this paper is to understand how discretized curvature flows interact
with the dynamics of \eqref{eqn:intro:AC}.

\paragraph{Curved PDE fronts}

From a technical point of view, our work is chiefly inspired by
the results obtained in~\cite{Matano} by Matano and Nara. They considered the Cauchy problem for equation~\eqref{eqn:intro:AC_continuous} with an initial condition that roughly satisfies 
\begin{equation*}
     %
     u(x,y,0) <a - \epsilon  \qquad \hbox{ for } x \ll -1, \qquad  \qquad
    u(x,y,0) > a + \epsilon  \qquad \hbox{ for } x \gg 1,
\end{equation*}
again with $\epsilon > 0$.
The authors show that for $t\gg 1$ the solution $u$ becomes monotone around $\Phi(0) = \frac{1}{2}$, which, via the implicit theorem argument, 
allows a phase $\gamma(y,t)$ to be defined via the requirement
\begin{equation}\label{eqn:intro:phase}
    u(\gamma(y,t), y, t) = \Phi(0). 
\end{equation}
This phase 
is particularly convenient because it determines the large time behaviour of the solution $u$ via the asymptotic limit
\begin{equation}\label{eqn:intro:u_approx}
    \lim_{t\to\infty} |u(x,y,t) - \Phi\big(x - \gamma(y,t)\big)| = 0.
\end{equation}
 Moreover,  the authors showed that the phase $\gamma$ can be closely tracked by solutions $\theta$ to the PDE
\begin{equation}\label{eqn:intro:theta}
    \theta_t =  \theta_{yy} + \frac{c}{2}\theta_y^2 + c,
\end{equation}
by constructing super- and sub-solutions to~\eqref{eqn:intro:AC_continuous} of the form 
\begin{equation}\label{eqn:intro:super_sub_cont}
    u^\pm (x,y,t) = \Phi\left(\dfrac{x-\theta(y,t)}{\sqrt{1+\theta_y^2}} \pm Z(t)\right) \pm z(t),
\end{equation}
where $Z$ and $z$ are small correction terms compensating for the initial differences in phase and amplitude. 
The main advantage of the PDE~\eqref{eqn:intro:theta} is that it transforms into a standard heat equation via the Cole-Hopf transformation, which leads to explicit expressions for the solution. 

Describing the phase $\gamma$ with the dynamics of the PDE~\eqref{eqn:intro:theta}
has  two main advantages~\cite{Matano}. First, 
the solution $\theta$ approximates solutions of the mean curvature flow with a drift term $c$, allowing for a physical interpretation of the phase $\gamma$. Second, 
this description
can be used to
establish convergence results for initial conditions $u^0$
that are uniquely ergodic,
which includes
the case that $u^0$ is periodic or almost-periodic
in the transverse direction.
These results are hence part of an ever-increasing family of stability results for travelling fronts in dissipative PDEs,
which include the classic one-dimensional papers
\cite{fife1977approach,SATTINGER1977}
and their higher-dimensional counterparts
\cite{kapitula1997multidimensional,XIN1992,LEVXIN1992}.

\paragraph{Discrete setting}
Substituting the planar wave Ansatz
\begin{equation}
\label{eq:int:plane:wave:lde}
    u_{ij}(t) = \Phi( n (i,j)- c t )
\end{equation}
into the LDE \eqref{eqn:intro:AC}, we
see that the wave pair
$(c, \Phi)$
must satisfy the
mixed functional differential equation (MFDE)
\begin{equation}\label{eqn:intro:mfde}
    -c\Phi'(\xi) = \Phi(\xi + \sigma_h) + \Phi(\xi-\sigma_h) + \Phi(\xi + \sigma_v) + \Phi(\xi-\sigma_v) -  4\Phi(\xi) + g\big(\Phi(\xi);a\big),
\end{equation}
which we consider together with the boundary conditions \eqref{eq:int:phi:bc}. This MFDE has been well-studied by now and various detailed existence and uniqueness results can be found in the seminal paper \cite{Mallet-Paret1999} and the survey \cite{hupkes2018traveling}. For now we simply point out the qualitative differences between the $c = 0$ and $c \neq 0$ cases and the explicit dependence on the propagation direction, which can be rather delicate. Indeed, for a single fixed $a \in (0,1)$ certain directions can support freely travelling waves with smooth profiles, while others only feature pinned step-like profiles \cite{CMPVV99,hoffman2010universality}.

\begin{figure}
    \centering
    \includegraphics[width=\linewidth]{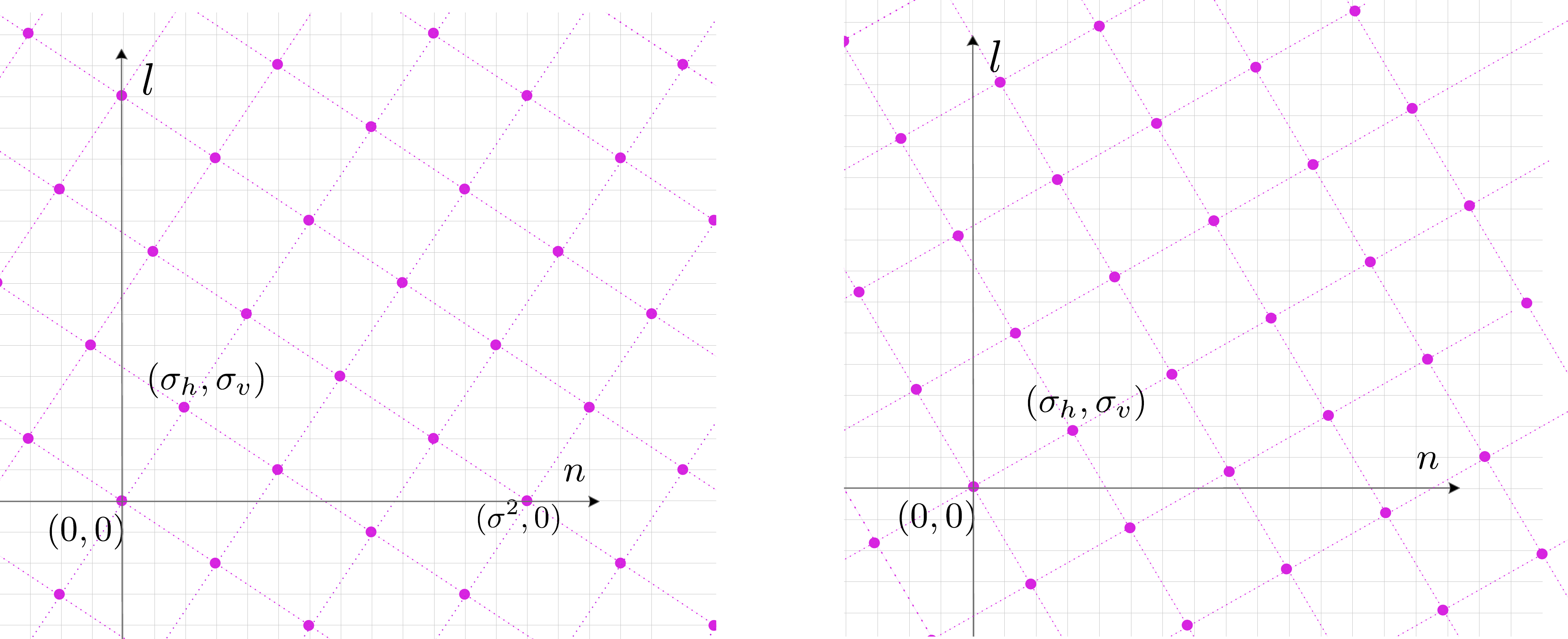}
    \caption{Both panels show the sublattice $\Zs$ obtained after the coordinate transformation~\eqref{eqn:intro:coordinate}, for the rational direction $(\sigma_h, \sigma_v) =(2, 3)$ on the left and the irrational angle  $\pi/6$ on the right. 
    We see that the left lattice is a proper subset of $\Z^2$. On the right however the purple dots only coincide with $\Z^2$ at the origin. Moreover,  the sets $\{i\sigma_h + j\sigma_v: (i,j)\in \Z^2\}$ and $\{i\sigma_v - j\sigma_h: (i,j)\in \Z^2\}$ are both dense in $\R$. This feature significantly differentiates the analysis between the rational and irrational directions. }
    \label{fig:intro:coordinate}
\end{figure}

For our purposes in this paper, the main consequence of the spatial discreteness is that it is no longer possible to construct sub- and super-solutions by
applying relatively straightforward phase modulations to the profile $\Phi$ as in \eqref{eqn:intro:super_sub_cont}. Indeed, the shifts
in \eqref{eqn:intro:mfde} prevent us from simply factorizing out a common factor $\Phi'(\xi)$ from the associated residuals as was possible in the series \cite{matano2019asymptotic,NARAMATANO2011,BHM}. Inspired by normal form theory, we circumvent this problem by using a super-solution Ansatz of the form
\begin{multline}\label{eqn:intro:super_sol_ansatz}
     u^+_{n,l}(t) = \Phi\big(n-\theta_l(t) + Z(t)\big)  +  \sum_{k=-N}^N p_k\big(n-\theta_l(t) + Z(t)\big)\big(\theta_{l+k}(t) - \theta_l(t)\big) \\
       + \sum_{k=-N}^N \sum_{k'=-N}^N q_{k, k'}\big(n-\theta_l(t) + Z(t)\big)\big(\theta_{l+k}(t) - \theta_l(t)\big)\big(\theta_{l+k'}(t) - \theta_l(t)\big)  + z(t),     
\end{multline}
in which $N = 2 \max\{|\sigma_h|, |\sigma_v|\}$. The 
auxiliary functions $(p_k)$, and $(q_{k,k'})$ are chosen in such a way that the dangerous slowly decaying terms caused by the lattice anisotropy are cancelled. To achieve this, it is necessary to carefully analyze the spectral stability properties of the underlying planar wave $(c,\Phi)$ and exploit the Fredholm theory for linear MFDEs that was developed by Mallet-Paret \cite{Mallet-Paret1999_Fredholm}.

The Ansatz \eqref{eqn:intro:super_sol_ansatz} (but with 
different functions $p$, $q$
and $\theta$) first appeared
in \cite{hoffman2017entire} - where it was used to study the evolution of initial conditions of the form
\begin{equation}
    u_{i,j}(0) = \Phi\big(n(i,j)\big) + v^0_{i, j},
    \qquad \qquad
     \lim_{|i|+|j| \to \infty} |v^0_{i,j}| \to 0.
\end{equation}
The authors established algebraic decay rates for the convergence
\begin{equation}
    u_{ij}(t) \to \Phi(n(i,j) - ct),
\end{equation}
hence establishing the stability of the planar wave
\eqref{eq:int:plane:wave:lde} under localized perturbations,
which form a (restrictive) subset of the  general class
\eqref{eq:int:init:cond:lde:gen} considered here.

The main novel aspect compared to \cite{hoffman2017entire} is that we need to incorporate nonlinear terms in the evolution of $\theta$ in order to capture the curvature-driven interface dynamics resulting from the non-local nature of the perturbations.
Indeed, our evolution equation
for $\theta$ takes the form
\begin{equation}\label{eqn:intro:theta_nonlinear}
    \dot{\theta}_l(t) = \dfrac{1}{d}\sum_{k=-N}^{N}  a_k\left(e^{d(\theta_{l+k}(t) - \theta_l(t))} -1 \right)  + c,
\end{equation}
for a set of coefficients $(a_k)$ that is prescribed by the normal form analysis discussed above.
For now, we simply mention that the parameter $d$ can be directly expressed in terms of
important geometric and spectral quantities associated to the wave $(c, \Phi)$. As we discuss in the sequel, this will allow us to make the connection between \eqref{eqn:intro:theta_nonlinear} and a discretized mean curvature flow.

As in the continuous case, solutions to \eqref{eqn:intro:theta_nonlinear} can be used to approximate the behaviour of the phase
$\gamma$ appearing
in \eqref{eq:int:interface:formation:disc}. This control is sufficiently strong to 
establish the convergence
$\gamma(t) \to c t+ \mu$
for initial conditions of the form
\begin{equation}
    u_{i,j}(0) = \Phi\big(n(i,j) - \kappa_l\big) + v^0_{i, j},
    \qquad \qquad
     \lim_{|i|+|j| \to \infty} |v^0_{i,j}| \to 0,
\end{equation}
where $\kappa_l$ is an arbitrary periodic sequence.
The main significance
compared to the earlier
results in \cite{hoffman2015multi,hoffman2017entire} is that this corresponds to an `infinite-energy' shift in the underlying wave position, during which the periodic wrinkles are flattened out under the flow of
\eqref{eqn:intro:theta_nonlinear}.

In our earlier work
\cite{jukic2019dynamics} we restricted attention to the horizontal direction $(\sigma_h, \sigma_v) = (1,0)$, which greatly simplified the analysis of \eqref{eqn:intro:super_sol_ansatz} and \eqref{eqn:intro:theta_nonlinear}. Indeed, we were able to choose $N=1$, with $a_1 = a_{-1} = 1$ and $p_{-1} = p_1 = 0$, which means that the linear terms reduce
to the standard discrete heat equation. Solutions could hence be represented explicitly in terms of modified Bessel functions of the first kind, for which detailed bounds are available in the literature.
In addition, the remaining auxiliary functions satisfied the useful identities
\begin{equation}
    q_{-1,+1} = q_{+1,-1} = 0,
    \qquad \qquad
    q_{-1,-1} = q_{+1,+1},
\end{equation}
allowing the quadratic terms in the super-solution residual to be analyzed in a transparant fashion.

For general rational directions, some of the coefficients $a_k$ can become negative, in which case
\eqref{eqn:intro:theta_nonlinear} no longer admits a comparison principle.
In addition, we can no longer represent our solutions
in terms of special functions for which powerful off-the-shelf estimates are 
available.
We resolve these issues in {\S}\ref{sec:lde}-\ref{sec:theta} by developing 
an approximate comparison principle and using
the saddle-point method to extract the necessary decay rates on the Green's function
for the linear part of 
\eqref{eqn:intro:theta_nonlinear}.

\paragraph{Mean curvature flows}
Matano and Nara proved in~\cite{Matano} that the solution $\theta(t)$ to the PDE~\eqref{eqn:intro:theta} can be approximated by solutions $\Gamma$ to the PDE
\begin{equation}
\label{eq:mn:cv:flow:gm:t}
 \dfrac{\Gamma_t}{\sqrt{1+\Gamma_y^2}} 
 = \dfrac{\Gamma_{yy}}
 {(1+\Gamma_y^2)^{3/2}}
 + c.
\end{equation}
This equation is known as a mean curvature flow equation with an additional drift term $c$. Indeed, writing
$\nu(y,t)$ for the rightward-pointing normal vector of the interfacial graph $\{\Gamma(y,t), y) \}$,
together with $V(y,t)$ for the horizontal velocity vector
and $H(y,t)$ for the curvature,
we can make the identifications
\begin{equation}
    \nu = \big[ 1 + \Gamma_y^2 \big]^{-1/2} (1 , - \Gamma_y),
    \qquad \qquad
    V = (\Gamma_t, 0 \big),
    \qquad \qquad
    H = \big[ 1 + \Gamma_y^2 \big]^{-3/2} \Gamma_{y y}.
\end{equation}
In particular, 
\eqref{eq:mn:cv:flow:gm:t}
can be written in the form
\begin{equation}
  \label{eq:int:simpl:mean:curv:PDE}
    V\cdot \nu = H + c,
\end{equation}
which reflects the rotational invariance of the wavespeed $c$.

In the discrete setting there is no `canonical' notion of a mean curvature flow due to the absence of a suitable normal vector for the interface
$(\Gamma_l, l)$.
Indeed, for a fixed index $l \in \Z$ one can
consider the angle

\begin{equation}
    \varphi_{l;k}(\Gamma)= \arctan{\dfrac{\Gamma_{l} - \Gamma_{l+k}}{k}},
\end{equation}
for any $k \in \Z$,
which measures the orientation of the vector that is transverse to the connection
between $(\Gamma_l,l)$
and $(\Gamma_{l+k}, l+k)$;
see Figure~\ref{fig:intro:zoom}. These can all be considered as normal directions in some sense.

However, it is possible and natural to apply appropriate discretization schemes to \eqref{eq:int:simpl:mean:curv:PDE}. In order to take the lattice anisotropy into account, we start by
writing
$c_{\varphi}$ for the
wavespeed associated to the planar wave solutions
\begin{equation}
    u_{ij}(t) = \Phi_{\varphi}\big(n \cos \varphi + l \sin \varphi - c_{\varphi} t \big)
\end{equation}
to \eqref{eqn:intro:AC}
that travel at an \textit{additional} angle of $\varphi$ relative to our original planar wave 
\eqref{eq:int:plane:wave:lde}. This allows us to define the
directional dispersion
    $$\mathcal{D}(\varphi) = \dfrac{c_{\varphi}}{\cos \varphi},$$
which measures the speed at which level sets of 
the wave $(c_{\varphi}, \Phi_{\varphi})$ move along the $n$-direction.

Setting out to discretize the terms in \eqref{eq:int:simpl:mean:curv:PDE}, we first introduce
the average
\begin{equation}
\label{eq:int:c:gamma}
    [\overline{c}_{\Gamma}]_l = 
    \dfrac{1}{2N} \sum_{0<|k|\leq N} c_{\varphi_{l;k}(\Gamma)},
\end{equation}
where we use $2N$ neighbours in order to account for all the interactions present in 
\eqref{eqn:intro:theta_nonlinear}.
In addition, 
we introduce the notation
\begin{equation}
\label{eq:int:lap:beta:gamma}
    [\beta_{\Gamma}]_l = \sqrt{1+\sum_{0 < |k| \le N} \dfrac{A_k}{k^2}{(\Gamma_{l+k} - \Gamma_l)^2} },
    \qquad \qquad 
    [\Delta_\Gamma]_l
    =\sum_{0 < |k| \le N} \frac{2 B_k}{ k^2}  (\theta_{l+k} - \theta_l),
\end{equation}
which depends on two
sequences
$(A_k)$ and $(B_k)$. These must satisfy the normalization conditions
\begin{equation}
    \label{eq:int:nrm:A:C}
\sum_{0 < |k| \le N} A_k = 1,
\qquad \qquad
\sum_{0 < |k| \le N} B_k = 1,
\qquad \qquad
\sum_{0 < |k| \le N} B_k/k = 0
\qquad \qquad
\end{equation}
in order to ensure
that $\beta_{\Gamma}$ and $\Delta_{\Gamma}$
reduce formally
to the symbols $\sqrt{1 + \Gamma_y^2}$
and $\Gamma_{yy}$ in the continuum limit.

These sequences weigh the contributions of each of the normal directions $\varphi_{l;k}$
to the components of our discrete curvature flow, which we formulate as
\begin{equation}
\label{eq:int:disc:cv:flow}
    \beta_{\Gamma}^{-1} \dot{\Gamma}
    = \kappa_H \beta_{\Gamma}^{-3} \Delta_\Gamma +
    \overline{c}_{\Gamma} .
\end{equation}
It turns out that \eqref{eq:int:disc:cv:flow}
and \eqref{eqn:intro:theta_nonlinear} can be matched
up to cubic terms
if and only if the parameters
are chosen as
\begin{equation}
\label{eq:int:cond:kappa:d}
    \kappa_H = \dfrac{1}{2}\sum_{k=-N}^N k^2 a_k ,
    \qquad \qquad
    d = \frac{ [\partial_{\varphi}^2 \mathcal{D}(\varphi)]_{\varphi = 0}}{2\kappa_H}.
\end{equation}
The latter expression precisely matches the choice
that comes from the technical considerations that lead to  
\eqref{eqn:intro:theta_nonlinear} during the construction of our super-solution \eqref{eqn:intro:super_sol_ansatz}.
It also plays a key role in the related studies
\cite{HARAG2006ANISO,hupkes2019travelling}
that concern travelling corner solutions
in anisotropic media.

\paragraph{Outlook}
In this paper we have restricted our attention to rational directions, primarily
due to the fact that we lose the periodicity of the Fourier
transform for irrational directions. In fact, the relevant Fourier symbol becomes quasi-periodic, making it very cumbersome
to extract the necessary decay estimates. We are working on further reduction steps to bypass this issue, which could eventually allows us to consider general rounded interfaces. On the other hand, we do believe that the approach developed here is already strong enough to handle further questions such as the stability of the corner solutions constructed
in \cite{hupkes2019travelling}
or the propagation of wavefronts through structured networks.

\paragraph{Organization}
This paper is organized as follows. After stating our main results in {\S}\ref{sec:main},
we discuss the asymptotic formation of interfaces
in {\S}\ref{sec:omega}
and {\S}\ref{sec:gamma} by exploiting the properties of $\omega$-limit points. These sections simplify the ideas in \cite{jukic2019dynamics} and adapt them to the more general setting considered in this paper. We proceed in
{\S}\ref{sec:lde} by studying
the linearization of our phase
LDE \eqref{eqn:intro:theta_nonlinear}. In particular, we use techniques inspired by the saddle-point method to extract our required decay rates and establish a quasi-comparison principle. These are 
used in {\S}\ref{sec:theta} to incorporate the nonlinear terms in \eqref{eqn:intro:theta_nonlinear} and build the bridge
with the discrete curvature flow \eqref{eq:int:disc:cv:flow}. These ingredients allow
us to construct  sub- and super-solutions 
for \eqref{eqn:intro:AC}
in {\S}\ref{sec:sub:sup}, 
which are subsequently used in {\S}\ref{sec:asymp}
to establish our final stability results.

\paragraph{Acknowledgments}

Both authors acknowledge support from the Netherlands Organization for Scientific Research (NWO) (grant 639.032.612).

\begin{figure}
    \centering
    \includegraphics[width = \linewidth]{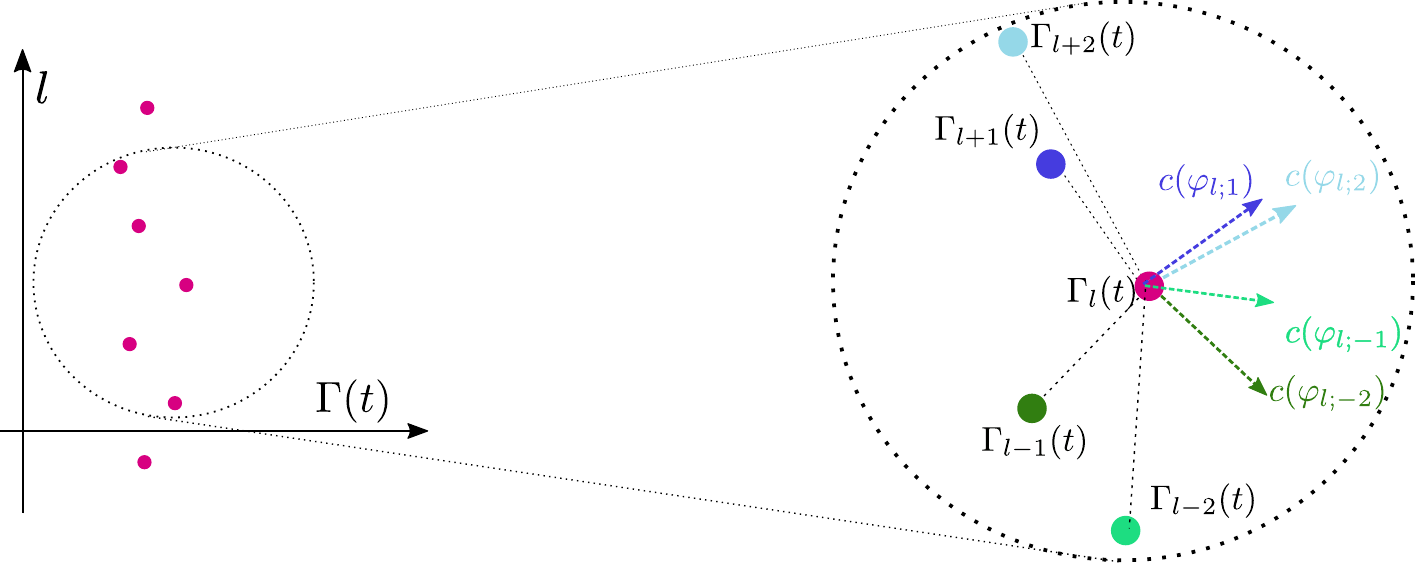}
    \caption{Here we provide the geometric motivation behind the
    definition~\eqref{eq:int:c:gamma} 
    for $\overline{c}_\Gamma$ with $N=2$.
    Since there is no uniquely defined normal direction for discrete graphs, we take the average of the velocities
     associated to the directions 
     transverse to the
     connecting lines between $(\Gamma_l,l)$  and $(\Gamma_{l+k} , l+k)$. Here we consider each $0 < |k| \le 2$.  }
    \label{fig:intro:zoom}
\end{figure}

\section{Main results}\label{sec:main}

In this paper we are interested in the discrete Allen-Cahn equation
\begin{equation}\label{eqn:main:AC_equation}
    \dot{u}_{i,j}(t) = [\Delta^+ u(t)]_{i,j} + g\big(u_{i,j}(t)\big)
\end{equation}
posed on the planar lattice $\Z^2$. The plus-shaped discrete Laplacian $\Delta^+:\ell^\infty(\Z^2)\to \ell^\infty(\Z^2)$ acts as a sum of differences over the nearest neighbors
\begin{equation}\label{eqn:main:laplace}
    \left[\Delta^+ u\right]_{i,j} = u_{i+1,j} + u_{i,j+1} + u_{i-1,j} + u_{i,j-1}  - 4u_{i,j},
\end{equation}
while the nonlinear function $g$
satisfies the following standard bistability condition.
\begin{itemize}
    \item[(Hg)]
    The nonlinearity $g:\R\to\R$ is  $C^3$-smooth  and there exists $a\in (0,1)$ such that
    \begin{equation*}
        g(0) = g(a) = g(1) = 0, \qquad \qquad g'(0) = g'(1) < 0.
    \end{equation*}
    In addition, we have the inequalities
    \begin{equation*}
       g(x)>0 \ \text{for}\ x\in (-\infty,0) \cup (a,1), \quad g(x)<0 \ \text{for}\ x\in (0, a) \cup (1, \infty).
    \end{equation*}
\end{itemize}
In this paper we focus on travelling waves propagating in rational directions. That is, we pick a direction $(\sigma_h, \sigma_v)\in \Z^2$ with $\mathrm{gcd}(\sigma_h, \sigma_v) = 1$ and consider wave-profiles $\Phi_*$
that connect the two stable equilibria of the nonlinear function $g$,
while traveling with the speed $c_*$ in the direction $(\sigma_h, \sigma_v)$.

It is convenient to pass to a new $(n,l)$-coordinate system
that is oriented parallel $(n)$ and transverse $(l)$ to the direction of wave-propagation.
In particular, we write
\begin{align*}
    n = i\sigma_h + j\sigma_v,
    \qquad \qquad 
    l = i\sigma_v - j\sigma_h
\end{align*}
and introduce the notation
\begin{align*}
    \Zs &= \left\{(n,l)\in \Z^2: \exists (i,j)\in \Z^2: n = i\sigma_h + j\sigma_v, \ l = i\sigma_v - j\sigma_h \right\}
    \subset \Z^2
\end{align*}
for the image of the original grid $\Z^2$.
Upon introducing the quantities
\begin{equation}
    \sigma_*= \sqrt{\sigma_h^2+\sigma_v^2},
    \qquad \qquad
    \sigma_\infty = \max \{ |\sigma_h|, |\sigma_v| \},
\end{equation}
we point out the mappings
\begin{equation}
    (i+\sigma_h, j+\sigma_v) \mapsto 
    (n+\sigma_*^2, l),
    \qquad \qquad
    (i+\sigma_v, j - \sigma_h)
    \mapsto
    (n, l+\sigma_*^2),
\end{equation}
which implies that for any $(n,l)\in \Zs$ the point $(n+a\sigma_*^2, l+b\sigma_*^2)$ is also an element of $\Zs$ for any $(a, b)\in \Z^2$, see Figure~\ref{fig:intro:coordinate}.


In this new coordinate system the discrete Laplace operator~\eqref{eqn:main:laplace} transforms as
\begin{equation}\label{eqn:main:laplace_new}
[\Delta^\times u ]_{n,l} = u_{n+\sigma_h, l+\sigma_v} + u_{n+\sigma_v, l-\sigma_h} + u_{n-\sigma_h, l-\sigma_v}  + u_{n-\sigma_v, l+\sigma_h} - 4u_{n,l}.
\end{equation}
In particular, the initial value problem that we consider in this paper can be written in the form
\begin{align}
    \dot{u}_{n,l}(t) &= [\Delta^\times u(t)]_{n,l} + g\big(u_{n,l}(t)\big), 
    \qquad (n,l)\in {\Zs}, \qquad t>0, \label{eqn:main:discrete_AC_new}\\
    u_{n,l}(0) &=u_{n,l}^0 , \label{eqn:main:initial_condition_new} 
\end{align}
for some initial condition $u^0 \in \ell^\infty(\Z^2_\times)$.
 Our second assumption
 imposes a `front-like' property
 on this initial condition $u^0$. 
\begin{itemize}
    \item[(H$0$)] The initial condition $u^0\in \ell^\infty(\Zs)$ satisfies
    \begin{equation}
        \limsup_{n\to-\infty} \sup_{l\in \Z:(n,l)\in \Zs} u^0_{n,l} <a, \qquad \qquad  \liminf_{n\to+\infty} \inf_{l\in \Z: (n,l)\in \Zs} u^0_{n,l} > a .
    \end{equation}
\end{itemize}

\subsection{Travelling waves}

A travelling wave solution is any solution of the form
\begin{equation}
\label{eq:mr:def:planar:wave}
    u_{n,l}(t) = \Phi_*(n-c_*t)
\end{equation}
for some wave-profile $\Phi_*$ and speed $c_*\in \R$. Any such pair must necessarily satisfy
the MFDE
\begin{equation}\label{eqn:main:mdfe}
    -c_* \Phi_*'(\xi) = \Phi_*(\xi+\sigma_h) + \Phi_*(\xi+\sigma_v) + \Phi_*(\xi-\sigma_h) + \Phi_*(\xi - \sigma_v) - 4 \Phi_*(\xi) + g\big(\Phi_*(\xi)\big),
\end{equation}
which we augment with the boundary
conditions
\begin{equation}\label{eqn:main_results_boundary_cond}
    \lim_{\xi\to-\infty} \Phi_*(\xi) = 0, \qquad \lim_{\xi\to\infty} \Phi_*(\xi) = 1. 
\end{equation}
The existence of such pairs $(c_*, \Phi_*)$
was established by Mallet-Paret in
\cite{Mallet-Paret1999}, both for rational and irrational directions. The
wave-speed $c_*$ is unique once the direction $(\sigma_h, \sigma_v)$ and the detuning parameter $a$ have been fixed, while the wave-profile $\Phi_*$ is monotonically increasing and unique up to translations provided that $c_* \neq 0$. In contrast to the continuous setting, there can be a range
of values for $a$ where $c_* = 0$ holds;
see \cite{hupkes2019travelling} for a detailed discussion.
The assumption below ensures that we are outside of this so-called pinning regime.

\begin{itemize}
    \item[(H$\Phi$)] There exists a wave-speed $c_*\neq0$ and a monotone wave profile $\Phi_*$ that satisfy the MFDE
    \eqref{eqn:main:mdfe}
    together with the boundary conditions
  \eqref{eqn:main_results_boundary_cond}
  and the phase normalization $\Phi_*(0) = \frac{1}{2}$.
\end{itemize}




To examine the stability properties of the wave-pair $(\Phi_*, c_*)$ under the dynamics of \eqref{eqn:main:discrete_AC_new},
one usually starts by considering  the linear variational problem
\begin{equation*}
    \dot{v}_{n,l}(t) = [\Delta^\times v(t)]_{n,l} + g'\big(\Phi_*(n-c_*t)\big) v_{n,l}(t).
\end{equation*}
Taking the discrete Fourier transform along
the transverse direction $l$, 
the problem decouples
into the set of one-dimensional LDEs 
\begin{equation}\label{eqn:main:ft}
\begin{aligned}
    \dot{v}_n(t)&= e^{i\omega \sigma_v} v_{n+\sigma_h}(t) + e^{-i\omega \sigma_h} v_{n+\sigma_v}(t) + e^{-i\omega \sigma_v} v_{n-\sigma_h}(t) + e^{i\omega \sigma_h}  v_{n-\sigma_v}(t) - 4v_n(t) \\
    & \qquad   + g'\big(\Phi_*(n-c_*t )\big)v_n(t),
\end{aligned}
\end{equation}
indexed by the frequency variable $\omega\in [-\pi, \pi]$.
As shown in~\cite[\S 2]{HS13},  there is a close relationship
between the Green’s function for each of the LDEs~\eqref{eqn:main:ft} and 
their associated linear operators
\begin{equation*}
    \mathcal{L}_\omega: W^{1, \infty}(\R;\C) \to L^\infty(\R;\C) , \qquad \omega\in [-\pi, \pi]
\end{equation*}
which act as
\begin{equation}
\label{eq:mr:def:l:omega}
\begin{aligned}
        \left[\mathcal{L}_\omega p\right](\xi) &= c_*p'(\xi) + e^{i\omega \sigma_v} p(\xi+\sigma_h) + e^{-i\omega \sigma_h} p(\xi+\sigma_v) + e^{-i\omega \sigma_v} p(\xi-\sigma_h) + e^{i\omega \sigma_h}  p(\xi-\sigma_v) \\
        &\qquad - 4p(\xi) + g'\big(\Phi_*(\xi)\big) p(\xi) .
        \end{aligned}
\end{equation}

A special role is reserved
for the operator $\mathcal{L}_0$,
which encodes the linearized behaviour
of the wave $\Phi_*$ under perturbations that are homogeneous in the transverse direction.
We briefly summarize several key Fredholm properties of this operator that were obtained
by Mallet-Paret in the classic paper \cite{Mallet-Paret1999_Fredholm}.
\begin{lemma}[see \cite{Mallet-Paret1999_Fredholm}]
\label{lem:mr:fred:props:l:0}
Assume that (Hg) and $(H\Phi)$ are satisfied.
Then the operator $\mathcal{L}_0:W^{1, \infty}(\R;\C) \to L^\infty(\R;\C)$ is  Fredholm with index zero. It has a 
one-dimensional kernel spanned by
the strictly positive function $\Phi_*'$.
In addition, its range admits the characterization
\begin{equation}\label{eqn:main:range_L0}
        \mathcal{R}(\mathcal{L}_0) = \left\{ f\in L^{\infty}(\R;\R): \int_\R \psi_*(\xi) f(\xi) \, d\xi = 0 \right\}
    \end{equation}
for some strictly positive bounded function\footnote{In fact, $\psi_*$ spans
the kernel of the formal adjoint $\mathcal{L}_0^*$ that arises from $\mathcal{L}_0$ by flipping the sign of $c$.}
$\psi_* \in C^2(\R;\R)$ that we normalize to have
\begin{equation}
    \int_\R \psi_*(\xi) \Phi_*'(\xi) d\xi = 1.
\end{equation}
\end{lemma}

Since clearly $\Phi_*' \notin \mathcal{R}(\mathcal{L}_0)$ we see that $\lambda = 0$ is a simple eigenvalue of the operator $\mathcal{L}_0$.
The following result
states that this property extends to a branch of simple  eigenvalues $\lambda_\omega$ for the operators $\mathcal{L}_\omega$ with $\omega\approx 0$. 

\begin{lemma}[{see \cite[Prop. 2.2]{hoffman2015multi}}]
\label{lemma:main_results:lambda_omega}
Assume that (H$g$) and (H$\Phi$) are satisfied. Then there exists a constant $0<\omega_0\ll 1$ together with  pairs
\begin{equation*}
    (\lambda_\omega, \phi_\omega) \in \C\times   W^{1, \infty}(\R;\C),
\end{equation*}
defined for each $\omega \in (-\omega_0, \omega_0)$, that satisfy the following properties.
\begin{enumerate} [(i)]
    \item For each $\omega\in (-\omega_0, \omega_0)$ we have the characterization
\begin{equation*}
    \mathrm{Ker}(\mathcal{L}_{\omega} - \lambda_\omega) =  \mathrm{span}\left\{\phi_\omega \right\} ,
\end{equation*}
together with the algebraic simplicity condition
\begin{equation*}
\phi_\omega \notin \mathcal{R}(\mathcal{L}_{\omega} - \lambda_\omega)   .
\end{equation*}
\item We have $\lambda_0 = 0$, $\phi_0 = \Phi_*'$ and the maps $\omega\mapsto \lambda_\omega$, $\omega\mapsto \phi_\omega$ are analytic.
\item\label{item:main:phi_psi_integral}  
For each $\omega\in (-\omega_0, \omega_0)$ we have the normalization
\begin{equation*}
    \langle \phi_\omega, \psi_* \rangle_{L^2} = 1.
\end{equation*}
\end{enumerate}
\end{lemma}

Our following assumption states
that the map 
$\omega\mapsto \lambda_\omega$ touches
the origin in a quadratic tangency, opening
up to the left of the imaginary axis.
This is a rather standard condition
that was also used 
in~\cite{hoffman2015multi}
and~\cite{hoffman2017entire}
to show that transverse phase deformations
decay at the standard rates prescribed by the heat equation.
We remark that Lemma 6.3 in \cite{hoffman2015multi} 
guarantees that this condition
is satisfied whenever the
propagation direction is close to horizontal
or diagonal. Furthermore, numerical experiments in~\cite[\S 6]{hoffman2015multi} suggest that this extends to all
directions where the wavespeed does not vanish.

\begin{itemize}
    \item[(HS$)_1$] The branch of eigenvalues $(\lambda_\omega)_{\omega\approx 0}$ 
    satisfies the inequality
    \begin{equation*}
       [\partial_\omega^2  \lambda_\omega]_{\omega = 0} < 0.
    \end{equation*}
\end{itemize}

Our final spectral assumption is far less standard and requires some technical preparations. To this end,
we introduce the set
of shifts
\begin{equation}
\label{eq:mr:def:shifts:tau}
    (\tau_1, \tau_2, \tau_3, \tau_4) = (\sigma_h, \sigma_v, -\sigma_h, -\sigma_v)
\end{equation}
and their associated translation operators
$T_{\nu}$ that act as 
\begin{equation}
\label{eq:mr:def:T:nu}
    [T_\nu h](\xi) = h(\xi + \tau_\nu), \qquad \nu \in \left\{1, 2, 3, 4\right\}
\end{equation}
for any function $h\in C(\R)$.
These can be used to define
a collection of functions
$p^\diamond$, $p^{\diamond \diamond}$ and $q^{\diamond \diamond}$ that play
a key role in  {\S}\ref{sec:sub:sup}
where we construct  sub- and super-solutions for \eqref{eqn:main:discrete_AC_new}. For our purposes here, we are chiefly interested in the associated coefficients
$\alpha^\diamond_p$, $\alpha^{\diamond \diamond}_p$ and $\alpha^{\diamond \diamond}_q$ that are related
to the solvability condition
\eqref{eqn:main:range_L0}.

\begin{lemma}[{see \S\ref{sec:theta}}]
\label{lemma:main_results:coeff_diamond}
Assume that $(Hg)$ and $(H\Phi)$ both  hold. Then for every $\nu, \nu'\in \left\{1, 2, 3, 4\right\}$ there exist bounded functions 
\begin{equation*}
    p_\nu^\diamond, p_{\nu \nu'}^{\diamond \diamond}, q_{\nu \nu'}^{\diamond \diamond} : \R \to \R
\end{equation*}
that satisfy the identities
\begin{equation}\label{eqn:main:eqns_for_p}
  \begin{array}{lcl}
    \left[\mathcal{L}_0 p^\diamond_\nu\right](\xi) &=& [T_{\nu}\Phi'](\xi)  - \alpha^\diamond_{p;\nu}\Phi'(\xi),
    \\[0.3cm]
    [\mathcal{L}_0p^{\diamond \diamond}_{\nu \nu'}](\xi) &=&    \alpha_{p;\nu'}^\diamond p^\diamond_\nu(\xi) - [T_{\nu'} p_\nu^\diamond](\xi) - \alpha^{\diamond\diamond}_{p;\nu\nu'} \Phi'(\xi),
    \\[0.3cm]
     [\mathcal{L}_0q^{\diamond \diamond}_{\nu \nu'}](\xi) &=& - \alpha_{p;\nu}^\diamond \dfrac{d}{d\xi}p_\nu^\diamond(\xi)  + [T_{\nu'} \dfrac{d}{d\xi}p_\nu^\diamond](\xi)  -\dfrac{1}{2} g''\big(\Phi_*(\xi)\big)  p_\nu^\diamond(\xi) p_{\nu'}^\diamond(\xi) \\
     & & \qquad  - \dfrac{1}{2}\mathbf{1}_{\nu = \nu'}[T_\nu \Phi_*''](\xi) - \alpha^{\diamond\diamond}_{q;\nu\nu'} \Phi'(\xi).\\
\end{array}
\end{equation}
Here the  coefficients $\alpha_{p;\nu}^\diamond$, $\alpha_{p;\nu \nu'}^{\diamond \diamond}$
and
$\alpha_{q;\nu\nu'}^{\diamond \diamond }$ are given by
\begin{equation}\label{eqn:main:eqn_for_alpha_diamonds}
  \begin{array}{lcl}
    \alpha_{p;\nu}^\diamond & = &
   \int_\R [T_\nu \Phi'](\xi) \psi_*(\xi) d\xi ,
     \\[0.3cm]
   \alpha^{\diamond \diamond}_{p;\nu \nu'} &=& \int_\R \left[ \alpha_{p;\nu}^\diamond p^\diamond_{\nu'}(\xi) - [T_{\nu'} p_\nu^\diamond] (\xi)\right]\psi_*(\xi) d\xi ,
   \\[0.3cm]
   \alpha^{\diamond \diamond}_{q;\nu \nu'} 
    &=&
    \int_{\R}\left( - \alpha_{p;\nu}^\diamond \dfrac{d}{d\xi}p_{\nu'}^\diamond(\xi)  + [T_{\nu'} \dfrac{d}{d\xi}p_\nu^\diamond](\xi)  -\dfrac{1}{2} g''\big(\Phi_*(\xi)\big)  p_\nu^\diamond(\xi)p_{\nu'}^\diamond(\xi) \right)\psi_*(\xi) d\xi \\
    & & 
    \qquad 
    -  \dfrac{1}{2} \int_{\R} \mathbf{1}_{\nu = \nu'}[T_\nu \Phi_*''](\xi)\psi_*(\xi) d\xi. \\
\end{array}
\end{equation}
Moreover, the functions $p^{\diamond}_{\nu}, p^{\diamond \diamond}_{\nu \nu'}$ and $q^{\diamond \diamond}_{\nu \nu'}$ can be chosen in such a way that 
\begin{equation}\label{eqn:main:ip_function_p_psi}
    \langle p^{\diamond}_{\nu}, \psi_* \rangle_{L^2} = 0, \qquad \langle p^{\diamond\diamond}_{\nu\nu'}, \psi_* \rangle_{L^2} =0, \qquad \langle q^{\diamond\diamond}_{\nu\nu'}, \psi_* \rangle_{L^2} =0.
\end{equation}
\end{lemma}   

Upon introducing the convenient
notation
\begin{equation}
\label{eq:mr:def:shifts:sigma}
    (\sigma_1, \sigma_2, \sigma_3, \sigma_4) = (\sigma_v, -\sigma_h, -\sigma_v, \sigma_h),
\end{equation}
we now use the coefficients
\eqref{eqn:main:eqn_for_alpha_diamonds} to introduce the function
$f_{(\sigma_h, \sigma_v)}:[-\pi, \pi]\to \R$
that acts as
\begin{equation}
\begin{array}{lcl}
\label{eq:mr:def:f:sh:sv}
    f_{(\sigma_h, \sigma_v)}(\omega) 
    &=& \sum_{\nu=1}^4 \alpha_{p;\nu}^\diamond (\cos{(\sigma_\nu \omega)} - 1) 
    \\[0.2cm]
    & & \qquad + \sum_{\nu, \nu'=1}^4 \alpha_{p;\nu \nu'}^{\diamond\diamond} \Big(\cos{\big((\sigma_\nu  + \sigma_{\nu'})\omega\big)} - \cos{(\sigma_\nu \omega)} - \cos{(\sigma_{\nu'} \omega)}+ 1\Big)  .
\end{array}
 \end{equation}
For the sequel, it is convenient
to rewrite this expression in a more compact form. To this end, 
we write $N = \max_{\nu, \nu'\in \{1, 2, 3, 4\}} \left\{ \sigma_\nu, \sigma_\nu + \sigma_{\nu'}\right\}$
and introduce the sequence
\begin{equation}\label{eqn:main:a_k}
    a_k = \sum_{\nu=1}^4 \alpha_{p;\nu}^\diamond   \textbf{1}_{\{k = \sigma_\nu\}} + \sum_{\nu, \nu'=1}^4 \alpha_{p;\nu \nu'}^{\diamond \diamond }\big(\textbf{1}_{\{k = \sigma_\nu +\sigma_{\nu'}\}}- \textbf{1}_{\{k = \sigma_\nu\}} - \textbf{1}_{\{k = \sigma_{\nu'}\}}\big),
\end{equation}
which allows us to rewrite \eqref{eq:mr:def:f:sh:sv} as
\begin{equation}\label{eqn:main:function_f}
   f_{(\sigma_h, \sigma_v)}(\omega) = \sum_{k=-N}^N a_k (\cos{(k \omega)} - 1).  
\end{equation}
This function will appear later as the real part of the Fourier symbol associated to the linear dynamics of the transverse phase of the planar wave $(c_*, \Phi_*)$.

In the horizontal case $(\sigma_h, \sigma_v) = (1,0)$ we can take $N=1$, $a_{-1} = 1$, $a_1 = 1$ and 
\begin{equation}
    f_{(1,0)}(\omega) = 2(\cos \omega -1),
\end{equation}
but in general the coefficients $a_k$ can be negative. In order to ensure that our phase dynamics can be controlled, our final assumption 
requires the function $f$ to be strictly negative for all non-zero $\omega$.
\begin{itemize}
    \item [(HS$)_2$] 
    The inequality
               $f_{(\sigma_h, \sigma_v)}(\omega) < 0$
               holds for all $\omega\in [-\pi, \pi]\backslash\{0\}$.   
\end{itemize}

We now set out to obtain
some geometric intuition concerning
the coefficients \eqref{eqn:main:eqns_for_p}
and the Fourier symbol \eqref{eqn:main:function_f}. 
We first note that the pair $(  c_*,\Phi_*)$ can be perturbed in order to yield waves travelling in  directions that are `close' to $(\sigma_h, \sigma_v)$.
In particular, we follow the approach from~\cite{hupkes2019travelling}
and look for solutions to the Allen-Cahn equation~\eqref{eqn:main:discrete_AC_new} of the form
\begin{equation}\label{eqn:main:perturbed:ansatz}
    u_{n,l}(t) = \Phi_\varphi (n\cos{\varphi} + l\sin{\varphi} - c_\varphi t),
\end{equation}
which travel at an angle $\varphi$ through the rotated lattice $\Z^2_{\times}$.
Inserting this Ansatz into~\eqref{eqn:main:discrete_AC_new},
we find that the pair $(c_\varphi, \Phi_\varphi)$ must satisfy the MFDE
\begin{equation}\label{eqn:main:perturbed:mfde}
    \begin{aligned}
        -c_\varphi \Phi_\varphi'(\xi) &= \Phi_\varphi(\xi + \sigma_h \cos{\varphi} + \sigma_v \sin{\varphi}) + \Phi_\varphi(\xi + \sigma_v \cos{\varphi} - \sigma_h \sin{\varphi}) \\
        &\qquad  +\Phi_\varphi(\xi - \sigma_h \cos{\varphi} - \sigma_v \sin{\varphi})  +\Phi_\varphi(\xi - \sigma_v \cos{\varphi} + \sigma_h \sin{\varphi}) 
        \\
        & \qquad -4 \Phi_\varphi(\xi) + g\big(\Phi_\varphi(\xi) \big).
    \end{aligned}
\end{equation}
Using standard bifurcation arguments
one can show that the pair $(\Phi_*, c_*)$
can be embedded into a smooth branch
of waves $(c_\varphi, \Phi_\varphi)$
for $\varphi\approx 0$.
\begin{lemma}[{see \cite[Prop. 2.2]{hupkes2019travelling} and \cite[Thm. 2.7]{hoffman2015multi}}]
\label{lem:mr:dir:dep:waves}
Assume that (H$g$) and (H$\Phi$) are satisfied. Then there exists a constant $\delta_\varphi>0$ together with pairs 
\begin{equation*}
    (c_\varphi, \Phi_\varphi)\in\R\times W^{1, \infty} (\R;\R),
\end{equation*}
defined for every $\varphi \in (-\delta_\varphi, \delta_\varphi)$, such that the following holds true.
\begin{enumerate} [(i)]
    \item For every $\varphi \in (-\delta_\varphi, \delta_\varphi)$ the pair $(c_\varphi, \Phi_\varphi)$ satisfies the MFDE~\eqref{eqn:main:perturbed:mfde}
    together with the boundary conditions
    \eqref{eqn:main_results_boundary_cond}.
    \item For every $\varphi \in (-\delta_\varphi, \delta_\varphi)$
    we have the normalization 
    $\langle \Phi_{\varphi} - \Phi_* , \psi_* \rangle = 0$.
    \item The maps $\varphi\mapsto c_\varphi$ and $\varphi \mapsto \Phi_\varphi$ are $C^{2}$-smooth, 
    with $(c_0, \Phi_0) = (c_*, \Phi_*)$. 
\end{enumerate}
\end{lemma}

Our next result shows that there is a close link between the coefficients 
\eqref{eqn:main:eqns_for_p},
the pairs $(\lambda_{\omega}, \phi_{\omega})$
constructed in 
Lemma~\ref{lemma:main_results:coeff_diamond}
and the waves
$(c_{\varphi}, \Phi_{\varphi})$
described in Lemma \ref{lem:mr:dir:dep:waves}.
These identities can be stated in a compact fashion by virtue
of the choices 
\eqref{eq:mr:def:shifts:tau} and
\eqref{eq:mr:def:shifts:sigma}.

\begin{lemma}[{see \S\ref{sec:theta}}]
\label{lemma:main:expressions_c_lambda}
Assume that (H$g$) and (H$\Phi$) are satisfied.
Then the following identities hold. 
\begin{enumerate}[(i)]
    \item\label{item:analysis_theta:c} $c_* = -\sum_{\nu=1}^4\tau_\nu \alpha_{p;\nu}^\diamond$,
    \item\label{item:analysis_theta:partial_c} $[ \partial_\varphi c_\varphi]_{\varphi = 0} =  -\sum_{\nu=1}^4 \sigma_\nu \alpha_{p;\nu}^\diamond  = -\sum_{k=-N}^N a_k k  $,
    \item\label{item:analysis_theta:c_2nd_derivative} $[ \partial^2_\varphi c_\varphi]_{\varphi = 0} =   - c_* + 2\sum_{\nu=1}^4 \sum_{\nu'=1}^4\sigma_\nu \sigma_{\nu'} \alpha_{q;\nu\nu'}^{\diamond \diamond}$,
   
    \item\label{item:analysis_theta:partial_phi} $[\partial_\varphi \Phi_\varphi]_{\varphi = 0} = - \sum_{\nu=1}^4 \sigma_\nu p_{\nu}^\diamond$,
    \item\label{item:analyis_theta:partial_lambda} $[\partial_{\omega} \lambda_\omega]_{\omega = 0} = \sum_{\nu=1}^4 \sigma_\nu \alpha_{p;\nu}^\diamond $,
    \item\label{item:analyis_theta:partial_lambdax2} $[\partial_{\omega}^2 \lambda_\omega]_{\omega = 0} = - \sum_{\nu=1}^4 \alpha_\nu^\diamond \sigma_\nu^2  - \sum_{\nu, \nu'=1}^4  2\alpha_{p;{\nu\nu'}}^{\diamond\diamond}\sigma_\nu\sigma_{\nu'} = - \sum_{k=-N}^N a_k k^2$.
\end{enumerate}

\end{lemma}

Combining item~\textit{(\ref{item:analyis_theta:partial_lambdax2})}
and \eqref{eq:mr:def:f:sh:sv}, we readily see that
\begin{equation*}
    f''_{(\sigma_h, \sigma_v)}(0) = [\partial_\omega^2 \lambda_\omega]_{\omega=0}.
\end{equation*}
This identity in combination with (HS$)_1$ implies that the function $f_{(\sigma_h, \sigma_v)}$ looks like a downwards parabola locally around $\omega=0$. This information
was sufficient to obtain the `localized' stability results in~\cite{hoffman2015multi}
and~\cite{hoffman2017entire},
but our more general setup here requires 
global information on the function $f_{(\sigma_h,\sigma_v)}$.
An important role is reserved for the parameter
\begin{equation}\label{eqn:main:d}
d =  
- \dfrac{ \partial^2_{\varphi} [ c_{\varphi} / \cos \varphi]_{\varphi = 0}} {[\partial_\omega^2  \lambda_\omega]_{\omega = 0}}
= 
-\dfrac{c_* + [\partial_\varphi^2  c_\varphi]_{\varphi = 0} }{ [\partial_\omega^2  \lambda_\omega]_{\omega = 0}}
= \dfrac{2\sum_{\nu=1}^4 \sum_{\nu'=1}^4\sigma_\nu \sigma_{\nu'} \alpha_{q;\nu\nu'}^{\diamond \diamond}}
{\sum_{\nu=1}^4 \alpha_\nu^\diamond \sigma_\nu^2  + \sum_{\nu, \nu'=1}^4  2\alpha_{p;{\nu\nu'}}^{\diamond\diamond}\sigma_\nu\sigma_{\nu'}}
,
\end{equation}
which is well-defined on account of assumption (HS$)_1$. 
It measures the ratio between the quadratic terms in the directional dispersion $c_{\varphi}/\cos\varphi$ and the
branch of eigenvalues $\lambda_{\omega}$.
This parameter also played a crucial role
throughout the construction of travelling corners for \eqref{eqn:main:AC_equation};
see \cite[Eqs. (7.38) and (7.76)]{hupkes2019travelling} where it appears as the quadratic coefficient on the center manifold that governs the transverse dynamics.

\subsection{Interface formation}\label{sec:interface formation}

In this subsection we provide a construction
for the set of phases $\big(\gamma_l(t)\big)_{l\in \Z}$ that should be seen as an \textit{approximation}
for the level set $u = \frac{1}{2}$. 
Indeed, due to the 
discreteness of the lattice one cannot 
necessarily find integers $n_*(l,t)$
for which $u_{n_*(l,t), l}(t) = \frac{1}{2}$ holds exactly - even when restricted to large times $t \gg 1$. Instead,
we establish the following monotonicity result, which for fixed $l$ and large $t \gg 1$ allows us to capture the
`crossing' of $u$ through $\frac{1}{2}$
between $n = n_*(l,t)$ and
$n = n_*(l,t) + \sigma_*^2$.

\begin{proposition}[{see $\S$\ref{sec:gamma}}] \label{prop:main:interface} Suppose that (H$g$), (H$\Phi$) and (H$0$) are satisfied.
There exists a time $T>0$ such that for every $l\in \Z$ and $t\geq T$ there exists a unique $n_* = n_*(l,t)$ with the property
\begin{equation}
\label{eq:mr:props:n:star}
    0<u_{n_*,l}(t)\leq\dfrac{1}{2}, 
    \qquad \qquad
    u_{n_*+\sigma_*^2,l}(t) > \dfrac{1}{2}.
\end{equation}
\end{proposition}

We now set out use an interpolation argument
to construct $\gamma_l(t)$ from
the quantities in \eqref{eq:mr:props:n:star}.
The main consideration is 
that for exact travelling waves 
$u_{n,l}(t) = \Phi_*(n - c_* t+\mu)$ we wish
to recover the standard phase $\gamma_l(t) = c_* t-\mu$, in view of the fact that $\Phi_*(0) = \frac{1}{2}$. To achieve this, we define the phases
\begin{equation}
    \theta_l^-(t)
    = \Phi_*^{-1}\big( u_{n_*(l,t), l}(t) \big),
    \qquad \qquad
    \theta_l^+(t)
    = \Phi_*^{-1}\big( u_{n_*(l,t) + \sigma_*^2, l}(t) \big)
\end{equation}
associated to the two values
\eqref{eq:mr:props:n:star}. 
Upon writing
\begin{equation}\label{eqn:main:zeta}
    \vartheta_*(l, t) =  -\sigma_*^2 \theta_l^-(t)/[ \theta_l^+(t) - \theta_l^-(t)],
\end{equation}
we note that the linear interpolation
\begin{equation}
    \theta_{\mathrm{lin};l,t}(\xi) = \sigma_*^{-2} \theta_l^+(t) \xi   -\sigma_*^{-2}   \theta_l^-(t) (\xi  - \sigma_*^2)
\end{equation}
satisfies
\begin{equation}
    \theta_{\mathrm{lin};l,t}(0) = \theta_l^-(t),
    \qquad \qquad
    \theta_{\mathrm{lin};l,t}\big(  \vartheta_*(l,t) \big)
     = 0,
    \qquad \qquad
    \theta_{\mathrm{lin};l,t}( \sigma_*^2) = \theta_l^+(t).
\end{equation}
This motivates the phase-interpolated definition 
\begin{equation}\label{eqn:main:def_of_gamma}
    \gamma_l(t) =  n_*(l,t) +  \vartheta_*(l,t),
\end{equation}
which ensures that the `stretched' profile
$\tilde{\Phi}(\xi) = \Phi_*\Big(\theta_{\mathrm{lin};l,t}\big(\xi - n_*(l,t) \big)\Big)$ satisfies
\begin{equation}
    \tilde{\Phi}\big(n_*(l,t)\big) = u_{n_*(l,t),l}(t),
    \qquad \qquad
    \tilde{\Phi}\big(\gamma_{l}(t)\big) =
    \frac{1}{2},
    \qquad \qquad 
    \tilde{\Phi}\big(n_*(l,t) + \sigma_*^2 \big) = u_{n_*(l,t) + \sigma_*^2,l}(t).
\end{equation}

\begin{figure}
    \centering
    \includegraphics[width=\linewidth]{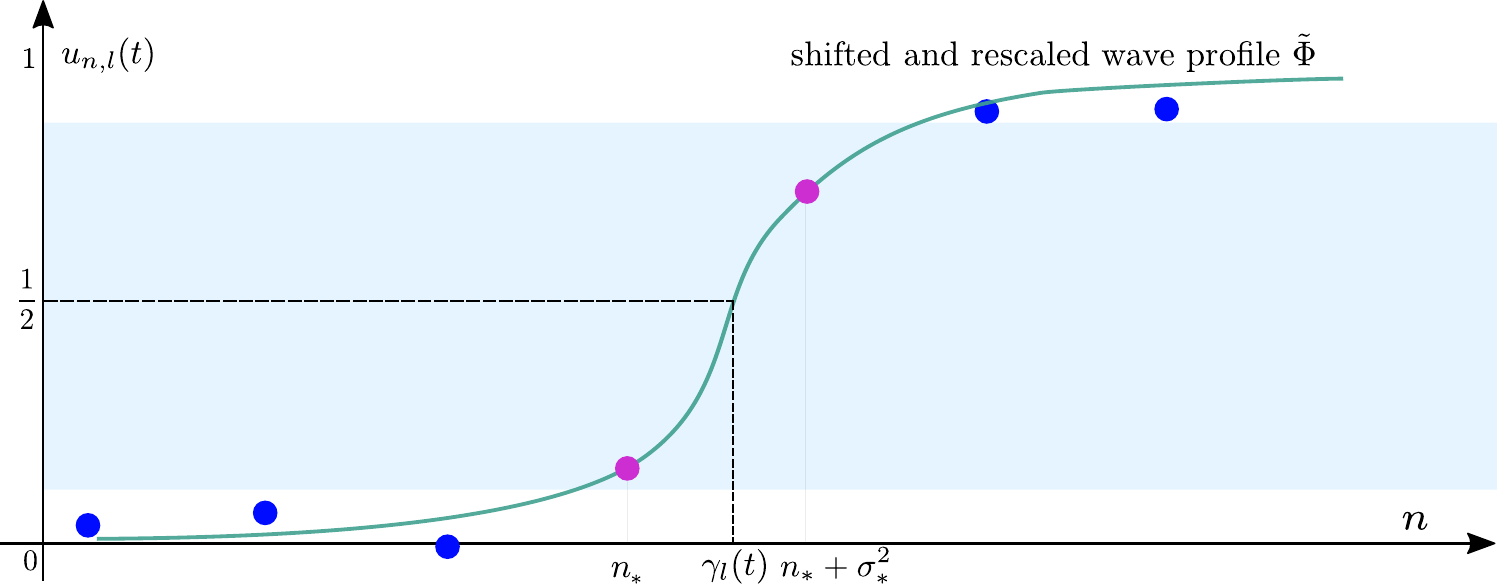}
    \caption{In order to construct the phase $\gamma_l(t)$ for a fixed pair $(l,t)$ we first identify an interfacial region around the value $\Phi_*(0) = \frac{1}{2}$ (shaded in blue) where the (discrete) function $ n \mapsto u_{n,l}(t) $ is monotone. We subsequently stretch the waveprofile to match the (pink) points $\left(n_*, u_{n_*, l}(t)\right)$ and $\left(n_*+\sigma_*^2, u_{n_*+\sigma_*^2, l}(t)\right)$ introduced in \eqref{eq:mr:props:n:star}.  }
    \label{fig:intro:gamma}
\end{figure}

Notice indeed that for the special case $u_{n,l}(t) = \Phi_*(n - c_* t + \mu)$ we have 
\begin{equation}
    \theta_l^-(t) = n_*(l,t) - c_*t + \mu,
    \qquad \qquad
    \theta_l^+(t) = n_*(l,t) + \sigma_*^2 - c_*t + \mu,
\end{equation}
which gives $
    \vartheta_*(l,t) = - \theta_l^-(t)$
and hence $\gamma_l(t) = c_*t - \mu$, as we desired. 

The result below states that our phase indeed tracks the behaviour of $u$ in an asymptotic sense.
We emphasize that there are several other choices for the phase that lead to 
similar results.
For example, 
our previous construction in~\cite{jukic2019dynamics} did not stretch the wave and merely aligned
it with $u$ at the point $n_*(l,t)$.
Our more refined approach here
allows us to streamline our arguments
and avoid the discontinuities in $\gamma_l(t)$ that complicated our previous analysis at times. 

\begin{proposition}
[{see $\S$\ref{sec:gamma}}]
\label{thm:main_results:gamma_approx_u} 
Suppose that (H$g$), (H$\Phi$) and (H$0$) are satisfied. Then we have the limit
\begin{equation}\label{eqn:main:gamma_approx_u}
    \lim_{t\to\infty} \sup_{(n,l)\in \Zs} |u_{n,l}(t) - \Phi_*\big(n-\gamma_l(t)\big)| = 0.
\end{equation}
\end{proposition}

\subsection{Interface asymptotics}
We are now ready to discuss
the main technical results of this paper.
These concern the asymptotic behaviour
of the phase $\gamma(t)$
that we introduced in~\eqref{eqn:main:def_of_gamma},
which can be approximated by
solutions
to the scalar nonlinear LDE
\begin{equation}
\label{eqn:main:main_eqn_for_theta}
\dot{\theta}(t) = \Theta_\mathrm{ch}\big(\theta(t)\big).
\end{equation}
Here the function
$\Theta_\mathrm{ch}:\ell^\infty(\Z)\to\ell^\infty(\Z)$ acts as
\begin{equation}\label{eqn:main:Theta}
    [\Theta_\mathrm{ch}(\theta)]_l = \begin{cases}
    \dfrac{1}{d} \sum_{k=-N}^N a_k \left(e^{d\big(\theta_{l+k}(t) - \theta_l(t)\big)} - 1 \right)   + c_*, &\quad d\neq 0,
    \\[0.3cm]
    \sum_{k=-N}^N a_k \big(\theta_{l+k}(t) - \theta_l(t) \big)   + c_*, &\quad d=0,
\end{cases}
\end{equation}
where we have recalled the 
coefficients $(a_k)$ and parameter $d$ 
that were introduced in 
\eqref{eqn:main:a_k}
respectively
\eqref{eqn:main:d}. The label `ch' refers to the fact that a Cole-Hopf transformation can be used to recast the nonlinear system for $d \neq 0$ into the linear system prescribed for $d = 0$. This reduction is essential
for our analysis  in {\S}\ref{sec:theta},
where we obtain decay rates for solutions to \eqref{eqn:main:main_eqn_for_theta}, 
based on the linear theory 
that we develop in {\S}\ref{sec:lde}.

The decision to use \eqref{eqn:main:Theta} is
hence primarily based on technical considerations. Nevertheless, 
it is possible to build a bridge back to
the discrete curvature flow 
\eqref{eq:int:disc:cv:flow}. 
To this end, we
recall the definitions
\eqref{eq:int:c:gamma}
and \eqref{eq:int:lap:beta:gamma} and introduce the operator
$\Theta_{\mathrm{dmc}}:\ell^\infty(\Z) \mapsto \ell^\infty(\Z) $ 
that acts as 
\begin{equation}\label{eqn:main:Theta_dmc}
\Theta_{\mathrm{dmc}}(\theta) =  \kappa_H \frac{\Delta_\theta}{\beta_\theta^2} + \beta_\theta \overline{c}_{\theta} ,
\end{equation}
which depends on the sequences
\eqref{eq:int:nrm:A:C}
and the curvature coefficient $\kappa_H > 0$.

The result below shows that $\Theta_{\mathrm{dmc}}$ 
can be tailored to agree with
$\Theta_{\mathrm{ch}}$ up to terms
that are cubic in the first-differences
\begin{equation*}
    [\partial \theta]_l = \theta_{l+1} - \theta_l.
\end{equation*}
We will see in {\S}\ref{sec:theta}
that such terms decay at a rate of $O(t^{-3/2})$, which in theory is sufficiently fast to be absorbed by our error terms. However,
due to the loss of the comparison principle 
we did not attempt to compare the actual solutions to the respective LDEs as was possible in \cite[Prop. 8.2]{jukic2019dynamics}.

\begin{proposition}[{see $\S$\ref{sec:theta}}]\label{prop:main:mcf} 
Assume that (H$g$), (H$\Phi$),  (H0), (HS$)_1$, (HS$)_2$ all hold. Assume furthermore that
$\kappa_H = - {[\partial^2_{\omega} \lambda_{\omega}]_{\omega = 0}}/{2}$. Then
there exists a unique set of coefficients $(A_k, B_k)_{k=-N}^N$
that satisfy
the identities~\eqref{eq:int:nrm:A:C}
and allow us to find a constant $K>0$ for which
\begin{equation}\label{eqn:main:Theta_ch-Theta_dmc}
\norm{\Theta_\mathrm{ch}(\theta) - \Theta_{\mathrm{dmc}}(\theta)}_{\ell^\infty} \leq K \norm{\partial \theta}^3_{\ell^\infty}
\end{equation}
holds for all sequences $\theta\in \ell^\infty(\Z)$ with $\norm{\partial \theta}_{\ell^\infty}\leq 1$.
On the other hand, such coefficients do not exist if 
\eqref{eq:int:cond:kappa:d}
is violated.
\end{proposition}

Our main result below makes the asymptotic connection between $\gamma$ and 
solutions $\theta$ to \eqref{eqn:main:main_eqn_for_theta} fully precise. This allows us to gain detailed control over the long-term dynamics of the phase $\gamma(t)$, which can be used
to provide stability results outside the 
`local' regimes treated in \cite{hoffman2015multi}
and~\cite{hoffman2017entire}.

\begin{thm} [{see {\S \ref{sec:asymp}}}] \label{thm:main_result:gamma_mcf}
Assume that (H$g$), (H$\Phi$),  (H0), (HS$)_1$, (HS$)_2$ all hold and let $u$ be a solution of \eqref{eqn:main:discrete_AC_new} with the initial condition \eqref{eqn:main:initial_condition_new}. Then for every $\epsilon >0$, there exists a constant $\tau_\epsilon > 0$ so that 
for any $\tau \ge \tau_{\epsilon}$, the solution $\theta$ of LDE~\eqref{eqn:main:main_eqn_for_theta} with the initial value 
$\theta(0) = \gamma(\tau)$
satisfies 
\begin{equation}
\label{eq:phase:gamma:apx:by:v}
    \norm{\gamma(t) - \theta(t-\tau)}_{\ell^\infty} \leq \epsilon, 
    \qquad \qquad 
    t\geq \tau. 
\end{equation}
\end{thm}

Our final result should be seen as an example of an asymptotic analysis that is made possible
by the phase tracking \eqref{eq:phase:gamma:apx:by:v}.
In particular, we
show that the planar travelling wave 
\eqref{eq:mr:def:planar:wave}
is stable with asymptotic phase under localized perturbations from a front-like background state that
is periodic in $l$. Indeed, 
such an assumption provides sufficient control on the solution $\theta$ to~\eqref{eqn:main:main_eqn_for_theta} to 
establish the uniform convergence $\theta_l \to c_* t + \mu$.
We emphasize that the case $P = 1$ encompasses the
stability results from~\cite{hoffman2015multi} and \cite{hoffman2017entire}. The key point is that an asymptotic global phaseshift $\mu \neq 0$ for the case $P \ge 2$ can be seen as
an `infinite-energy' shift of the underlying planar wave. In such cases
the  quadratic terms
in \eqref{eqn:main:main_eqn_for_theta}
can no longer be absorbed into higher-order residuals as in~\cite{hoffman2015multi} and \cite{hoffman2017entire}.

 \begin{thm}[{see {\S \ref{sec:asymp}}}] \label{thm:mr:periodicity+decay:stability}
Assume that (H$g$), (H$\Phi$),  (H0), (HS$)_1$, (HS$)_2$ all hold and let $u$ be a solution of \eqref{eqn:main:discrete_AC_new} with the initial condition \eqref{eqn:main:initial_condition_new}.
Suppose furthermore that there exists a 
sequence $u^{0;\mathrm{per}} \in \ell^\infty(\Z^2_\times)$
so that the following two properties hold.
\begin{itemize}
    \item[(a)]{
      We have the limit
      \begin{equation}
      \label{eq:mr:lim:per:j:pm:infty}
    u^0_{n,l} - u^{0;\mathrm{per}}_{n,l} \to 0, \quad \text{as}\quad  |n|+|l|\to\infty .
\end{equation}
    }
    \item[(b)]{
      There exists an integer $P \ge 1$ so that
      \begin{equation}
          u^{0;\mathrm{per}}_{n, l+\sigma_*^2 P}
          =  u^{0;\mathrm{per}}_{n, l}
          \qquad \qquad
          \hbox{ for all }  (n,l) \in \Z^2_\times.
      \end{equation}
     }
\end{itemize}
Then there exists a constant $\mu\in \R$ for which
we have the limit
\begin{equation}
  \label{eqn:periodicity+decay:stability}
    \lim_{t\to\infty} \sup_{(n,l)\in \Zs} |u_{n,l}(t) - \Phi_*(n-c_*t -\mu)| = 0. 
\end{equation}
\end{thm}



\subsection{Numerical results}


\begin{figure}[t]
\centering
 \includegraphics[width = \linewidth]{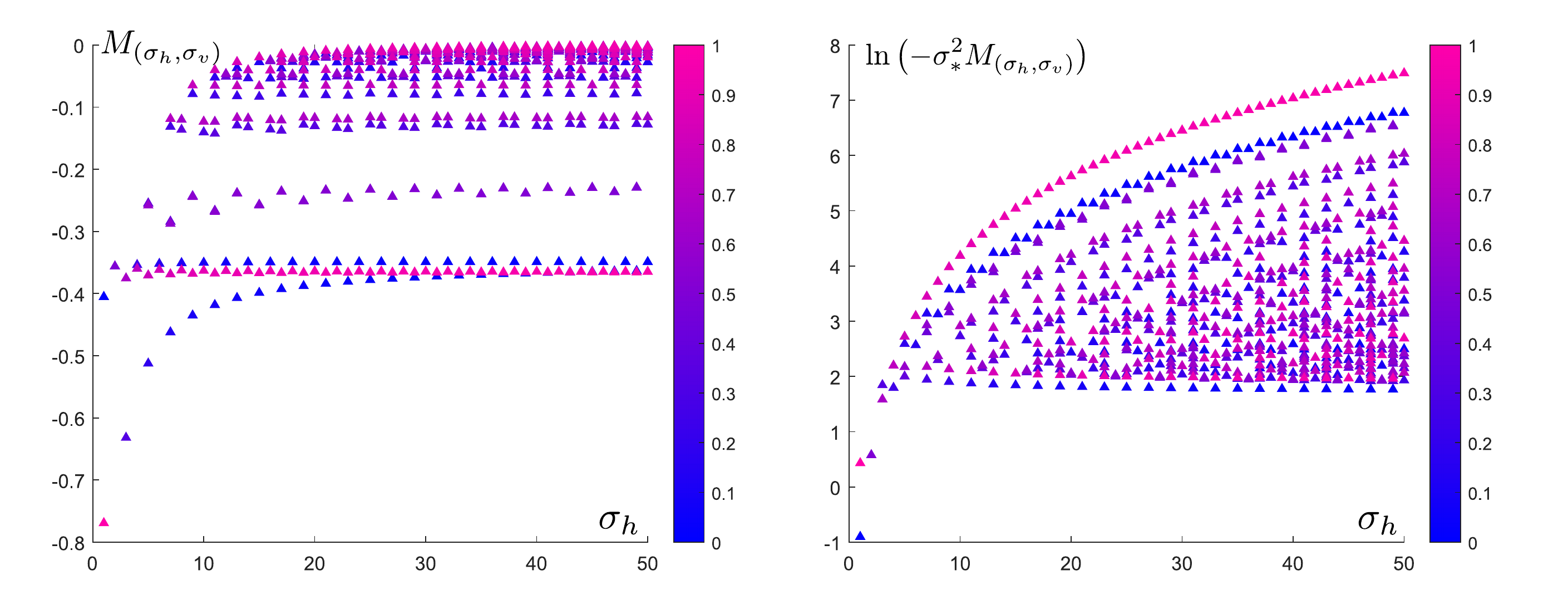}
    \caption{These plots represent the outcome of our numerical computations for the values $M_{(\sigma_h, \sigma_v)}$,
    where we used $g(u;a) = 6 u(u-1)(u-a)$ with   $a = 0.45$. For each fixed $\sigma_h$ (horizontal) we computed these values for each integer $1 \le \sigma_v \le \sigma_h$
    that has $\mathrm{gcd}(\sigma_h, \sigma_v) =1$,
    using the color code to represent
    the fraction ${\sigma_v}/{\sigma_h}$.
    On the left we see the formation of horizontal bands of the same color, suggesting the
    possibility to take limits along 
    convergent subsequences $ ({\sigma_v^{(n)}}/{\sigma_h^{(n)}})_{n > 0}$;
    see also Figure \ref{fig:main:numerics2}.
    The $\sigma_*^2$-scaling on the right 
    shows that our condition requiring $M_{(\sigma_h, \sigma_v)}$ to be negative 
    can be confirmed in a robust fashion.
     } 
 \label{fig:main:numerics1}
 \end{figure}
\begin{figure}[t]
\centering
 \includegraphics[width = \linewidth]{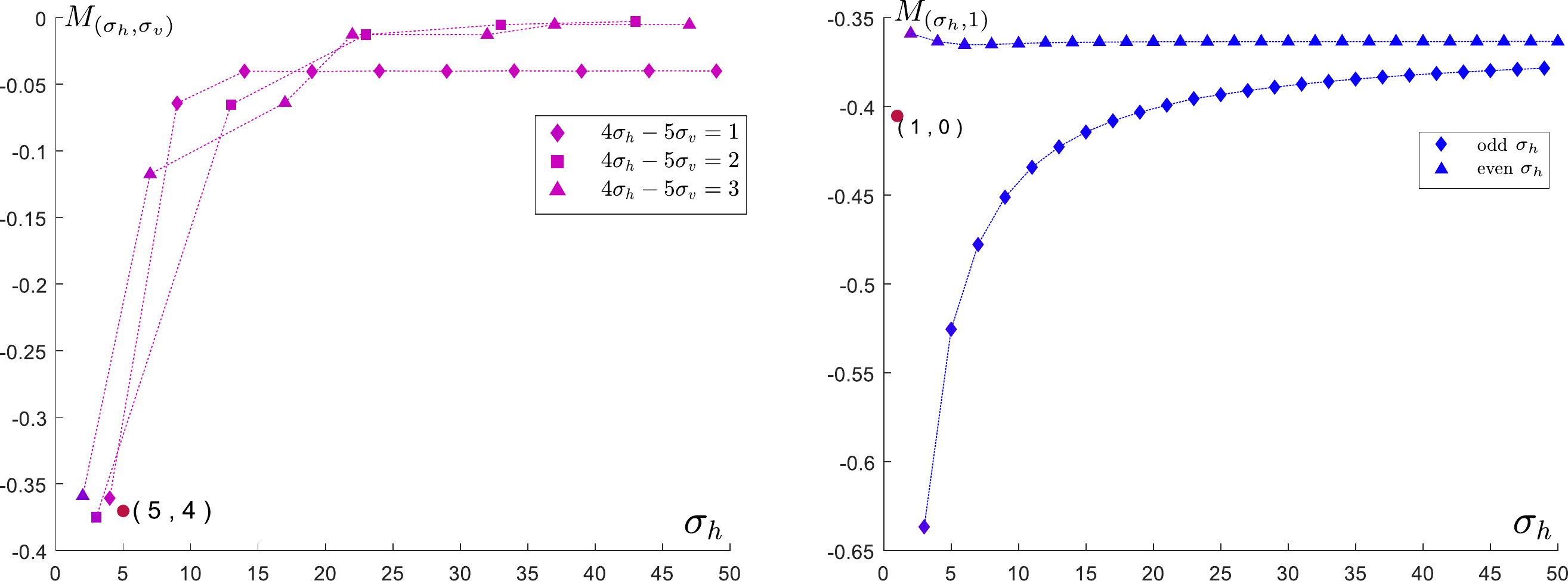}
    \caption{ 
    These plots track the values of
    $M_{(\sigma_h, \sigma_v)}$ along several
    subsequences of
    fractions $\sigma_v /  \sigma_h $
    that converge to $4/5$ (left)
    or zero (right). In all cases the limits
    are strictly above the
    values $M_{(5,4)}$ (left) and $M_{(1,0)}$ (right)
    corresponding to the limiting angles,
    supporting the inequality
    \eqref{eqn:main:numerics:ineq}.
    } 
 \label{fig:main:numerics2}
 \end{figure}

\begin{table}[t]
\centering
\begin{tabular}{||c||c|c|c|c|c|c|c|c|c|c|}
\hline
$k$   & 1   & 2     & 3    & 4     &5    & 6     & 7     & 8      & 9    & 10   \\ \hline
$a_{k}$ &0 & 0.896 & 0.195 &  0.068 & 0.966 & 0 &-0.143 & 0 & 0 & 0.05 \\ \hline
$a_{-k}$ & 0& 0.912 & 0.0925 & 0.179 & 1.005 & 0 & -0.199 & 0 &  0 & 0.03 \\ \hline
\end{tabular}
\caption{\label{tab:numerics:a_k} Numerically computed values for the coefficients $(a_k)$ defined in~\eqref{eqn:main:a_k} for the 
propagation direction $(\sigma_h, \sigma_v) = (2,5)$, with the nonlinearity $g(u;a) = 6 u(u-1)(u-0.45)$.  We computed these coefficients for a large range of angles and used them to calculate
the values $M_{(\sigma_h, \sigma_v)}$
depicted in Figures \ref{fig:main:numerics1}-\ref{fig:main:numerics2}.
}
\end{table} 
Our goal here is to numerically investigate the condition (HS$)_2$. 
In order to compensate for the fact that $f$ is locally quadratic around $0$,
we calculated the values $$ M_{(\sigma_h, \sigma_v)}: = \sup_{0 < |\omega| \le \pi} \frac{f_{(\sigma_h, \sigma_v)}(\omega)}{\omega^2}$$
for a large range of parameters $(\sigma_h, \sigma_v)\in \Z^2$. 

As a first step, we 
numerically solved the coupled set of equations
\begin{align}
    -c_* \Phi_*'(\xi) &= \Phi_*(\xi+\sigma_h) + \Phi_*(\xi-\sigma_h) + \Phi_*(\xi+ \sigma_v) + \Phi_*(\xi- \sigma_v) - 4 \Phi_*(\xi) + g\big(\Phi_*(\xi)\big), \label{eqn:numerical:phi}\\[0.3cm]
    c_* \psi_*'(\xi) &= \psi_*(\xi+\sigma_h) + \psi_*(\xi-\sigma_h) + \psi_*(\xi+ \sigma_v) + \psi_*(\xi- \sigma_v) - 4 \psi_*(\xi) + g'\big(\Phi_*(\xi)\big) \psi_*(\xi) \label{eqn:numerical:psi}
\end{align}
on a domain $[-L, L]$ for some large $L\gg1$, using the boundary conditions
\begin{equation*}
    \Phi_*(-L) = 0, \qquad \Phi_*(L) = 1, \qquad \psi_*(\pm L) = 0.
\end{equation*}
Due to the fact that the solutions are shift-invariant, we also fixed $\Phi(0) = \frac{1}{2}$ and $\psi(0) = 1$. 
In order to overcome the issue
that $L$ needs to be very large
when $\sigma_h$ or $\sigma_v$  is large,
we used the representation
$$\sigma_* = \sqrt{\sigma_h^2 +\sigma_v^2}, \qquad (\sigma_h, \sigma_v) = \sigma_*(\cos\zeta_*, \sin \zeta_*)$$
to introduce the rescaled functions
\begin{align*}
    \tilde{\Phi} (\xi) := \Phi_*({\xi}/{\sigma_*}), \qquad \tilde{\psi} (\xi) := \psi_*({\xi}/{\sigma_*}).
\end{align*}
These must satisfy the 
equivalent system of equations
\begin{align}
    -\frac{c_*}{\sigma_*} \tilde{\Phi}'(\xi) &= \tilde{\Phi}(\xi+\cos \zeta_*) + \tilde{\Phi}(\xi-\cos \zeta_*) + \tilde{\Phi}(\xi+\sin \zeta_*)+ \tilde{\Phi}(\xi-\sin \zeta_*)  + g\big(\tilde{\Phi}(\xi)\big) \label{eqn:numerical:phi2}\\[0.3cm]
    \frac{c_*}{\sigma_*} \tilde{\psi}'(\xi) &= \tilde{\psi}(\xi+\cos \zeta_*) + \tilde{\psi}(\xi-\cos \zeta_*) + \tilde{\psi}(\xi+\sin \zeta_*)+ \tilde{\psi}(\xi-\sin \zeta_*)  + g'\big(\tilde{\Phi}(\xi)\big)\tilde{\psi}(\zeta), \label{eqn:numerical:psi2}
\end{align}
which allowed us to keep $L$ fixed and use
a continuation approach to vary the angle $\zeta_*$.

We discretized the domain by dividing the segment $[-L, L]$ into $N_L$ parts of size $\Delta \xi$ for some integer $N_L \gg1$ and step size $\Delta \xi \ll 1$,
discretizing the first 
derivatives in \eqref{eqn:numerical:phi2}-\eqref{eqn:numerical:psi2} by the fourth-order central difference scheme. We proceeded by
using a nonlinear system solver to obtain the speed $c$ and the values $\big(\tilde{\Phi}(\xi_n),\tilde{\psi}(\xi_n)\big)$ in the nodes $\xi_n = -L + n \Delta \xi $, for $n = 0, \dots N_L$.
We subsequently used these values
to solve the systems~\eqref{eqn:main:eqns_for_p} and 
compute the coefficients needed to construct the function $f_{(\sigma_h, \sigma_v)}$ defined in~\eqref{eqn:main:function_f}. 
As an example, in 
Table \ref{tab:numerics:a_k}
we present the values of  $(a_k)$ for the angle of propagation $(2,5)$,
noting that both positive and negative values occur.

Our full results are visualized in Figures~\ref{fig:main:numerics1} and~\ref{fig:main:numerics2}.
 In all cases the value $M_{(\sigma_h, \sigma_v)}$ was  negative,
 hence validating (HS$)_2$.
 In addition, we observed
 that if we pick a sequence of angles $(\sigma_h^n, \sigma_v^n)$ 
 for which we have the convergence
 \begin{equation}
     \lim_{n \to \infty} \sigma_v^n / \sigma_h^n = \sigma_v^* / \sigma_h^*
 \end{equation}
 for some pair $(\sigma_h^*, \sigma_v^*) \in \mathbb{Z}^2$ not contained
 in this sequence, then
 \begin{equation}\label{eqn:main:numerics:ineq}
     \liminf_{n\to\infty}  M_{(\sigma_h^n, \sigma_v^n)} >   M_{(\sigma_h^*, \sigma_v^*)}.
 \end{equation}
 This behaviour closely resembles the crystallographic pinning phenomenon
 discussed in 
 \cite{hoffman2010universality,MPCP},
 where the role of $M$ is played 
 by the direction-dependent 
 boundary of the parameters $a$ where the 
 wave is pinned ($c_* = 0$).

\section{Omega limit points}
\label{sec:omega}

Both the construction of the phase $\gamma$ as well as the proof of Proposition~\ref{thm:main_results:gamma_approx_u} rely heavily on the  properties of so-called $\omega$-limit points. Intuitively, these track the long-time behaviour of $u$ after correcting for the velocity of the planar wave.  
To be more precise, 
let us consider a sequence
in $\Zs\times\R$
that is taken from the subset
\begin{equation}\label{eqn:omega:set_S}
    \mathcal{S} = \left\{(n_k, l_k, t_k)_{k\ge 0}: 0<t_1<t_2<\dots\to\infty, \ | n_k-ct_k|\leq M \text{ for some } M>0 \right\}.
\end{equation}
For any solution $u\in C^1\left([0, \infty), \ell^\infty(\Z_\times^2)\right)$, 
our goal is to analyze the limiting behaviour of the shifted solutions $u_{n+n_k, l+l_k}(t+t_k)$.
In the special case
that $u$ is the exact planar wave solution
$$u_{n,l}(t) = \Phi_*(n-c_*t),$$
the fact that the sequence $n_k-c_* t_k$ is bounded allows us to find a constant $\theta_0\in \R$ for which the convergence
\begin{equation}\label{eqn:omega_limit_points:perturbed_wave_ex}
    u_{n+n_k, l+l_k}(t+t_k)
    = \Phi_*(n+n_k -c_*t- c_*t_k )
    \to \Phi_*(n-c_*t+\theta_0)
\end{equation}
holds on some subsequence.
The limiting function is hence
equal to our planar wave, albeit with a perturbed phase $\theta_0$. 

Our main result here states that the convergence result~\eqref{eqn:omega_limit_points:perturbed_wave_ex} continues to hold for a much larger set of solutions of the discrete Allen-Cahn equation~\eqref{eqn:main:discrete_AC_new}. This generalizes our earlier results in~\cite{jukic2019dynamics} where we only considered horizontal directions. Although some minor technical obstacles need to be resolved, the main principles are comparable. In fact, 
we actually sharpened the setup slightly by avoiding the superfluous 
usage of the floor and ceiling functions in \cite[Prop. 3.1]{jukic2019dynamics}.  
This allows for a more efficient and readable analysis here and in the sequel.

\begin{proposition}\label{prop:omega_limit_points:main_prop}
Suppose that (H$g$), (H$\Phi$) and (H$0$) hold and let $u$ be a solution of \eqref{eqn:main:discrete_AC_new} with the initial condition \eqref{eqn:main:initial_condition_new}.
Then for any sequence $(n_k, l_k, t_k)_{k \ge 0} \in \mathcal{S}$  there exists a subsequence $(n_{i_k}, l_{i_k}, t_{i_k})_{k \ge 0}$ and a shift $\theta_0\in\R$ such that
 \begin{equation}
      u_{n + n_{i_k} , l+ l_{i_k} }(t+ t_{i_k}) \to \Phi_*(n-c_*t+ \theta_0) \quad \text{in} \quad C_{\mathrm{loc}}(\Z_\times^2\times \R).
 \end{equation}
\end{proposition}
\vspace{0.5cm}
The proof follows directly by combining the two main ingredients
that we state below.
First, in Proposition~\ref{cor:omega_limit_points:construction2},  we use Arzela-Ascoli to construct a solution $\omega\in C^1\left(\R; \ell^\infty(\Z_\times^2)\right)$ to the discrete Allen-Cahn equation~\eqref{eqn:main:discrete_AC_new} on $\Z_\times^2\times\R$ as a limit of the sequence $u_{n+n_k, l+l_k}(t+t_k)$. Furthermore, we show that this solution $\omega$ lies between two travelling waves. Proposition~\ref{prop:trapped_entire_sol:every_trapped_entire_solution_is_a_traveling_wave} subsequently states that this latter property is sufficient to guarantee  that $\omega$ is a travelling wave itself. 
This transfers the comparable result in~\cite{berestycki2007generalized} 
from the continuous to the discrete setting. 

\begin{proposition}\label{cor:omega_limit_points:construction2}
Consider the setting of Proposition~\ref{prop:omega_limit_points:main_prop}
and pick a sequence $(n_k, l_k, t_k)_{k \ge 0} \in \mathcal{S}$. Then there exists a subsequence $(n_{i_k}, l_{i_k}, t_{i_k})_{k \ge 0}$ and a function $\omega\in C^1(\R; \ell^\infty(\Z_\times^2))$ 
that satisfy the following claims.
 \begin{enumerate} [(i)]
     \item  We have the convergence 
    \begin{equation}
            u_{n + n_{i_k} , l+ l_{i_k} }(t+ t_{i_k}) \to  \omega_{n,l} (t)  \quad \text{in} \ C_{\text{loc}}(\Z_\times^2\times \R). 
    \end{equation}
    \item The function $\omega$ satisfies the discrete Allen-Cahn equation~\eqref{eqn:main:discrete_AC_new} on $\Zs\times \R$.
   \item There exists a constant $\theta\in \R$ such that 
    \begin{equation}\label{eqn:omega_limit_points:bounds_on_omega}
       \Phi_*(n-c_*t -\theta) \leq \omega_{n,l}(t) \leq \Phi_*(n-c_*t +\theta), \quad \text{for all}\ (n,l)\in \Z_\times^2.
    \end{equation}
 \end{enumerate} 
  \end{proposition}

\begin{proposition}\label{prop:trapped_entire_sol:every_trapped_entire_solution_is_a_traveling_wave}
	Assume that (H$g$) and (H$\Phi$) are satisfied
	and consider a
	function $\omega \in C^1\left(\R; \ell^\infty( \Z_\times^2)
	\right)$
	that satisfies
	the Allen-Cahn LDE~\eqref{eqn:main:discrete_AC_new} for all $t \in \R$.
	Assume furthermore that there exists a constant $\theta$ for which the bounds
	\begin{equation}\label{eqn:trapped_entire_sol:function_trapped_between_two_tw}
	 \Phi_*(n-c_*t -\theta) \leq \omega_{n,l} (t) \leq \Phi_*(n-c_*t +\theta) 
	\end{equation}
	hold for 
	all $(n,l) \in  \Z_\times^2$
	and $t\in \R$.
	Then there exists a constant $ \theta_0\in [-\theta, \theta] $  so that
	\begin{equation}
	\omega_{n,l}(t) = \Phi_*(n-c_*t +\theta_0), \quad \hbox{for all } (n,l)\in \Z_\times^2, \, t\in \R.
	\end{equation}
\end{proposition} 
 \begin{proof}[Proof of Proposition~\ref{prop:omega_limit_points:main_prop}]
 The claim follows directly  from  Propositions~\ref{cor:omega_limit_points:construction2} and~\ref{prop:trapped_entire_sol:every_trapped_entire_solution_is_a_traveling_wave}. 
 \end{proof}
 \subsection{Construction of \texorpdfstring{$\omega$}{Lg}}
Our first result provides preliminary upper and lower bounds for the solution  $u$. It is based upon a
standard comparison principle argument that can be traced back to Fife and Mcleod~\cite{fife1977approach}.

\begin{lemma}\label{lemma:omega_limit_points:upper and lower bounds}
Assume that (Hg), $(H\Phi)$ and (H0) are satisfied. 
Then there exists a time $T>0$ together with constants 
\begin{equation}
    q_1\in(0,a), \qquad q_2\in(0,1-a), 
    \qquad  \theta_1\in \R , 
    \qquad 
    \theta_2\in \R,
    \qquad \mu > 0, \qquad C > 0
\end{equation} 
so that the solution $u$ to \eqref{eqn:main:discrete_AC_new} with the initial condition \eqref{eqn:main:initial_condition_new} satisfies the estimates
\begin{align}
u_{n,l}(t) \leq \Phi_*\left(n+ \theta_1 - c_*(t-T) + Cq_1\big(1-e^{-\mu (t-T)}\big)\right) + q_1 e^{-\mu (t-T)}, \quad \forall t\geq T, \label{eqn:super_and_sub_solutions_for_the_discrete_AC:upper_bound} \\
u_{n,l}(t) \geq \Phi_*\left(n-\theta_2 - c_*(t-T) - Cq_2\big(1-e^{-\mu (t-T)}\big)\right) - q_2 e^{-\mu (t-T)}, \quad \forall t\geq T. \label{eqn:super_and_sub_solutions_for_the_discrete_AC:lower_bound}
\end{align}
\end{lemma}
\begin{proof}
The result can be shown by following the procedure outlined in the proof of Lemma 3.5 in~\cite{jukic2019dynamics}, using the inequality 
\begin{equation*}
    \alpha |c_*| -2(\cosh {\sigma_h c_*}-1) -2(\cosh {\sigma_v c_*}-1) \geq \dfrac{2K}{a-d}
\end{equation*}
to replace (3.14) in~\cite{jukic2019dynamics} and 
modifying the definition (3.16) in~\cite{jukic2019dynamics} to read
\begin{equation*}
    w_{n,l}(t) =  d + Me^{|c_*|(n + \alpha t)}.
\end{equation*}
\end{proof}
\begin{proof}[Proof of Proposition~\ref{cor:omega_limit_points:construction2}]
Fix an integer $T\in \N$ and denote by $M_T$ the number of points in $\Z_\times^2 $ that are also contained in the square $[-T, T]^2$, i.e. $M_T = \#\left\{(n,l)\in \Z_\times^2\cap [-T,T]^2 \right\}$. Consider the functions
\begin{equation*}
    u^k \in C\left([-T,T]; \R^{M_T\times M_T} \right)
\end{equation*}
that are defined by
\begin{equation*}
    u^k_{n,l}(t) = u_{n+n_k, l+l_k} (t+t_k)
\end{equation*}
for all sufficiently large $k$. From Lemma~\ref{lemma:omega_limit_points:upper and lower bounds} it follows that the solution $u$ and consequently the functions $u^k$ are globally bounded, which in view of~\eqref{eqn:main:discrete_AC_new} implies that the same holds for the derivative $\dot{u}$. The sequence $u^k$ hence satisfies the conditions of the Arzela-Ascoli theorem and is thus relatively compact in $C\left([-T,T]; \R^{M_T\times M_T} \right)$.
Applying~\eqref{eqn:main:discrete_AC_new} and using a standard diagonalisation argument, we obtain a subsequence $u^{i_k}$ and a function $\omega:\R\to\ell^\infty(\Z_\times^2)$ for which the convergence
\begin{equation}\sup_{(n,l,t)\in K}|u^{i_k}_{n,l}(t) - \omega_{n,l}(t) | + |\dot{u}^{i_k}_{n,l}(t) - \dot{\omega}_{n,l}(t)| \to 0
\end{equation}
holds for every compact $K\subset \Z_\times^2\times\R$. This  yields items \textit{(i)} and \textit{(ii)}, while item~\textit{(iii)} follows from Lemma~\ref{lemma:omega_limit_points:upper and lower bounds}.
\end{proof}

\subsection{Trapped entire solutions}
The main aim of this subsection is to establish Proposition~\ref{prop:trapped_entire_sol:every_trapped_entire_solution_is_a_traveling_wave}, which states that every entire solution of the discrete Allen-Cahn equation on ${\Z^2_\times} \times\R$ trapped between two travelling waves is a travelling wave itself. At the heart of the proof lies a version of the maximum principle for LDEs which we provide below in Lemmas~\ref{lemma:trapped_entire_sol:max_principle} 
and~\ref{lemma:trapped_entire_sol:max_principle2}. As a preparation, we define the quantities
\begin{equation}
    \sigma_\infty : =\max\{|\sigma_h|, |\sigma_v| \}, \qquad  m_* : = \sigma_\infty - 1. 
\end{equation}

\begin{lemma}\label{lemma:trapped_entire_sol:max_principle}
Pick $\kappa\in \R$ and let $E_\kappa\subset\Z^2_\times \times \R$  be defined as 
\begin{equation}\label{eqn:trapped_entire_sol:def_of_E_kappa}
E_\kappa = \left\{(n,l,t)\in \Z^2_\times \times \R: n-c_*t \geq \kappa \right\}.
\end{equation} 
Pick $B\in \R$ and $\epsilon>0$ and assume that the function $z\in C^1\big(\R, \ell^\infty (\Z^2_\times)\big)$ satisfies the conditions
\begin{enumerate} [(i)]
    \item \label{item:omega:comparison_principle:z>0} $z_{n,l}(t)\geq 0$ for all $(n,l,t)\in E_\kappa$;
    \item \label{item:omega:comparison_principle:bc}  $z_{n,l}(t)\geq \epsilon$  for all $(n,l,t)\in E_\kappa$ with $ n-c_*t\in [\kappa, \kappa + m_*]$;
    \item \label{item:omega:comparison_principle:diff_ineq}
$
        \dot z_{n,l}(t) - (\Delta^\times z)_{n,l}(t) + B z_{n,l}(t) \geq 0
$ for all $(n,l,t)\in E_\kappa$.

\end{enumerate}
Then, in fact $z_{n,l}(t)>0$ for all  $(n,l,t) \in E_\kappa$. 
\end{lemma}
\begin{proof}
Assume to the contrary that there exists $( n_0,l_0, t_0 )\in E_\kappa$ for which $z_{n_0, l_0} (t_0) =0$.  Since the function $z$ attains its minimum at this interior point, we know that $\dot{z}_{n_0,l_0}(t_0)= 0$. In addition, assumption \textit{(\ref{item:omega:comparison_principle:bc})} ensures that $(\Delta^\times z)_{n_0,l_0}(t_0) \geq 0$. On the other hand, assumption \textit{(\ref{item:omega:comparison_principle:diff_ineq})} gives
 \begin{equation*}
    0\leq  \dot{z}_{n_0,l_0}(t_0) - (\Delta^\times z)_{n_0,l_0}(t_0) + B z_{n_0,l_0}(t_0) = - (\Delta^\times z)_{n_0,l_0}(t_0) \leq 0.
 \end{equation*}
 Therefore, the equality $(\Delta^\times z)_{n_0,l_0}(t_0) = 0$ must hold. In particular, we have 
 \begin{equation*}
     z_{n_0-\sigma_h, l-\sigma_v}(t_0) = z_{n_0+\sigma_h, l+\sigma_v}(t_0) = z_{n_0-\sigma_v, l+\sigma_h}(t_0) =  z_{n_0+\sigma_v, l-\sigma_h}(t_0) = 0 .
 \end{equation*} 
 We note that the inclusion $$n_0 - \sigma_\infty \in [\kappa, \kappa+m_* ]$$ would immediately contradict property \textit{(\ref{item:omega:comparison_principle:bc})}. On the other hand, if $$n_0 - \sigma_\infty ~\geq \kappa +  m_* + 1$$ we can repeat this procedure with $n_0-\sigma_\infty$ until the desired contradiction
 is reached. 
\end{proof}

\begin{lemma}\label{lemma:trapped_entire_sol:max_principle2}
Pick $\kappa\in \R$ and let $F_\kappa\subset\Z^2_\times \times \R$ be defined as 
\begin{equation}\label{eqn:trapped_entire_sol:def_of_F_kappa}
F_\kappa = \left\{(n,l,t)\in \Z^2_\times \times \R: n-c_*t \leq \kappa \right\}.
\end{equation} 
Pick $B\in \R$ and $\epsilon>0$ and assume that the function $z\in C^1(\R, \ell^\infty \big(\Z^2_\times)\big)$ satisfies the conditions
\begin{enumerate} [(i)]
    \item $z_{n,l}(t)\geq 0$ for  $(n,l, t)\in F_\kappa$;
\item $z_{n,l}(t)\geq \epsilon $ for all $(n,l,t)\in F_\kappa$ with $n-c_*t\in [\kappa - m_*, \kappa ]$;
    \item 
$
        \dot z_{n,l}(t) - (\Delta^\times z)_{n,l}(t) + B z_{n,l}(t) \geq 0
$ for all $(n,l,t)\in F_\kappa$.

\end{enumerate}
Then, in fact $z_{n,l}(t)>0$ on $F_\kappa$. 
\end{lemma}
\begin{proof}
The proof is almost identical to that  of Lemma~\ref{lemma:trapped_entire_sol:max_principle}.
\end{proof}

\begin{lemma}\label{lemma:trapped_entire_sol:first_lemma}
Consider the setting of Proposition~\ref{prop:trapped_entire_sol:every_trapped_entire_solution_is_a_traveling_wave} and pick a  sufficiently small $\delta>0$. 
Choose a pair $(N,L)\in \Z_\times^2 $ together with a constant $\rho\in \R$. Suppose for some $\kappa \in \Z$ that the function 
\begin{equation}\label{eqn:trapped_entire_sols:v_sigma}
    v^\rho_{n,l}(t) = \omega_{n+N, l+L}\left(t+{N}/{c_*} + {\rho}/{c_*}\right)
\end{equation}
satisfies the inequality
\begin{equation}
\label{eq:trp:ineq:v:sig:omega}
v^\rho_{n,l}(t) \leq \omega_{n,l}(t)
\end{equation}
whenever  $n-c_*t\in [\kappa , \kappa +m_*]$. 
Then the following claims holds true. 
\begin{enumerate}[(i)]
    \item\label{item:omega:vrho_lemma:item1} If  $\omega_{n,l}(t) \geq 1-\delta$ 
    whenever $n-c_*t\geq \kappa $, then 
    in fact \eqref{eq:trp:ineq:v:sig:omega}
    holds for all $n-c_*t\geq \kappa $.
    \item\label{item:omega:vrho_lemma:item2} If $v^\rho_{n,l}(t) \leq \delta $ whenever $n-c_*t\leq \kappa+ m_* $, then 
    in fact \eqref{eq:trp:ineq:v:sig:omega}
    holds for all  $n-c_*t\leq \kappa + m_*$.
\end{enumerate}
\end{lemma}
\begin{proof}
We follow the outline of the proof from~\cite[\S 4]{jukic2019dynamics}, which can be seen as a spatially discrete version of~\cite[\S 3]{berestycki2007generalized} where continuous travelling waves were considered. 
We only establish~\textit{(\ref{item:omega:vrho_lemma:item1})}, since~\textit{(\ref{item:omega:vrho_lemma:item2})} can be obtained in a similar fashion using the set $F_\kappa$ from Lemma~\ref{lemma:trapped_entire_sol:max_principle2} instead of the set $E_{\kappa}$ from Lemma~\ref{lemma:trapped_entire_sol:max_principle}.

Due to the global bounds on the functions $\omega$ and $v^\rho$, the quantity
\begin{equation}
    \epsilon^* = \inf\left\{\epsilon>0: v^\rho \leq \omega+\epsilon \ \text{in} \ E_\kappa\right\}
\end{equation}
is finite and by continuity we have
\begin{equation}\label{eqn:omega_limit_points:trapped:ineqality}
    v^\rho \leq \omega+\epsilon^* \quad \text{in} \ E_\kappa.
\end{equation}
To prove the claim we must show that $\epsilon^*=0$. 

Assuming to the contrary that $\epsilon^* > 0$, we find sequences $(n_k, l_k, t_k)\in E_\kappa$ and $\epsilon_k \nearrow \epsilon^*$ such that  
\begin{equation}\label{eqn:trapped_waves:ineq_omega_eps_v}
    \omega_{n_k, l_k}(t_k) + \epsilon_k < v^\rho_{n_k, l_k}(t_k) \leq \omega_{n_k, l_k}(t_k) +\epsilon^*.
\end{equation}
The right inequality above together with the bounds~\eqref{eqn:trapped_entire_sol:function_trapped_between_two_tw} implies that the sequence $n_k-c_*t_k$ is bounded.  Applying a similar construction to that in the proof of Corollary~\ref{cor:omega_limit_points:construction2},  we obtain a function $\omega^\infty \in C^1\left(\R;\ell^{\infty}(\Zs)\right)$ for which we have the limits
\begin{equation}\label{eqn:omega_limit_points:trapped:limit}
    \begin{aligned}
&\lim_{k\to\infty} \omega_{n+n_k , l + l_k} (t+t_k) =     \omega^\infty_{n, l}(t),
\\
&\lim_{k\to\infty} v^\rho_{n+n_k, l + l_k } (t+t_k) =     \omega^\infty_{n+N  , l+L}\left(t+N/c_*+ \rho/c_*\right).
    \end{aligned}
\end{equation}

We define the function $z\in C^1\left(\R, \ell^\infty(\Zs)\right)$ by
\begin{equation}
    z_{n,l}(t) = \omega^\infty_{n,l}(t) - \omega^\infty_{n+N,l+L}(t) + \epsilon^*
\end{equation}
and claim that $z$ satisfies conditions~\textit{(\ref{item:omega:comparison_principle:z>0})}-\textit{(\ref{item:omega:comparison_principle:diff_ineq})} of Lemma~\ref{lemma:trapped_entire_sol:max_principle} on the set $$E_0=\left\{(n,l,t)\in \Zs\times \R: n-c_*t \geq 0\right\}.$$
To see this, we first note that 
$n+n_k - c_*t -c_*t_k\geq \kappa$ holds 
 by construction on the set $E_0$. 
 Since the inequality~\eqref{eqn:omega_limit_points:trapped:ineqality} survives the limit~\eqref{eqn:omega_limit_points:trapped:limit}, we have $z_{n,l}(t)\geq 0$ on $E_0$, verifying~\textit{(\ref{item:omega:comparison_principle:z>0})}.  Turning to~\textit{(\ref{item:omega:comparison_principle:bc})}, we note that  the inequality~\eqref{eq:trp:ineq:v:sig:omega} implies that
    $$
    z_{n,l}(t) \geq \epsilon^*>0, \ \qquad \text{for} \ n-c_*t\in [0,  m_*].
    $$
    To establish~\textit{(\ref{item:omega:comparison_principle:diff_ineq})}, we pick $\delta>0$ in such a way that the function $g$ is non-increasing on the interval $[1-\delta, 1]$. Recalling that $\omega^\infty \in [1-\delta, 1]$ on $E_0$ and that $g$ is locally Lipschitz, we obtain the bound
 \begin{align*}
     \dot{z}_{n,l}(t) - (\Delta^\times z)_{n,l}(t) &= g\big(\omega^\infty_{n,l}(t)\big) - g\big(\omega^\infty_{n+N,l+L}(t)\big) \\
     & \geq g(\omega^\infty_{n,l}(t)+\epsilon^*) - g\big(\omega^\infty_{n+N,l+L}(t)\big) 
     \\
     &\geq -B z_{n,l}(t)
 \end{align*}
for any $(n,l,t) \in E_0$. We may hence apply  Lemma~\ref{lemma:trapped_entire_sol:max_principle} and conclude that $z>0$ on $E_0$. However, the inequalities~\eqref{eqn:trapped_waves:ineq_omega_eps_v} imply that
$        z_{0,0}(0) = 0 $,
 which is a contradiction. Therefore $\epsilon^*=0$ must hold, as desired. 
\end{proof}
\begin{lemma}\label{lemma:trapped_entire_sol:sigma leq 0}
 Consider the setting of Propostion~\ref{prop:trapped_entire_sol:every_trapped_entire_solution_is_a_traveling_wave},
 fix an arbitrary pair $(N,L)\in \Z^2$
 and recall the functions $v^\rho$ defined in \eqref{eqn:trapped_entire_sols:v_sigma}.
Then the quantity 
   \begin{equation}
       \rho_*:= \inf \left\{\rho\in \R: v^{\tilde{\rho}}\leq \omega  \text{ in } {\Z^2_\times}\times\R  \text{ for all }  \tilde{\rho} \geq \rho    \right\}
   \end{equation}
satisfies 
$\rho_*\leq 0$. 
\end{lemma}
   
\begin{proof} 
One can obtain this result by following the outline presented in the proof of~\cite[Lemma 4.4]{jukic2019dynamics}. Instead of~\cite[Lemma 4.3]{jukic2019dynamics},  one now needs to employ  Lemma~\ref{lemma:trapped_entire_sol:first_lemma}.   
\end{proof}

\begin{proof}[Proof of Proposition~\ref{prop:trapped_entire_sol:every_trapped_entire_solution_is_a_traveling_wave}]
From Lemma~\ref{lemma:trapped_entire_sol:sigma leq 0}, it follows that 
    \begin{equation}
        \omega_{n,l}(t)\geq \omega_{n+N,l+L}(t+{N}/{c_*}) \quad \text{on }{\Z^2_\times}\times\R. 
    \end{equation}
Since the pair $(N,L)\in \Z_\times^2$ is arbitrary, we can conclude that the function $\omega$ depends only on the difference $n-c_*t$. 
In particular, there exists a function $\varphi$ such that $\omega_{n,l}(t) = \varphi(n-c_*t)$. The result now follows directly from the fact that solutions to the travelling wave problem \eqref{eqn:main:mdfe}-\eqref{eqn:main_results_boundary_cond} for  $c_* \neq 0$ 
are unique up to translation; see~\cite{Mallet-Paret1999}.   
\end{proof}

\section{Large time behaviour of  \texorpdfstring{$u$}{Lg}}\label{sec:gamma}

In this section we establish Proposition~\ref{thm:main_results:gamma_approx_u} by studying the qualitative large time behaviour of the solution $u$ within the interfacial region 
\begin{equation}
    I_t = \left\{(n,l)\in \Zs: \Phi_*\left(-\sigma_*^2-1 \right)\leq u_{n,l}(t) \leq \Phi_*\left(\sigma_*^2+1\right) \right\}, 
\end{equation}
which represents the points at which a solution $u$ is close to $\Phi_*(0) = 1/2$.  The boundary values $\Phi_*(-\sigma_*^2-1)$  and $\Phi_*(\sigma_*^2+1)$  were carefully chosen to ensure that $I_t$ is nonempty for large $t$, which we show in  Proposition~\ref{prop:large_time_behaviour:interface_prop}. In addition, we  show that for a fixed pair $(l,t)$ the map $n \mapsto u_{n,l}(t)$ is monotone within $I_t$, in the sense that the differences $u_{n+\sigma_*^2, l}(t) - u_{n, l}(t) $ are  bounded from below uniformly in time. 

In addition to the monotonicity within $I_t$, the map $n \mapsto u_{n,l}(t)$ cannot exit throughout the lower boundary once it enters the interfacial region from below. Similarly, it cannot reenter the interval once it has left through the upper boundary. All together, these
results provide sufficient control in the crucial
region away from the stable equilibria zero and one
to uniquely define the phase $\gamma$ by the procedure described in \S\ref{sec:interface formation}.

The results of this section are a generalization of the results presented in~\cite[\S 5]{jukic2019dynamics},
requiring us to take into account several technical differences that arise due to the additional complexities of working with $\Zs$ rather than $\Z^2$. Moreover, our construction of the phase $\gamma(t)$ here
is more refined than the setup in ~\cite{jukic2019dynamics}, which also causes several modifications to the proofs.

\begin{proposition}\label{prop:large_time_behaviour:interface_prop}
Consider the setting of Proposition~\ref{thm:main_results:gamma_approx_u}. Then there exists $T>0$ such that the following claims hold true.
\begin{enumerate}[(i)]
    \item\label{item:large_time:interface_prop:nonempty} For each $t\geq T$ and $l\in \Z$ there exists $n_* = n_*(l,t)\in \Z$ for which      
 \begin{equation}\label{eqn:phase gamma:nonempty set around 1/2}
      \Phi_*(-\sigma_*^2-1) < u_{n_*,l}(t) \leq \frac{1}{2}.  
  \end{equation}
\item We have the inequality 
\begin{equation}\label{eqn:phase gamma:derivative bounded from below}
	\begin{aligned} 
		&\inf_{t \ge T, \,(n, l)\in I_t} u_{n+\sigma_*^2, l}(t) - u_{n,l} (t) > 0. \\	
	\end{aligned}
	\end{equation}
\item Consider any $t\geq T$ and $(n,l) \in \Zs$ for which
$u_{n,l}(t) \leq \Phi_*(-\sigma_*^2-1)$ holds.
Then we also have $u_{n-\sigma_*^2,l}(t) \leq \Phi_*(-\sigma_*^2-1)$.
\item  Consider any $t\geq T$ and $(n,l) \in \Zs$
for which $u_{n,l}(t) \geq \Phi_*(\sigma_*^2+1)$
holds. Then we also have
$u_{n+\sigma_*^2,l}(t) \geq \Phi_*(\sigma_*^2+1)$.
\end{enumerate}
\end{proposition}
\begin{proof}[Proof of Proposition~\ref{prop:main:interface}]
The statement follows directly from item~\ref{item:large_time:interface_prop:nonempty} of Proposition~\ref{prop:large_time_behaviour:interface_prop}. 
\end{proof}

In the following proposition we provide an asymptotic flatness result for the phase $\gamma$. This feature is a crucial property that allows us to construct the super- and sub-solutions 
that we use in the proof of Theorem~\ref{thm:main_result:gamma_mcf} and consequently Theorem~\ref{thm:mr:periodicity+decay:stability}.   
 
\begin{proposition}\label{prop:large_time_behaviour:l_differences_gamma}
Consider the setting of the Proposition~\ref{prop:large_time_behaviour:interface_prop} and recall the phase $\gamma:[T, \infty) \to \ell^\infty(\Z)$ defined in~\eqref{eqn:main:def_of_gamma}. Then we have the limit
\begin{equation*}
    \lim_{t\to\infty} \sup_{l\in \Z} |\gamma_{l+1}(t) - \gamma_l(t)| = 0.
\end{equation*}
\end{proposition}

\subsection{Phase construction}
In this subsection we prove Proposition~\ref{prop:large_time_behaviour:interface_prop}, mainly by relying on the convergence results
from  Proposition~\ref{prop:omega_limit_points:main_prop}. 
As a preparation,
we define the set
\begin{equation}
    \mathcal{I}(T,R) : = \left\{ (n,l,t) \in \Zs\times [T, \infty): |n-c_*t |\leq R \right\}
\end{equation}
for any pair of positive constants $T$ and $R$
and we also remind the reader of the set of sequences $\mathcal{S}$ defined in~\eqref{eqn:omega:set_S}. 
\begin{lemma}\label{lemma:monotonicity_of_u_interface_region}
Consider the setting of Proposition~\ref{prop:large_time_behaviour:interface_prop} and pick a constant $R>0$. Then there exists a constant $T>0$ such that
\begin{equation*}
       \inf_{ (n,l,t) \in \mathcal{I}(T,R)} u_{n+\sigma_*^2,l}(t) - u_{n,l}(t) >0.
\end{equation*}
\end{lemma}
\begin{proof}
Assume to the contrary that there exists a constant $R>0$ such that 
\begin{equation}\label{eqn:monotonicity_interface_region}
    \inf_{(n,l,t)\in \mathcal{I}(T,R)} u_{n+\sigma_*^2,l}(t) - u_{n,l}(t) \leq 0
\end{equation}
holds for every $T>0$. That implies that we can find a sequence $(n_k, l_k, t_k)_{k \ge 0} \in \mathcal{S}$ such that
\begin{equation}\label{eqn:large_behaviour_u:inf_in_interface_ineq}
    u_{n_k+\sigma_*^2,l_k}(t_k) - u_{n_k,l_k}(t_k) \leq \dfrac{1}{k}.
\end{equation}
By virtue of Proposition~\ref{prop:omega_limit_points:main_prop}, we can find $\theta_0\in \R$ and pass to a subsequence for which we have the convergence
\begin{equation*}
    u_{n+n_k, l+l_k}(t+t_k)\to \Phi_*(n-c_*t +\theta_0) \text{ in } C_{\text{loc}}(\Zs\times\R).
\end{equation*}
Therefore, letting $k\to\infty$ in~\eqref{eqn:large_behaviour_u:inf_in_interface_ineq} leads to  
\begin{equation}\label{eqn:large_behaviour_u:omega_ineq}
     \Phi_*(\sigma_*^2 + \theta_0) - \Phi_*(\theta_0) \leq 0,
\end{equation}
which contradicts the monotonicity of the function $\Phi_*$.
\end{proof}

\begin{proof}[Proof of Proposition~\ref{prop:large_time_behaviour:interface_prop}]
We first establish~\textit{(iv)}. Arguing by contradiction, assume that there exists a sequence $(n_k, l_k, t_k)_{k \ge 0}$, $0<t_1<t_2<...\to \infty$ such that
\begin{equation}\label{eqn:large_behaviour_u:contradiction_item_3}
    u_{n_k, l_k}(t_k) \geq \Phi_*(\sigma_*^2+1) \quad \text{and} \quad   u_{n_k+\sigma_*^2, l_k}(t_k) < \Phi_*(\sigma_*^2+1). 
\end{equation}
The bounds in Lemma~\ref{lemma:omega_limit_points:upper and lower bounds} imply that the sequence $n_k - c_*t_k$ is bounded by some constant $R$. We can now apply Lemma~\ref{lemma:monotonicity_of_u_interface_region} to conclude 
$u_{n_k+\sigma_*^2, l_k}(t_k) -  u_{n_k, l_k}(t_k) > 0 $, which contradicts~\eqref{eqn:large_behaviour_u:contradiction_item_3} due to the strict monotonicity of the function $\Phi_*$. Items \textit{(i)} and 
\textit{(iii)} follow in a similar way. 

To prove item \textit{(ii)}, we 
choose a $T$ that satisfies \textit{(i)}, \textit{(iii)} and \textit{(iv)}  and pick $t\geq T$ together with $(n,l)\in I_t$. Upon further increasing $T$, Lemma~\ref{lemma:omega_limit_points:upper and lower bounds} implies that $n-c_*t $ is bounded by some constant $R>0$ that only depends on $T$. Therefore, we have shown that
\begin{equation*}
    \left\{(n,l,t):t\geq T, (n,l)\in I_t\right\}\subseteq \mathcal{I}(T, R).
\end{equation*}
The desired bound now follows directly from Lemma~\ref{lemma:monotonicity_of_u_interface_region}.
\end{proof}

\begin{lemma}\label{lemma:phase_gamma:gamma-ct_bounded}
   Consider the setting of Proposition~\ref{prop:large_time_behaviour:interface_prop}
and recall the phase $ \gamma: [T, \infty) \to \ell^\infty(\Z) $
defined in \eqref{eqn:main:def_of_gamma}. Then there exists $T_*\geq T$  such that
the difference $n_*(l,t) - c_* t$ is uniformly bounded
for $t\geq T_*$ and $l\in \Z$ . In particular,  we can find a constant $M>0$ so that
\begin{align*}
    \norm{\gamma(t) - c_* t}_{\ell^\infty} \leq M ,
    \qquad \qquad t \ge T_*.
\end{align*}
\end{lemma}
\begin{proof}
The proof is analogous to that of Lemma 5.4 in~\cite{jukic2019dynamics}. 
\end{proof}

\begin{proof}[Proof of Proposition~\ref{thm:main_results:gamma_approx_u}] 
    Arguing by contradiction once more, let us assume
    that there exists $\delta>0  $ together with sequences
    $(n_k, l_k) \in \Zs$ and $ T\leq t_1 < t_2 < \dots \to \infty $
    for which
	\begin{equation}\label{eqn:phase_gamma:assume_contrary_in_proof_that_gamma_approx_u}
	 | \Delta_k | := |u_{n_k,l_k}(t_k) - \Phi_*\big(n_k - \gamma_{l_k}(t_k)\big)|\geq \delta.
	\end{equation}
	Analogously as in the proof of \cite[Thm. 2.2]{jukic2019dynamics}, one can show that $n_k - c_* t_k$ is a bounded sequence. In addition, from Lemma~\ref{lemma:phase_gamma:gamma-ct_bounded} we also know that $n_*(l_k, t_k) - c_* t_k$ is bounded. Therefore, the sequence $n_*(l_k, t_k) - n_k$ is also bounded, allowing us to identify it with a constant $m\in \Z$. Applying Proposition~\ref{prop:omega_limit_points:main_prop} we find $\theta_0\in\R$ such that  the limit 
	\begin{equation}\label{eqn:phase:gamma-ct_bounded:limit}
	    u_{n+n_k, l+l_k}(t+t_k) \to \Phi_*(n-c_*t +\theta_0)
	\end{equation}
	holds for all $(n,l,t)\in \Zs\times\R$, after passing to a further subsequence. 	Recalling the definition~\eqref{eqn:main:def_of_gamma}, this leads to
	 \begin{equation*}
	 \begin{aligned}
	\Phi_*\big(n_k - \gamma_{l_k}(t_k)\big)  &= \Phi_*\big(n_k - n_*(l_k, t_k) - \vartheta_*(l_k, t_k)
	\big) \\
	& = \Phi_*\big(-m -  \vartheta_*(l_k, t_k) \big).
	\end{aligned}
	\end{equation*}
Due to~\eqref{eqn:phase:gamma-ct_bounded:limit} and definition~\eqref{eqn:main:zeta} of $\vartheta_*(l, t)$
we obtain the convergence
	\begin{equation}
	   \vartheta_*(l_k, t_k) \to -\dfrac{\sigma_*^2 \Phi_*^{-1}\big( \Phi_*(m+\theta_0)\big)}{\Phi_*^{-1}\big(\Phi_*(m+\sigma_*^2 + \theta_0)\big) - \Phi_*^{-1}\big(\Phi_*(m+\theta_0)\big)}  = -m-\theta_0
	\end{equation}
as $k \to \infty$, which in turn implies that 
\begin{equation*}
    \Phi_*\big(n_k - \gamma_{l_k}(t_k)\big)  \to \Phi_*(\theta_0). 
\end{equation*}
We hence find that
\begin{equation*}
    \Delta_k \to \Phi_*(\theta_0) - \Phi_*(\theta_0) = 0 
\end{equation*}
as $k\to\infty$, which clearly contradicts~\eqref{eqn:phase_gamma:assume_contrary_in_proof_that_gamma_approx_u}.
\end{proof}

\subsection{Phase asymptotics}
In this subsection we establish the asymptotic flatness result for the phase $\gamma(t)$ that was stated in Proposition~\ref{prop:large_time_behaviour:l_differences_gamma}. A key ingredient
is that 
the first differences of the function $l \mapsto n_*(l,t)$ 
can be uniformly bounded for large $t$.  

\begin{lemma}\label{lemma:phase_asymp:diff_of_n}
 Consider the setting of Proposition~\ref{prop:large_time_behaviour:interface_prop}. Then there exists a constant $\Tilde{T}>T$ so that for every $t\geq \tilde{T}$ and $l\in \mathbb{Z}$ we have
\begin{equation*}
    |n_*(l+1,t) - n_*(l,t)|\leq \sigma_*^2.
\end{equation*}
\end{lemma}
\begin{proof}
Assume to the contrary that there exist sequences $(n_k, \tilde{n}_k,l_k)_{k \ge 0} \subset \Z^3$, $(t_k)_{k \ge 0} \subset (0, \infty)$ with $T<t_1 < t_2<...\to\infty$ for which
\begin{equation} \label{eqn:phase_asymp:diff_geq_sigma1}
    |n_k - \tilde{n}_k|> \sigma_*^2
\end{equation} and
\begin{equation}\label{eqn:phase_asymp:inequalities1}
	\begin{cases}
	u_{n_k, l_k}(t_k) \leq 1/2,\\
	u_{n_k+\sigma_*^2, l_k}(t_k) > 1/2, 
	\end{cases}	\qquad \qquad \qquad \qquad
	\begin{cases}
	u_{\tilde{n}_k, l_k+1}(t_k)  \leq 1/2,\\
	u_{\tilde{n}_k+\sigma_*^2, l_k+1}(t_k) >   1/2 .
	\end{cases}    
\end{equation}
Since both sequences $n_*(l_k+1, t_k)-c_*t_k$ and $n_*(l_k, t_k)-c_*t_k$ are bounded, we can assume that their difference is constant and equal to $m \in \mathbb{Z}$, i.e.
\begin{equation*}
    m = \tilde{n}_{k} - n_k. 
\end{equation*}
With this notation we can apply Proposition~\ref{prop:omega_limit_points:main_prop} to find a constant $\theta_0\in \R $ for which
\begin{equation*}
	\begin{cases}
	u_{n_k, l_k}(t_k) \to \Phi_*(\theta_0), \\
	u_{n_k+\sigma_*^2, l_k}(t_k) \to \Phi_*(\theta_0 + \sigma_*^2), 
	\end{cases} \qquad  
	\begin{cases}
    u_{m+{n}_k, l_k+1}(t_k) \to \Phi_*(m+ \theta_0),\\	
    u_{m+\sigma_*^2+{n}_k, l_k+1}(t_k) \to  \Phi_*(m+ \theta_0 + \sigma_*^2).	\end{cases}    
\end{equation*}
Combining these limits with the inequalities~\eqref{eqn:phase_asymp:inequalities1} we find that $\theta_0$  necessarily satisfies
\begin{equation*}
    - \sigma_*^2 \leq \theta_0 \leq 0,  \qquad - \sigma_*^2 \leq  m+ \theta_0\leq 0.
\end{equation*}
This in turn implies that $|m|\leq \sigma_*^2 $, contradicting the strict inequality in~\eqref{eqn:phase_asymp:diff_geq_sigma1}.

\end{proof}

\begin{proof}[Proof of Proposition~\ref{prop:large_time_behaviour:l_differences_gamma}]
	Assume to the contrary that there exists $ \delta > 0 $ 
	together with subsequences $(l_k)_{k \ge 0} \subset \Z$ and 
	$ T\leq t_1<t_2<\dots \to \infty $
	for which
	\begin{equation}\label{eqn:phase_gamma:assumption_gamma_not_flat}
    	\delta\leq |\gamma_{l_k+1}(t_k) - \gamma_{l_k}(t_k)|.	    
	\end{equation}
	Lemma~\ref{lemma:phase_asymp:diff_of_n} assures us that it is possible to pass to a subsequence that has $$ n_*(l_k+1,t_k) = n_*(l_k, t_k) + m,  $$
	for some integer $m\in [0, \sigma_*^2]$. Recalling the definition~\eqref{eqn:main:zeta} for $\vartheta_*(l,t)$, we find
	\begin{equation}\label{eqn:diff_gamma_assumption1}
	\begin{aligned}
	\gamma_{l_k+1}(t) - \gamma_{l_k}(t) &= m + \vartheta_*( l_k+1,t_k) - \vartheta_*( l_k,t_k)  . 
	\end{aligned}
	\end{equation}
We now employ Proposition~\ref{prop:omega_limit_points:main_prop} to find $\theta_0\in \R$ such that for all $(n,l,t)\in \Zs\times\R$ we have
\begin{equation*}
    u_{n+n_k, l+l_k}(t+t_k) \to  \Phi_*(n-c_*t +\theta_0), \quad \text{as}\quad k\to\infty,  
\end{equation*}
which further implies that
\begin{equation*}
     \vartheta_*( l_k,t_k) \to -\theta_0, \qquad \vartheta_*( l_k+1,t_k) \to - m - \theta_0.
\end{equation*}
Taking the limit in~\eqref{eqn:diff_gamma_assumption1} we obtain
	\begin{equation*}
	    \delta \leq |m-\theta_0 -m +\theta_0|= 0,
	\end{equation*}
	which is a clear contradiction.
\end{proof}

\section{Linearized phase evolution}\label{sec:lde}

In this section we consider the lattice differential equation
\begin{equation}\label{eqn:lde:linear:main_eq}
    \dot{h}_l(t) = \sum_{k = -N}^N a_k \big[ h_{l+\mu_k}(t) - h_l(t)\big], \quad t>0
\end{equation}
with the initial condition
\begin{equation}\label{eqn:lde:linear:init}
    h(0) =h^0\in \ell^\infty(\Z).
\end{equation} 
In order to highlight the general applicability of our results, we step back here from the specific framework associated to~\eqref{eqn:main:main_eqn_for_theta}. Instead, we impose the following general assumption on the coefficients $ a = (a_k)_{k=-N}^N \subset \R$ 
and the shifts $\mu = (\mu_k)_{k=-N}^N\subset \Z$. 

\begin{itemize}
\item[(h$\alpha$)]\label{item:lde:halpha}
The function $f:[-\pi, \pi]\to \R$ defined by
\begin{equation}
\label{eq:lde:bnds:def:f}
    f(\omega):=\sum_{k=-N}^Na_k (\cos(\mu_k \omega) -1 ) 
\end{equation}    
   is strictly negative on $[-\pi, \pi]\backslash \{0\}$. Futhermore, the constant  
 $\Lambda\in \R$ defined by 
\begin{equation}\label{eqn:lde:def_of_lambda}
\Lambda:= \sum_{k=-N}^N a_k \mu_k^2 = -f''(0)
\end{equation} 
satisfies  $\Lambda> 0$.
\end{itemize}
Let us first observe that the assumption (h$\alpha$) implies that  we can find $m>0$ and $\kappa>0$ such that
\begin{align}
    f(\omega) &\leq -\dfrac{\Lambda}{2} \omega^2 \ & & \text{for} \ \omega\in [-\kappa, \kappa], & & & &\label{eqn:lde:f_quadratic}\\
    f(\omega) &< -m  & & \text{for}\ \omega \in [-\pi, -\kappa] \cup [\kappa, \pi]. & & & &\label{eqn:lde:f_bounded}.
\end{align}

For any $n\in \N_0$ we inductively define the $(n)$-th discrete derivative  $\partial^{(n)}:\ell^\infty(\Z)\to \ell^\infty(\Z) $ by writing
\begin{align*}
    [\partial^{(0)}\Gamma]_j :=  \Gamma_j, \qquad  
    [\partial^{(1)}\Gamma]_j := \Gamma_{j+1} - \Gamma_j 
\end{align*}
together with
\begin{equation}\label{eqn:lde:nth_derivative}
    [\partial^{(n)}\Gamma]_j = \left[\partial^{(1)} \left(\partial^{(n-1)}\Gamma)\right)\right]_j
\end{equation}
for $n>1$.
The first goal of this section is to establish decay estimates of the form \begin{equation}\label{eqn:lde:decay_intro}
    ||\partial^{(n)}h(t)||_{\ell^\infty}\sim~O(t^{-\frac{n}{2}})
\end{equation}
for the solution $h(t)$ of the system~\eqref{eqn:lde:linear:main_eq}-\eqref{eqn:lde:linear:init}.  
These rates are consistent with the estimates for solutions of the continuous heat equation $h_t =  h_{xx}$, which can be readily obtained by taking $x$-derivatives of the explicit representation 
\begin{equation}\label{eqn:lde:cont_heat_eqn}
    h(x,t) = \dfrac{1}{\sqrt{4\pi t}}\int_{\R} e^{-\frac{(x-y)^2}{4t}} h(y, 0) dy.
\end{equation}

Our goal is to find a 
solution formula for~\eqref{eqn:lde:linear:main_eq} equivalent to~\eqref{eqn:lde:cont_heat_eqn}, in the sense that it takes the form of the convolution between the fundamental solution with the initial condition. By finding such a representation, we can 
transfer discrete derivatives onto the fundamental solution to establish~\eqref{eqn:lde:decay_intro}.

\begin{thm}[{see \S\ref{subsec:lde:strategy}}]
\label{thm:lde:n-th-diff-theta}
Assume that condition (h$\alpha$) holds and pick $n\in \N_0$. Then there exists a constant $C=C(n)$ so that for any $h^0\in \ell^\infty(\Z)$, the $n$-th discrete derivative of the solution $h\in 
C^1\left([0, \infty);\ell^\infty(\Z)\right)$ to the initial value problem~\eqref{eqn:lde:linear:main_eq}-\eqref{eqn:lde:linear:init} satisfies the  bound
\begin{equation*}
    \norm{\partial^{(n)} h(t)}_{\ell^\infty } \leq C \min\left\{\norm{\partial^{(n)}  h^0}_{\ell^\infty}, \norm{ h^0}_{\ell^\infty} t^{-\frac{n}{2}} \right\}.
\end{equation*}
\end{thm}




The second main result of this section concerns lower and upper bounds for the  solution $h(t)$ that are sharper than the
$\ell^\infty$-bounds in Theorem~\ref{thm:lde:n-th-diff-theta}. 
In particular, we show that if the initial condition $h^0$ is bounded away from $0$, then the solution
$h(t)$ is positive for large time $t\gg 1$. 
Moreover, under the additional assumption that the first differences of $h^0$ are flat enough we obtain the same conclusion for all time $t\geq 0$. 
The key issue is that some of the coefficients $(a_k)$ are allowed to be negative, which causes the usual comparison principle to fail. Indeed, it can (and does)  happen that  a solution $h(t)$ admits negative values for a short time even if the initial condition is strictly positive.

\begin{proposition} [{see \S\ref{subsec:lde:strategy}}]
\label{prop:lde:comparison_principle}
Consider the setting of Theorem~\ref{thm:lde:n-th-diff-theta} and
pick $\varepsilon>0$. 
Then there exists a time $T=T(\varepsilon)>0$ and $C=C(T, \varepsilon)$ such that for all $t\geq T$ the following properties hold.
\begin{enumerate}[(i)]
    \item For any $h^0\in \ell^\infty(\Z)$ that has $h_k^0 \geq 0$ for all $k \in \Z$, we have the bounds
\begin{align}
   h_k(t) &\geq \inf_{j\in \Z} h^0_j - C \norm{\partial h^0}_{\ell^\infty},  \quad & & k\in \Z, \qquad t\in [0, T],  & &  \label{eqn:lde:cp:flat_diff}
   \\
   h_k(t) &\geq  \inf_{j\in \Z} h^0_j - \varepsilon\norm{h^0}_{\ell^\infty},   & &  k\in \Z,  \qquad t\geq T. & &  \label{eqn:lde:cp:large_time}
 \end{align}
 \item For any $h^0\in \ell^\infty(\Z)$ that has $h_k^0 \leq 0$ for all $k \in \Z$, we have the bounds
\begin{align}
   h_k(t) &\leq \sup_{j\in \Z} h^0_j + C \norm{\partial h^0}_{\ell^\infty},  \ \quad & & k\in \Z, \qquad t\in [0, T], & & \label{eqn:lde:cp:flat_diff:sup}
   \\
   h_k(t)&\leq  \sup_{j\in \Z} h^0_j + \varepsilon\norm{h_0}_{\ell^\infty},  & &   k\in \Z,  \qquad t\geq T. & &
\end{align}
\end{enumerate} 
\end{proposition}

\subsection{Strategy}\label{subsec:lde:strategy}
In order to find an explicit formula for the solution $h$ of the initial problem~\eqref{eqn:lde:linear:main_eq}-\eqref{eqn:lde:linear:init}, we note that a spatial  Fourier transform leads  to the decoupled sets of ODEs 
\begin{align*}
    \dfrac{d}{dt}\hat{h}(\omega, t) =\sum_{k=-N}^N a_k (e^{i\mu_k \omega} - 1)\hat{h}(\omega, t)
\end{align*}
for $\omega\in [-\pi, \pi]$. Introducing the function
\begin{equation}
\label{eq:lde:bnds:def:p}
    p(\omega) = \sum_{k=-N}^N a_k \sin{(\mu_k\omega)},
\end{equation} 
we hence obtain 
the convolution formula 
\begin{equation}\label{eqn:lde:convolution}
       h_l(t)  = \sum_{k\in \Z} h^0_k M_{l-k}(t),
\end{equation}
where the fundamental solution $M(t)$ is defined by \begin{equation}\label{eqn:lde:fundamental solution}
    M_l(t) = \dfrac{1}{2\pi}\int_{-\pi}^{\pi} e^{i l\omega }e^{tf(\omega)  + it
    p(\omega)} d\omega. 
\end{equation}
Notice that  assumption (h$\alpha$) ensures that $\norm{M(t)}_{\ell^\infty}\leq 1$ for every $t\geq 0$. 

The proof of Theorem~\ref{thm:lde:n-th-diff-theta} relies on the following two lemmas, in which we focus on the decay estimates for the $\ell^1$-norm of the n-th differences $\partial^{(n)}M(t)$. We obtain the necessary estimates by dividing the sum into two parts, based on the size of the term $|l/t + a\cdot\mu|$. We note that the constant  $-a\cdot \mu = -p'(0)$ is often referred to as the group velocity.  
It tracks the speed of the `center' of $M$ and - in context of \S\ref{sec:main} - is closely related to $[\partial_{\omega} \lambda_{\omega}]_{\omega=0}$ and $[\partial_\varphi c_\varphi]_{\varphi = 0}$.

\begin{lemma}[{see \S\ref{subsec:lde:global}}]
\label{lemma:lde:l_1_noncompact_interval}
Consider the setting of Theorem~\ref{thm:lde:n-th-diff-theta}. Then there exist positive constants $K = K( n)$ and $C=C( n)$ such that 
\begin{equation}\label{eqn:lde:non_compact_claim1}
    \sum_{|l/t+a\cdot\mu|\geq K} |[\partial^n M(t)]_l | \big[ 1 + |l| \big] \leq Ce^{-t}.
\end{equation} 
\end{lemma} 

\begin{lemma}[{see \S\ref{subsec:lde:core}}]
\label{lemma:lde:l_1_compact_interval}
Consider the setting of Theorem~\ref{thm:lde:n-th-diff-theta} and pick $K>0$. Then there exists $C = C(K,  n)>0$ such that 
\begin{equation*}
     \sum_{|l/t+a\cdot\mu|\leq K}|[\partial^n M(t)]_l| \leq C\min\left\{1, t^{-\frac{n}{2}}\right\}. 
 \end{equation*}
\end{lemma}

\begin{proof}[Proof of Theorem~\ref{thm:lde:n-th-diff-theta}]
In view of the convolution formula~\eqref{eqn:lde:convolution}, we have 
\begin{equation*}
    \norm{\partial^{(n)} h(t)}_{\ell^\infty} \leq \norm{h^0}_{\ell^\infty}  \norm{\partial^{(n)} M(t)}_{\ell^1}. 
\end{equation*}
Employing Lemmas~\ref{lemma:lde:l_1_noncompact_interval} and~\ref{lemma:lde:l_1_compact_interval} in combination with the fast decay of the exponential we obtain a constant $C = C(n)$ for which 
\begin{equation*}
    \norm{\partial^{(n)} h(t)}_{\ell^\infty} \leq C\norm{h^0}_{\ell^\infty}  t^{-\frac{n}{2}} .
\end{equation*}
On the other hand, by transferring the n-th difference operators to the sequence $h^0$, we can write
\begin{equation*}
    \norm{\partial^{(n)} h(t)}_{\ell^\infty} \leq \norm{\partial^{(n)}h^0}_{\ell^\infty}  \norm{ M(t)}_{\ell^1}.
\end{equation*}
Applying Lemmas~\ref{lemma:lde:l_1_noncompact_interval} and \ref{lemma:lde:l_1_compact_interval} with $n=0$ now leads to the desired bound. 
\end{proof}

To prove the  lower bounds for solution $h(t)$ that are formulated in Proposition~\ref{prop:lde:comparison_principle}, 
we first note that 
\begin{equation}\label{eqn:lde:cp:sum_Ml=1}
    \sum_{l\in \Z} M_l(t) = 1, \qquad t\geq 0.
\end{equation}
Indeed, if $h^0 \equiv 1$, then by uniqueness we must have $h(t) = 1$ for all $t>0$. 
Our next task is to extract more detailed information on the spatial distribution of the `mass' of $M$. In particular, we show that the bulk of this mass is contained in a region that is $O(\sqrt{t})$ wide. By combining our estimates with \eqref{eqn:lde:cp:sum_Ml=1}, the negative components of $M$ can be controlled asymptotically. 

\begin{figure}[t]
    \centering
    \includegraphics[width = \linewidth]{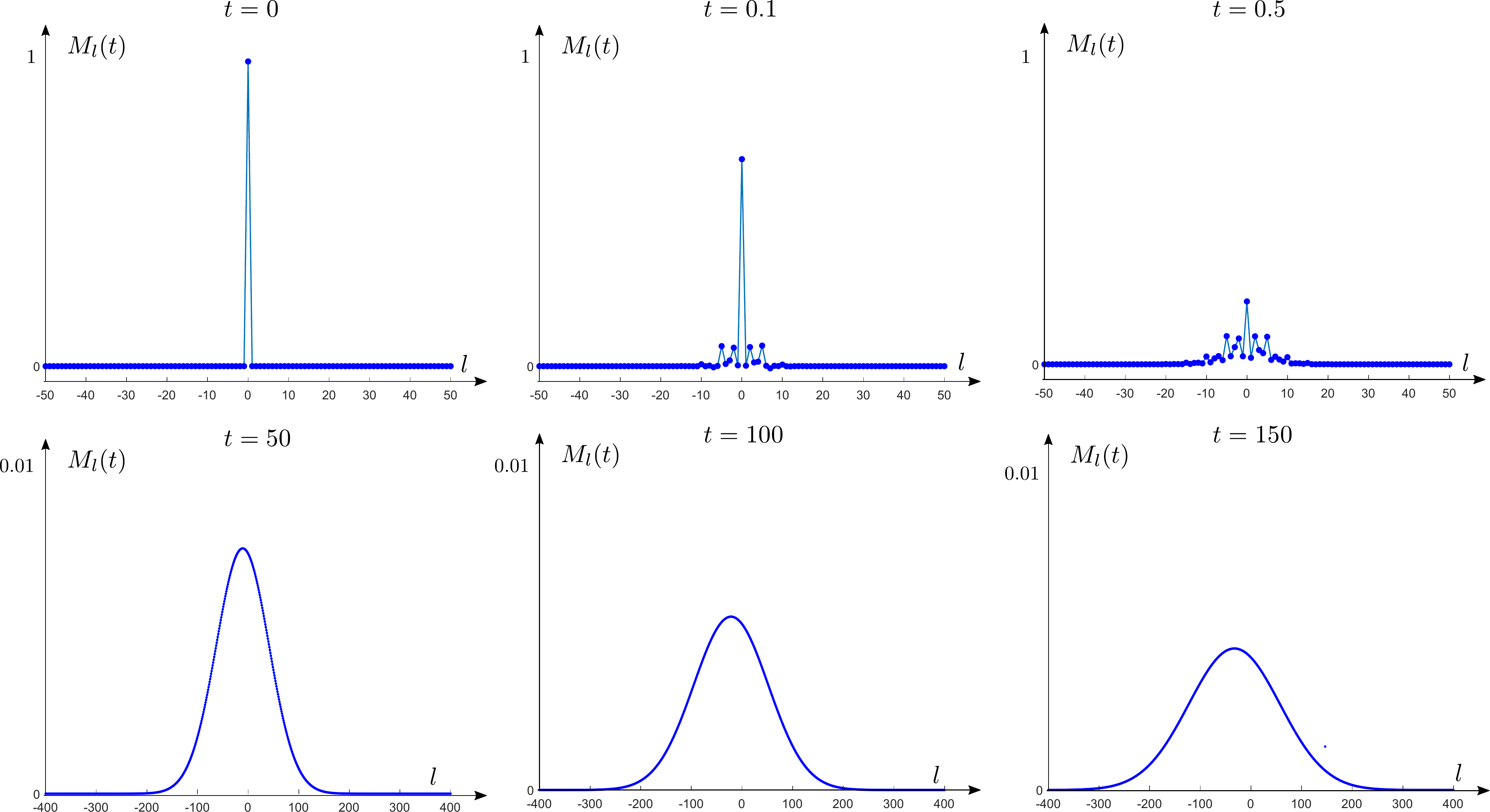}
    \caption{These six graphs represent
    the time evolution of the Green's function $M_l(t)$, 
    which we computed numerically
    by applying~\eqref{eqn:lde:fundamental solution}
    to the coefficients $(a_k)_{k=-10}^{10}$ 
    appearing in Table~\ref{tab:numerics:a_k}.  
    Observe the negative values for $M_l(t)$ 
    that are clearly visible for $t = 0.1$,
    together with the leftward movement of the 
    `center of mass', which travels at the speed
    $-a\cdot\mu =-0.22$.
    }
    \label{fig:bessel}
\end{figure}
\begin{lemma}[{see \S\ref{subsec:lde:core}}]
\label{lemma:lde:sum1}
Consider the setting of Theorem~\ref{thm:lde:n-th-diff-theta}
and pick positive constants $\kappa$ and $K_*$. Then there exist a time $T=T(\kappa, K_*)$ such that for all $t\geq T$ we have
\begin{equation*}
    \sum_{|l/t + a\cdot\mu| \leq \frac{K_*}{\sqrt{t}}} |M_l(t)| \leq 1+\kappa .
\end{equation*}
\end{lemma}

\begin{lemma}[{see \S\ref{subsec:lde:core}}] \label{lemma:lde:sum_2}
Consider the setting of Theorem~\ref{thm:lde:n-th-diff-theta}
and pick $\kappa>0$.  Then there exists a constant $K_*>0$  so that for all  $K\geq K_*$ and $t\geq 1$ we have the bound
\begin{equation*}
    \sum_{\frac{K}{\sqrt{t}} \leq |l/t + a\cdot\mu| \leq K } |M_l(t)| \leq \kappa.
\end{equation*}
\end{lemma}

In our next result we show that the kernel $M(t)$ behaves similarly to the Gaussian kernel. In particular, for the continuous kernel we have $$\dfrac{1}{\sqrt{t}}\int_\R e^{-x^2/t} \frac{|x|} {\sqrt{t}} \sim O(1).$$
To establish the equivalent estimate for the discrete kernel $M(t)$, we have to take into account that the kernel is not symmetric anymore, but that the center of mass `travels' in time with speed $-(a\cdot  \mu)t $ (see Figure~\ref{fig:bessel}). 

\begin{lemma}[{see~\S\ref{subsec:lde:core}}]
\label{lemma:lde:decay_sum_Ml_times_ln}
Consider the setting of Theorem~\ref{thm:lde:n-th-diff-theta}. There exists a constant $C$ such that for every $t>0$ we have
\begin{equation*}
    \sum_{ l\in \Z} |M_l(t) | \frac{|l + (a\cdot\mu)t|}{\sqrt{t}} \leq C .
\end{equation*}
\end{lemma}

\begin{proof}[Proof of Proposition~\ref{prop:lde:comparison_principle}]
We provide the proof only for \textit{(i)}, noting that item \textit{(ii)} can  be derived analogously.  
Upon introducing the shorthand $$x_l^t = \frac{l}{t} + a\cdot\mu,$$ we use 
 Lemmas~\ref{lemma:lde:l_1_noncompact_interval}, \ref{lemma:lde:sum1} and~\ref{lemma:lde:sum_2} to find constants $T$ and $K_*$ so that for all $t\geq T$ we have 
\begin{equation}\label{eqn:lde:cp:init_est}
    \sum_{|x_l^t| \leq \frac{K_*}{\sqrt{t}}} |M_l(t)| \leq 1+\kappa , \qquad  \qquad
    \sum_{|x_l^t| \geq \frac{K_*}{\sqrt{t}}}|M_l(t)| \leq \kappa.
\end{equation}
Combining these inequalities with~\eqref{eqn:lde:cp:sum_Ml=1} we arrive at the bound
\begin{align*}
     \sum_{|x_l^t| \leq \frac{K_*}{\sqrt{t}}} M_l(t) &= 1 -   \sum_{|x_l^t| > \frac{K_*}{\sqrt{t}}} M_l(t) \geq  1 -  \sum_{|x_l^t| > \frac{K_*}{\sqrt{t}}} |M_l(t)|  \geq 1 - \kappa.
 \end{align*}
Employing Lemma~\ref{lemma:lde:sum1} again, we conclude that
\begin{equation}
   -2\kappa \leq  \sum_{|x_l^t| \leq \frac{K_*}{\sqrt{t}}}   M_l(t) - |M_l(t)|   \leq 0,
\end{equation}
from which we obtain the lower bounds
\begin{align*}
    h_k(t) &= \sum_{l\in \Z} M_l(t) h^0_{l-k} 
      = \sum_{|x_l^t| \leq \frac{K_*}{\sqrt{t}}} |M_l(t)|   h^0_{l-k} + \sum_{|x_l^t| \leq \frac{K_*}{\sqrt{t}}} \left( M_l(t) - |M_l(t)| \right)  h^0_{l-k} + \sum_{|x_l^t| > \frac{K}{\sqrt{t}}} M_l(t) h^0_{l-k}
    \\[0.3cm]
     & \geq (1-{\kappa})    \inf_{j\in \Z} h_j^0 - {3\kappa} \norm{h_0}_\infty \\[0.3cm]
    &  \geq \inf_{j\in \Z} h_j^0 - {4\kappa} \norm{h_0}_\infty. 
\end{align*}
The estimate~\eqref{eqn:lde:cp:large_time} can now readily be derived by adjusting the constant $\kappa$ chosen in~\eqref{eqn:lde:cp:init_est}. 

In order to establish~\eqref{eqn:lde:cp:flat_diff} we first compute
\begin{align*}
    h_k(t) = \sum_{l\in \Z} M_l(t) h_{k-l}^0 = \sum_{l\in \Z} M_l(t) h_k^0 + \sum_{l\in \Z} M_l(t) (h^0_{k-l} - h_k^0) \geq h_k^0 - \norm{\partial h^0}_{\ell^\infty}\sum_{ l\in \Z} |M_l(t) | |l |.
\end{align*}
Using Lemma~\ref{lemma:lde:decay_sum_Ml_times_ln}, we can estimate
\begin{equation*}
    \sum_{ l\in \Z} |M_l(t) | |l | \leq C\sqrt{t} + t |a\cdot \mu|  \norm{M(t)}_{\ell^1}.
\end{equation*}
Lemmas~\ref{lemma:lde:l_1_noncompact_interval} and~\ref{lemma:lde:l_1_compact_interval} ensure that $\norm{M(t)}_{\ell^1(\Z)}$ is
uniformly bounded. We can therefore find $C = C(T)$ such that $\max_{t\in [0,T]}\sum_{ l\in \Z} |M_l(t) | |l | \leq C $, which leads to the desired bound~\eqref{eqn:lde:cp:flat_diff}.

\end{proof}

\subsection{Contour deformation}\label{subsec:contour}

The main difficulty towards proving  Lemmas~\ref{lemma:lde:l_1_noncompact_interval} and~\ref{lemma:lde:l_1_compact_interval} lies in the fact that  $[\partial^n M(t)]_l$ depends on the variable $l$ only through the expression $e^{i\omega l}$. By simply taking the absolute value of the integrand in expressions such as~\eqref{eqn:lde:fundamental solution}, we therefore lose all information on the decay coming from the $l$-variable. In order to overcome this issue, we pick $\epsilon\in \R$ and denote by $\Gamma_\epsilon$ rectangle consisting of paths $$\gamma_1 = [-\pi, \pi], \ \gamma_2 = [\pi, \pi+i\epsilon], \ \gamma_3 = [\pi+i\epsilon, -\pi+i\epsilon], \ \gamma_4 = [-\pi+i\epsilon, -\pi].$$ Due to the fact that $f$ and $g$ are $2\pi$-periodic in the real variable, we have 
$$\int_{\gamma_2} e^{il\omega} e^{tf(\omega) + it p(\omega)} \ d\omega  = - \int_{\gamma_4} e^{il\omega} e^{tf(\omega) + it p(\omega)} \ d\omega. $$
Therefore, we obtain
\begin{align*}
 0 =  \oint_{\Gamma_\epsilon} e^{il\omega} e^{tf(\omega) + it p(\omega)} \ d\omega = \int_{\gamma_1} e^{il\omega} e^{tf(\omega) + it p(\omega)} \ d\omega +  \int_{\gamma_3} e^{il\omega} e^{tf(\omega) + it p(\omega)} \ d\omega.
\end{align*}
Recalling~\eqref{eqn:lde:fundamental solution}, this allows us to compute
\begin{equation}\label{eqn:shift_contour_Ml}
\begin{aligned}
    M_l(t) &= \dfrac{1}{2\pi} \int_{-\pi+i\epsilon}^{\pi+i\epsilon} e^{il\omega} e^{tf(\omega) + it p(\omega)} \ d\omega
    =  \dfrac{1}{2\pi} \int_{-\pi}^{\pi} e^{i l(\omega +i\epsilon) }e^{tf(\omega+i\epsilon) + it p(\omega+i\epsilon)} \ d\omega.
\end{aligned}
\end{equation}
Writing $z = x+iy$ with $x, y\in \R$, we recall the formulas
\begin{align*}
    \cos{z} &= \cos{x}\cosh{y} - i\sin{x}\sinh{y}, & &  \sin{z} = \sin{x}\cosh{y} + i\cos{x}\sinh{y}
\end{align*}
to obtain
\begin{align} \label{eqn:lde:def_of_Ml2}
    M_l(t) = \dfrac{1}{2\pi} e^{-\epsilon l}\int_{-\pi}^{\pi} e^{i l \omega  }e^{t f(\omega, \epsilon)  + it
    p(\omega, \epsilon)}\ d\omega,
\end{align}
where the functions $f, p:[-\pi, \pi]\times\R \to \R$ are
now defined by
\begin{align*}
    f(\omega, \epsilon) = \sum_{k=-N}^Na_k (\cos{(\mu_k\omega)e^{-\mu_k\epsilon}} -1 ),  \qquad p(\omega, \epsilon) &=  \sum_{k=-N}^Na_k \sin(\mu_k\omega)e^{-\mu_k\epsilon},
\end{align*}
extending the definitions \eqref{eq:lde:bnds:def:f}
and \eqref{eq:lde:bnds:def:p}.

The main strategy is to choose  suitable values for $\epsilon$ in order to isolate the relevant decay
rates in various $(l,t)$ regimes. 
Indeed, the representation
\eqref{eqn:lde:def_of_Ml2} does retain
sufficient spatial information for our purposes
when applying crude estimates to the integrands.
To appreciate this,
we recall that the Fourier symbol of the difference operator $\partial^{(1)}$ is $e^{i\omega}-1$
and introduce the real-valued expressions
\begin{align}
    P_l(\epsilon, t, n) &= |1-e^{-\epsilon}|^{n} e^{-\epsilon l} \int_{-\pi}^{\pi } e^{t f(\epsilon, \omega) }\ d\omega,  \label{eqn:lde:def_of_P}
    \\ 
    R_l(\epsilon, t, n) &=e^{-\frac{\epsilon n}{2}}e^{-\epsilon l}\int_{-\pi}^{\pi } |\omega|^n e^{t f(\epsilon, \omega) }\ d\omega \label{eqn:lde:def_of_R}.
\end{align}
The result below shows %
that their sum can be used to extract the desired bounds on $\partial^{n} M(t)$.
In particular, the problem of estimating the $\ell^1$-norm of the sequence $\partial^n M(t)$ is reduced to finding the corresponding bounds
for the $\ell^1$-norm of the sequences $P(\epsilon, t, n)$ and $R(\epsilon, t, n)$. 

\begin{lemma}\label{lemma:lde:M_bounded_P_R}
Consider the setting of Theorem~\ref{thm:lde:n-th-diff-theta}. Then for every $l\in \Z$, $\epsilon\in\R$ and $t\geq 0$ we have
\begin{equation}\label{eqn:lde:partial_M_P_R}
  \left|[\partial^{n} M(t)]_l\right| \leq   \dfrac{2^{n/2}}{4\pi}\big(P_l(\epsilon, t, n) + R_l(\epsilon, t, n)\big) .
\end{equation}
\end{lemma}

\begin{proof}
Taking $n$-th differences in~\eqref{eqn:lde:def_of_Ml2} we obtain the expression
\begin{align*}
    [\partial^{n} M(t)]_l = \dfrac{1}{2\pi} e^{-\epsilon l} \int_{-\pi}^{\pi } e^{i\omega l}\left( e^{-\epsilon} e^{i\omega} - 1 \right)^n e^{tf(\omega, \epsilon) + it p(\omega, \epsilon)} d\omega,
\end{align*}
which leads to the bound
\begin{equation}\label{eqn:lde:estimate_partial_n}
    \left|[\partial^{n} M(t)]_l\right| \leq \dfrac{1}{2\pi} e^{-\epsilon l} \int_{-\pi}^{\pi } \left| e^{-\epsilon} e^{i\omega} - 1 \right|^n e^{tf(\omega, \epsilon) } \ d\omega.
\end{equation}
Next, we compute
\begin{align*}
    |e^{-\epsilon} e^{i\omega} - 1 | & = \sqrt{e^{-2\epsilon} - 2e^{-\epsilon}\cos{\omega }+1} \\
    & = \sqrt{\left(1- e^{-\epsilon}\right)^2 + 2 e^{-\epsilon}(1-\cos{\omega})}.
\end{align*}
Employing the standard inequality $(a+b)^p\leq 2^{p-1} (a^p + b^p)$ for non-negative real numbers $a$, $b$ and $p$ together with ${1-\cos{\omega}}\leq \frac{|\omega|^2}{2}$, we obtain the bound 
\begin{align*}
     |e^{-\epsilon} e^{i\omega} - 1 |^n \leq  2^{\frac{n}{2} -1 } \left( |1-e^{-\epsilon}|^n  + e^{-\frac{n\epsilon}{2}} |\omega|^n\right),
\end{align*}
from which~\eqref{eqn:lde:partial_M_P_R} readily follows. 
\end{proof}

At times, it is convenient to split the exponents in~\eqref{eqn:lde:def_of_P}-\eqref{eqn:lde:def_of_R} in a slightly different fashion. To this end, we introduce
 two auxiliary functions $g:\Z\times\R\times(0, \infty) \to \R$ and $q:\R\times (0, \infty)$ defined by 
\begin{align}
    g(l,\epsilon, t) &:= -\epsilon l  + t \sum_{k=-N}^Na_k (e^{-\mu_k \epsilon} - 1), 
\label{eqn:lde:def_of_g}
\\
    q(\epsilon, \omega ) &:= {\sum_{k=-N}^N a_k (\cos{(\mu_k \omega)} -1 )e^{-\mu_k\epsilon}}. \label{eqn:lde:def_of_h}
\end{align}
This allows us to rewrite~\eqref{eqn:lde:def_of_P} and \eqref{eqn:lde:def_of_R} in the form
\begin{align}
    P_l(\epsilon, t, n) &= |1-e^{-\epsilon}|^{n} e^{g(l,\epsilon,t)} \int_{-\pi}^\pi e^{t q(\epsilon, \omega)}\,  d\omega  \label{eqn:lde:def_P2} ,
    \\ 
    R_l(\epsilon, t, n) &= e^{-\frac{\epsilon n}{2}} e^{g(l,\epsilon, t)} \int_{-\pi}^\pi |\omega|^n e^{t q(\epsilon, \omega)}\,  d\omega . \label{eqn:lde:def_R2}
\end{align}
 Note that $g$ vanishes for $\epsilon = 0$, while $q$ reduces to $f$. In the reminder of this subsection we provide several preliminary bounds for $q$ and the integral expressions above. 

\begin{lemma}\label{lemma:lde:function_s}
Pick $n\in \N_0$, introduce  two positive constants \begin{equation}\label{eqn:lde:C1_C2}
    C_1  = C_1(n) =  \int_{-\infty}^\infty u^n e^{-u^2}\, du , \qquad \qquad
    C_2  = C_2(n) =   n^{n/2} e^{-n/2} 2^{2 -n/2}
\end{equation}
and consider the functions  
\begin{equation*}
    s_{n,\nu}(x, t) = |x|^n e^{-\nu t x^2 },
\end{equation*} 
together with the sequences
\begin{equation*}
    x_k^t = \frac{k}{t} + b
\end{equation*} 
for any $b \in \R$.
Then the following claims hold.
\begin{enumerate}[(i)]
    \item\label{item:lde:lemma_s:int_s} For any $t>0$ and $\nu>0$ we have
    $$\int_{-\infty}^\infty s_{n,\nu}(x, t)dx \leq  C_1  t^{-\frac{n+1}{2}} \nu^{-\frac{n+1}{2}}.$$
    \item\label{item:lde:lemma_s:sum_s} 
    For any $t>0$ and $\nu>0$ the series $\sum_{k\in \Z } s_n(x_k^t, t)$ converges and we have the upper bound
    \begin{equation}\label{eqn:lde:sum_s}
        \sum_{k\in \Z} s_{n, \nu} (x_k^t, t) \leq  C_1  t^{-\frac{n-1}{2}} \nu^{-\frac{n+1}{2}} + C_2 t^{-\frac{n}{2}} \nu^{-\frac{n}{2}}.
    \end{equation}
  
   \item\label{item:lde:lemma_s:tail_bound}  For any $t > 0$, $\nu>0$ and $K_0 > 0$ we have the tail bound
   \begin{equation}\label{eqn:tail:bound}
         \sum_{|x_k^t| \ge K_0} s_{0, \nu} (x_k^t, t) \leq 
         \big( 2 + \frac{2\sqrt{t}}{\sqrt{\nu}} \big)         e^{-\nu t K_0^2}.
    \end{equation}
\end{enumerate}
\end{lemma}
\begin{proof}
Item \textit{(\ref{item:lde:lemma_s:int_s})} follows directly after substituting $u = \sqrt{\nu t} x $ and observing that for every $n\geq0$ we have $C_1(n) < \infty$. 
To prove item \textit{(\ref{item:lde:lemma_s:sum_s})}, we first note that 
the function $x\mapsto s_{n,\nu}(x, t)$ is symmetric around $0$,  increasing on the interval $[0, \frac{\sqrt{n}}{\sqrt{2\nu t}}]$ and decreasing on  $[ \frac{\sqrt{n}}{\sqrt{2\nu t}}, \infty)$. 
Choosing integers $N_1$,  $N_2$ and $M$ in such a way that 
$$  -\frac{\sqrt{n}}{\sqrt{2\nu t}}\in [x_{N_1}^t, x_{N_1+1}^t], \qquad \qquad    \frac{\sqrt{n}}{\sqrt{2\nu t}}\in [x_{N_2}^t, x_{N_2+1}^t], \qquad \qquad  0\in [x_M^t, x_{M+1}^t],$$ 
we can hence write
\begin{align*}
    \sum_{k\in \Z} s_{n,\nu}(x_k, t ) 
    & = \sum_{k=-\infty}^{N_1-1} \frac{1}{x_{k+1}^t - x_k^t}\int_{x_k^t}^{x_{k+1}^t}s_{n,\nu}(x_k^t, t )\, dx
     + \sum_{k=N_1+2}^{M} \frac{1}{x_{k}^t - x_{k-1}^t}\int_{x_{k-1}^t}^{x_{k}^t}s_{n,\nu}(x_k^t, t )\, dx \\[0.3cm]
    &\indent  + \sum_{k=M+1}^{N_2-1} \frac{1}{x_{k+1}^t - x_k^t}\int_{x_k^t}^{x_{k+1}^t}s_{n,\nu}(x_k^t, t )\, dx
    + \sum_{k=N_2+2}^\infty \frac{1}{x_{k}^t - x_{k-1}^t}\int_{x_{k-1}^t}^{x_{k}^t}s_{n,\nu}(x_k^t, t )\, dx 
    \\[0.3cm]
    &\indent  + s_{n,\nu}(x_{N_1}^t, t) +s_{n, \nu}(x_{N_1+1}^t, t)  + s_{n,\nu}(x_{N_2}^t, t) +s_{n, \nu}(x_{N_2+1}^t, t).
\end{align*}
Noting that $x_{k+1}^t - x_k^t = 1/t$ and recalling  \textit{(\ref{item:lde:lemma_s:int_s})}, we find
\begin{align*}
    \sum_{k\in \Z} s_{n,\nu}(x_k, t ) 
    & \leq  t\int_{-\infty}^{x_{N_1}^t} s_{n,\nu}(x, t )\, dx
     + t \int_{k=x_{N_2}^t}^{\infty} s_{n,\nu}(x, t )\, dx + 4 s_{n,\nu} (\sqrt{n}/\sqrt{2\nu t}) \\[0.3cm]
     & \leq   C_1  t^{-\frac{n-1}{2}} \nu^{-\frac{n+1}{2}} + C_2 t^{-\frac{n}{2}} \nu^{-\frac{n}{2}} .
\end{align*}
This proves~\eqref{eqn:lde:sum_s}, as desired.

 For \textit{(\ref{item:lde:lemma_s:tail_bound})}, we first
choose integers $N_1$
and $N_2$  in such a way that 
$$  -K_0 \in [x_{N_1}^t, x_{N_1 +1}^t], \qquad \qquad    K_0 \in [x_{N_2-1}^t, x_{N_2}^t].$$
Using the fact
that $x\mapsto s_{0,\nu}(x,t)$ is even and decreasing on $[0, \infty)$, we compute
\begin{equation*}
\begin{aligned}
    \sum_{|x_k^t| \ge K_0}^N s_{0,\nu}(x_k^t, t) 
    &\leq s_{0,\nu}(x^t_{N_1})
    + s_{0,\nu}(x^t_{N_2})
    +
    2t \int_{K_0}^{\infty} e^{-\nu t x^2}\, dx
\\
   &\le  2 s_{0,\nu}(K_0) + 2 \frac{\sqrt{t}}{\sqrt{\nu}} \int_{\sqrt{\nu t}K_0}^\infty e^{-u^2} \, d u.
\end{aligned}
\end{equation*}
The desired estimate
now follows from
the Chernoff bound,
which states that $\mathrm{erfc}(x) \leq e^{-x^2}$ holds for all $x>0$ (see \cite{chang2011chernoff, chiani2002improved}).
\end{proof}

\begin{lemma}\label{lemma:lde:q_bounds}
Consider the setting of Theorem~\ref{thm:lde:n-th-diff-theta} and pick $0<\delta<\frac{\Lambda}{2}$.  Then there exist positive constants $\overline{\epsilon}$, $\overline{m}$ and $\overline{\kappa}$ such that the following statements hold. 
\begin{enumerate}[(i)]
    \item \label{item:lde:q_bounds:general}For all $\epsilon\in [-\overline{\epsilon}, \overline{\epsilon}]$ we have
    \begin{align}
        q(\epsilon, \omega )& \leq -\left(\frac{\Lambda}{2} - \delta\right)\omega^2, & & \omega \in [-\overline{\kappa}, \overline{\kappa}], & &  \label{eqn:lde:bounds_q_quadratic} \\
        q(\epsilon, \omega) & \leq -\overline{m},  & &   \omega\in [-\pi, -\overline{\kappa}]\cup[\overline{\kappa}, \pi]. & & & & \label{eqn:lde:bounds_q_uniform}
    \end{align}

\item\label{item:lde:q_bounds:integral}  
Pick $n\in \N_0$ and recall the constant $C_1 = C_1(n)$ from~\eqref{eqn:lde:C1_C2}. Then the estimate 
\begin{equation*}
    \int_{-\pi}^\pi |\omega|^n e^{t q(\epsilon, \omega)} d\omega  \leq C_1 t^{-\frac{n+1}{2}} \left(\frac{\Lambda -2\delta}{2}\right)^{-\frac{n+1}{2}}  + \frac{2\pi^{n+1}}{n+1}e^{-\overline{m}t}
\end{equation*}
holds for all $t>0$ and $\epsilon\in [-\overline{\epsilon}, \overline{\epsilon}]$.
In particular, for $n=0$ we have
\begin{equation}\label{eqn:lde:bound_on_q_n=0}
    \int_{-\pi}^\pi e^{t q(\epsilon, \omega)} d\omega  \leq   \frac{{\sqrt{2 \pi}}}{\sqrt{t(\Lambda - 2\delta)}}   + 2\pi e^{-\overline{m}t}. 
\end{equation}
\end{enumerate}
\end{lemma}

\begin{proof}
To prove item~\textit{(\ref{item:lde:q_bounds:general})}, we start by defining an auxiliary function $\overline{q} (\epsilon, \omega) = q(\epsilon, \omega) + \left(\frac{\Lambda}{2} - \delta\right)\omega^2$,
which satisfies
\begin{equation}\label{eqn:lde:h_estimates_eqn1}
 \overline{q}(\epsilon, 0) = \overline{q}_\omega(\epsilon, 0) = 0,
 \qquad  \qquad
 \overline{q}_{\omega\omega}(\epsilon, 0) = -\sum_{k=-N}^N a_k \mu_k^2 e^{-\mu_k \epsilon}+ \Lambda - 2\delta.
\end{equation}
Recalling the definition \eqref{eqn:lde:def_of_lambda}
and exploiting continuity,
there exists $\overline{\epsilon}>0$ such that 
\begin{equation*}
    -\sum_{k=-N}^N a_k \mu_k^2 e^{-\mu_k \epsilon} < - \Lambda + \delta, \qquad \text{for all} \ \epsilon\in [-\overline{\epsilon}, \overline{\epsilon}],
\end{equation*} and consequently 
$$\overline{q}_{\omega\omega}( \epsilon, 0) \leq - \delta < 0 , \qquad \text{for all}  \ \epsilon \in [-\overline{\epsilon}, \overline{\epsilon}]. $$
Combining this bound with~\eqref{eqn:lde:h_estimates_eqn1} allows us to find $\overline{\kappa}>0$ such that $\overline{q}(\epsilon, \omega) \leq 0 $ for all $\omega \in [-\overline{\kappa}, \overline{\kappa}] $ and $\epsilon\in [-\overline{\epsilon}, \overline{\epsilon}]$ which proves~\eqref{eqn:lde:bounds_q_quadratic}. 
To establish~\eqref{eqn:lde:bounds_q_uniform}, we note that assumption \textit{(h$\alpha$)} implies that there exists a constant $\overline{m}>0$ such that
\begin{equation*}
    q(0, \omega) \leq -2\overline{m} , \ \text{for}\ \omega\in [-\pi, -\overline{\kappa}]\cup [\overline{\kappa}, \pi].
\end{equation*}
Therefore, by possibly reducing $\overline{\epsilon}$ we can conclude that for $ \epsilon\in [-\overline{\epsilon},\overline{\epsilon} ] $ and $\omega\in  [-\pi, -\overline{\kappa}]\cup [\overline{\kappa}, \pi]$ we have $
     q(\epsilon, \omega) \leq -\overline{m} $,
as desired.

To prove item~\textit{(\ref{item:lde:q_bounds:integral})} we use the bounds from (i)
to compute
\begin{align*}
    \int_{-\pi}^\pi |\omega|^n e^{t q(\epsilon, \omega)} d\omega \leq \int_{-\overline{\kappa}}^{\overline{\kappa}} |\omega|^n e^{-\frac{t}{2}(\Lambda - 2\delta)\omega^2 }\, d\omega + 2\int_{\overline{\kappa}}^\pi \omega^n e^{-\overline{m} t} \, d\omega. 
\end{align*}
We may now employ item~\textit{(\ref{item:lde:lemma_s:int_s})} from Lemma~\ref{lemma:lde:function_s} with $\nu = \frac{\Lambda - 2\delta}{2}$ and explicitly evaluate the second integral to obtain the desired bound.
\end{proof}

\begin{cor}\label{lemma:lde:est_for_int_h}
 Consider the setting of Theorem~\ref{thm:lde:n-th-diff-theta} and pick $n\in N_0$.  Then there exist constants $\overline{\epsilon}>0$ and  $C = C(n)>0$  such that for all $t>0$ and $\epsilon \in (-\overline{\epsilon}, \overline{\epsilon})$ we have the estimate
\begin{equation}\label{eqn:lde:est_for_int_e^tq}
    \int_{-\pi}^\pi |\omega|^n e^{t q(\epsilon, \omega)}\, d\omega \leq C\min\big\{1 ,  t^{-\frac{n+1}{2}} \big\}.
\end{equation}
\end{cor}

\begin{proof}
For $0 < t< 1$ the uniform bound follows from item \textit{(\ref{item:lde:q_bounds:general})} of Lemma~\ref{lemma:lde:q_bounds}, which implies that $q(\epsilon, \omega) \le 0$.
On the other hand, we may apply item \textit{(\ref{item:lde:q_bounds:integral})} from the same result with $\delta = \frac{\Lambda}{4}$ to find
\begin{align*}
    \int_{-\pi}^\pi |\omega|^n e^{t q(\epsilon, \omega)} d\omega \leq C_1 {2^{{n+1}}} t^{-\frac{n+1}{2}}\Lambda^{-\frac{n+1}{2}}  +  \dfrac{2\pi^{n+1}}{n+1} e^{-\overline{m}t},
\end{align*}
which can be absorbed into~\eqref{eqn:lde:est_for_int_e^tq} on account of the fast decay of the exponential. 
\end{proof}

\subsection{Global and outer bounds}\label{subsec:lde:global}
In order to prove the $\ell^1$-decay estimates of the sequence ${\partial^n M(t)}$ we first establish $\ell^\infty$-bounds,
which decay at a rate
that is faster by a factor $1/\sqrt{t}$. In particular, we have the following result. 

\begin{lemma}\label{lemma:lde:l_infty_bounds}
Assume that (h$\alpha$) is satisfied and pick $n\in\N_0$. Then there exists $C = C(n)>0$ so that the $n$-th difference of the sequence $M(t)$ satisfies the bound
\begin{equation}\label{eqn:lde:l_infty_est}
      ||\partial^n M(t)||_{\ell^\infty} \leq C\min \left\{1, t^{-\frac{n+1}{2}} \right\}.
\end{equation}
\end{lemma}

\begin{proof}
Picking $\epsilon =0 $, we see that $P$ and $g$ vanish,  
which in view of \eqref{eqn:lde:partial_M_P_R}
and Corollary~\ref{lemma:lde:est_for_int_h} implies the desired bound
\begin{equation}
    |[\partial^n M(t)]_l|
    \le \frac{2^{n/2}}{2 \pi}
    \int_{-\pi}^\pi | \omega|^n  e^{t q(0,\omega)} \, d \omega
    \le \frac{2^{n/2} C}{4 \pi} \min \big\{1 ,  t^{-\frac{n+1}{2}} \big\}.
\end{equation}
\end{proof}

\begin{proof}[Proof of Lemma~\ref{lemma:lde:l_1_noncompact_interval}]
Consider the constant
$\overline{\epsilon} > 0$
introduced in Lemma~\ref{lemma:lde:q_bounds}, which guarantees that $q(\overline{\epsilon}, \cdot) \le 0$
and $q(-\overline{\epsilon}, \cdot) \le 0$. 
In view of the bound~\eqref{eqn:lde:partial_M_P_R} 
and 
the representation
\eqref{eqn:lde:def_P2}-\eqref{eqn:lde:def_R2},
it suffices 
to show that there exist $K>|a \cdot \mu| $ and $C>0$ so that
\begin{equation}
\label{eq:lin:bnds:sum:exp:g}
     \sum_{l \ge ( K - a\cdot \mu) t} (1 + l ) e^{g(l, \overline{\epsilon}, t)}
     < C e^{-t},
     \qquad \qquad 
     \sum_{l \le - ( K - a\cdot \mu) t} (1 + |l|) e^{g(l, -\overline{\epsilon}, t)}
     \le C e^{-t}.
\end{equation}
Focusing on the former,
we pick 
\begin{equation}
    K >  1 + |a \cdot \mu |
    +2 \overline{\epsilon}^{-1} \big[ 2 + \sum_{k=-N}^N |a_k| \big] ,
\end{equation} 
which allows us to compute
\begin{equation}
\begin{array}{lcl}
       g(l,\overline{\epsilon}, t)
       &=& -\frac{1}{2}\overline{\epsilon} l 
       -\frac{1}{2}\overline{\epsilon} l + t \sum_{k=-N}^Na_k (e^{-\mu_k \overline{\epsilon}} - 1),
\\[0.2cm]
& \le & 
-\frac{1}{2}\overline{\epsilon} l  - t\Big[ \frac{1}{2} \overline{\epsilon} (K - a \cdot \mu)
- \sum_{k=-N}^Na_k (e^{-\mu_k \overline{\epsilon}} - 1) \Big]
\\[0.2cm]
& \le &
-\frac{1}{2}\overline{\epsilon} l
\end{array}
\end{equation}
for all $l \ge (K - a \cdot \mu)t$ and $t > 0$.
Using $t e^{-t} \le 1$, this in turn yields
\begin{equation}
\begin{array}{lcl}
     \sum_{l \ge ( K - a\cdot \mu) t}(1 + l) e^{g(l, \overline{\epsilon}, t)}
     &\le& 2 \sum_{l \ge ( K - a\cdot \mu) t} l e^{-\frac{1}{2}\overline{\epsilon} l}
     \\[0.2cm]
& \le &  2 (1 - e^{-\frac{1}{2}  \overline{\epsilon}})^{-2} 
 \big[ ( K - a\cdot \mu) t  + 1 \big]
 e^{- \frac{1}{2} ( K - a\cdot \mu)   \overline{\epsilon}  t } 
\\[0.2cm]
& \le &  2 (1 - e^{-\frac{1}{2}  \overline{\epsilon}})^{-2} 
  ( K - a\cdot \mu) ( t  + 1) 
 e^{- 2  t } 
\\[0.2cm]
& \le &
4 (1 - e^{-\frac{1}{2}  \overline{\epsilon}})^{-2} 
 ( K - a\cdot \mu) 
 e^{- t } . 
 \end{array}
\end{equation}
Here we used the bound 
\begin{align*}
    \sum_{l=l_*}^\infty l r^l = r \dfrac{d}{dr} \left(\sum_{l = l_*}^\infty r^l \right) =  r \dfrac{d}{dr} \left(\frac{r^{l_*}}{1-r} \right) =\frac{l_* r^{l_*}}{1-r} + \frac{r^{l_*+1}}{(1-r)^2} \leq \frac{(l_*+1)r^{l_*}}{(1-r)^2}
\end{align*}
with $r = e^{-\frac{1}{2}\overline{\epsilon}}$ and $l_* = \lfloor(K-a\cdot \mu)t  \rfloor $. The second
inequality in \eqref{eq:lin:bnds:sum:exp:g}
can be obtained in a similar fashion. 
\end{proof}

\subsection{Core bounds}\label{subsec:lde:core}

In this subsection we prove Lemmas~\ref{lemma:lde:l_1_compact_interval},~\ref{lemma:lde:sum1} and~\ref{lemma:lde:sum_2}, which all deal with $\ell^1$-bounds on compact intervals. Recalling the characterization~\eqref{eqn:lde:def_P2}-\eqref{eqn:lde:def_R2}, we start by providing useful bounds for the exponent $g$ when $|l/t + a\cdot\mu|$ is bounded.
To obtain these estimates,  we show that for
compact sets of $\epsilon$ 
the function $g$ can be controlled by an upwards parabola in $\epsilon$. 
\begin{lemma}\label{lemma:lde:est_for_g}
 Consider the setting of Theorem~\ref{thm:lde:n-th-diff-theta} and pick constants $\overline{\epsilon}>0$ and $K>0$. 
 Let $\delta >0$ be any number that satisfies \begin{equation}\label{eqn:lde:def_of_delta_g}
    \delta \geq \max\left\{\dfrac{K}{2\overline{\epsilon}} - \dfrac{\Lambda}{2},  \dfrac{\overline{\epsilon}}{3}\sum_{k=-N}^N |a_k \mu_k^3| e^{|\mu_k|\overline{\epsilon}} \right\}
 \end{equation} 
 and write 
 \begin{equation}\label{eqn:lde:def_of_nu}
    \nu_\delta = \dfrac{1}{2(\Lambda + 2\delta)}.
\end{equation}
 Then for every  pair $(l,t)\in \Z\times(0,\infty)$ with $|l/t + \alpha\cdot \mu |\leq K$, the choice 
 \begin{equation}\label{eqn:lde:epsilon*}
     \epsilon^* = \epsilon^*(l,t)= 2\nu_\delta \left(\frac{l}{t} + a\cdot\mu \right)\in [-\overline{\epsilon}, \overline{\epsilon}], 
\end{equation}
satisfies the inequality
\begin{equation}\label{eqn:lde:minimum_of_g}
    g(\epsilon^*, l, t) \leq -\nu_\delta t\left(\dfrac{l}{t} + a \cdot \mu \right)^2.
\end{equation}
\end{lemma}

\begin{proof}
 By expanding  the function $g$ around $\epsilon = 0$, we obtain
\begin{align*}
    g(\epsilon, l, t) = -\epsilon l + \Big(-\epsilon a\cdot\mu + \frac{\Lambda}{2}\epsilon^2  - \dfrac{\epsilon^3}{6}\sum_{k=-N}^N a_k \mu_k^3 e^{-\mu_k \tilde{\epsilon}}  \Big) t
\end{align*}
for some $\tilde{\epsilon}$ with $|\tilde{\epsilon}|\leq |\epsilon|$, which we rewrite as 
\begin{align}\label{eqn:lde:lemma_g:expansion_g}
    g(\epsilon, l, t) = -\epsilon l  + \Big(-\epsilon a\cdot\mu  + (\frac{\Lambda}{2} +\delta)\epsilon^2 \Big) t +\epsilon^2\Big(-\delta - \dfrac{\epsilon}{6}\sum_{k=-N}^N a_k \mu_k^3 e^{-\mu_k \tilde{\epsilon}} \Big)t.
\end{align}
For any $\epsilon\in [-\overline{\epsilon}, \overline{\epsilon}]$, $l\in \Z$, $t>0$  our condition on $\delta$ ensures that
\begin{align*}
    g(\epsilon, l, t) \leq -\epsilon l + \Big(-\epsilon a\cdot\mu  + \big(\frac{\Lambda}{2} + \delta\big)\epsilon^2 \Big) t = \dfrac{t}{4\nu_\delta}\Big((\epsilon - \epsilon^*)^2 - (\epsilon^*)^2\Big),
\end{align*}
since the last term in~\eqref{eqn:lde:lemma_g:expansion_g} is negative. It hence suffices to show that $|\epsilon^*|\leq \overline{\epsilon} $, but that follows directly from our assumption on $\delta$. 
\end{proof}

\begin{proof}[Proof of Lemma~\ref{lemma:lde:l_1_compact_interval}]
In view of Lemma~\ref{lemma:lde:M_bounded_P_R} it suffices to show that 
\begin{equation*}
     \sum_{|l/t + a\cdot\mu|\leq K} P_l(\epsilon^*, t, n) \leq  C t^{-\frac{n}{2}}, \quad  \sum_{|l/t + a\cdot\mu|\leq K} R_l(\epsilon^*, t, n) \leq  C t^{-\frac{n}{2}}, \quad t \geq 1
\end{equation*}
for some constant $C>0$.
Here we make the choice
$\epsilon^* = \epsilon^*(l,t)$
as defined by \eqref{eqn:lde:epsilon*}
in Lemma~\ref{lemma:lde:est_for_g}, using the value
$\overline{\epsilon} > 0$ 
that was introduced in Lemma \ref{lemma:lde:q_bounds},
together with an arbitrary $\delta > 0$ that satisfies~\eqref{eqn:lde:delta_bounds_h}. Without loss, we make the further restriction $\overline{\epsilon} < 1$,
which allows us to write
$|1-e^{-\epsilon^*}|\leq 2 |\epsilon^*|$.
We will provide the proof only for $P$, noting that the estimate for $R$ can  be derived analogously.

The bound \eqref{eqn:lde:minimum_of_g}
allows us to compute
\begin{equation*}
  P_l(\epsilon^*, t, n) \leq 
  2^n |\epsilon^*|^n
  e^{-\nu_{\delta} t \left(l/t + a\cdot \mu\right)^2} \int_{-\pi}^\pi e^{t q(\epsilon^*, \omega)}\, d\omega ,
\end{equation*}
which
in combination with Corollary
\ref{lemma:lde:est_for_int_h}
yields
\begin{align*}
     P_l(\epsilon^*, t, n) &\leq  
     C(n) 2^n  t^{-\frac{1}{2}}|\epsilon^*|^n e^{-\nu_\delta t \left(l/t + a\cdot \mu\right)^2}
\\
& \leq
C(n) 4^n  t^{-\frac{1}{2}}
\nu_{\delta}^n \left(l/t + a\cdot \mu\right)^n e^{-\nu_\delta t \left(l/t + a\cdot \mu\right)^2} .
\end{align*}
Applying
item~\textit{(\ref{item:lde:lemma_s:sum_s})} from Lemma~\ref{lemma:lde:function_s} with $\nu = \nu_{\delta}$
now yields the desired estimate, upon redefining $C$.
\end{proof}

\begin{proof}[Proof of Lemma~\ref{lemma:lde:sum1}]
Our goal is to exploit the representations \eqref{lemma:lde:M_bounded_P_R} for $n=0$ and \eqref{eqn:lde:def_P2}
to obtain the estimate
\begin{equation}
    |M_l(t)| \leq \dfrac{1}{2\pi} e^{g(l, \epsilon^*, t)} \int_{-\pi}^\pi e^{t q(\epsilon^*, \omega)} d\omega,
\end{equation}
where we again use the values
$\epsilon^* = \epsilon^*(l,t)$
defined by \eqref{eqn:lde:epsilon*}
in Lemma~\ref{lemma:lde:est_for_g}, but now
picking $0 < \delta < \Lambda / 2$ to be small enough to ensure that
\begin{equation}\label{eqn:lde:delta_bounds_h}
    \sqrt{\dfrac{\Lambda + 2{\delta}}{\Lambda - 2{\delta}}} \leq 1+ \dfrac{\kappa}{2}
\end{equation}
holds. In order to validate
the condition \eqref{eqn:lde:def_of_delta_g}
with $K:=K_*/\sqrt{T}$
, we pick a sufficiently large $T > 0$ and decrease
the value of $\overline{\epsilon} > 0$ from
Lemma~\ref{lemma:lde:q_bounds} to ensure that 
\begin{equation}\label{eqn:lde:cond1_delta}
\dfrac{K_*}{2 \sqrt{T} \overline{\epsilon}} - \dfrac{\Lambda}{2} < 0 < \delta,
\qquad \qquad 
     \dfrac{\overline{\epsilon}}{3}\sum_{k=-N}^N |a_k \mu_k^3| e^{|\mu_k|\overline{\epsilon}} \leq \delta 
\end{equation} 
both hold.
Combining \eqref{eqn:lde:bound_on_q_n=0}
and \eqref{eqn:lde:minimum_of_g}
and writing 
$x_l^t = l/t+a\cdot\mu$,
we hence obtain
\begin{equation*}
    |M_l(t)| \leq \dfrac{1}{2\pi} e^{ - \nu_{\delta} t (x_l^t)^2}
    \Big[\dfrac{{\sqrt{2\pi}}}{\sqrt{t}\sqrt{\Lambda - 2{\delta}}}  + 2\pi e^{-\overline{m}t}
    \Big]
\end{equation*}
whenever $| x_l^t| \le K_*/\sqrt{t}$ and $t \ge T$.
Applying 
item \textit{(\ref{item:lde:lemma_s:sum_s})} of Lemma~\ref{lemma:lde:function_s} with $n=0$ and $\nu  = \nu_{\delta}$, we now compute
\begin{align}\label{eqn:lde:bounds:h:bounds2}
    \sum_{|x_l^t|\leq \frac{K_*}{\sqrt{t}}} |M_l(t)|&\leq  \dfrac{1}{2\pi}   \left(4+ \sqrt{2 t \pi } \sqrt{\Lambda + 2\delta} \right) 
    \Big[\dfrac{{\sqrt{2\pi}}}{\sqrt{t}\sqrt{\Lambda - 2{\delta}}}  + 2\pi e^{-\overline{m}t} \Big]
    \\
    & 
    \leq \sqrt{\dfrac{\Lambda + 2\delta}{\Lambda - 2\delta}} + {\dfrac{2\sqrt{2}}{\sqrt{\pi t(\Lambda - 2\delta)}}} + e^{-\overline{m}t}\left(4+ \sqrt{2 t } \sqrt{\pi}\sqrt{\Lambda + 2\delta} \right) 
\end{align}
for all $t\geq T$. The first term is smaller than $1+\frac{\kappa}{2}$ while the rest can be made smaller than $\frac{\kappa}{2}$ by further increasing $T$ if needed. 
\end{proof}

\begin{lemma}
 Consider the setting of Theorem~\ref{thm:lde:n-th-diff-theta}. 
 Then there exist constants $\overline{\epsilon}>0$ and  $C >0$  such that for 
 any sufficiently large $K>0$  we have the estimate
\begin{equation}
\label{eq:est:core:m:l}
\begin{array}{lclcl}
    |M_l(t)|  
& \le &
C \min\{1 , t^{-1/2} \} e^{- \overline{\epsilon} t  |l/t + a \cdot \mu|^2 / (8 K)}
\end{array}
\end{equation}
whenever $|l/t + a \cdot \mu| \le K$.
\end{lemma}
\begin{proof}
We first fix $\overline{\epsilon} > 0$
as provided in
Corollary \ref{lemma:lde:est_for_int_h}
and consider an arbitrary $K > 0$. Writing
\begin{equation}
    \delta = \frac{ K}{2\overline{\epsilon}} - \frac{\Lambda}{2},
    \qquad \qquad
    \nu_{\delta} = \frac{\overline{\epsilon}}{8K}
\end{equation}
we see that condition~\eqref{eqn:lde:def_of_delta_g} is satisfied provided that $K$ is sufficiently large. 
Recalling the notation $x_l^t = l/t + a \cdot \mu$,
we may again exploit
\eqref{lemma:lde:M_bounded_P_R}, \eqref{eqn:lde:def_P2},
\eqref{eqn:lde:est_for_int_e^tq}
and \eqref{eqn:lde:minimum_of_g}
to obtain the desired estimate
\begin{equation}
\begin{array}{lclcl}
    |M_l(t)|  &\leq & \dfrac{1}{2\pi} e^{g(l, \epsilon^*, t)} \int_{-\pi}^\pi e^{t q(\epsilon^*, \omega)} d\omega
\le &
\frac{C}{2 \pi} \min\{1 , t^{-1/2} \} e^{- \nu_{\delta} t  (x_l^t)^2}
\end{array}
\end{equation}
whenever $|x_l^t| \le K$.
\end{proof}

\begin{proof}[Proof of Lemma~\ref{lemma:lde:sum_2}]
Recalling the notation $x_l^t = l/t + a \cdot \mu$,
we
apply \eqref{eqn:tail:bound}
with $\nu= \frac{\overline{\epsilon}}{8K}$ and $K_0 = K/\sqrt{t}$
to the estimate
\eqref{eq:est:core:m:l},
which allows us to compute
\begin{equation}
    \begin{array}{lcl}
 \sum_{\frac{K}{\sqrt{t}}
 \le |x_l^t| \leq K}
 |M_l(t)|
 & \le & 
 C
 \min\{1 , t^{-1/2} \}
 \sum_{|x_l^t| \ge \frac{K}{\sqrt{t}}}
 e^{- \nu t  (x_l^t)^2}
\\[0.2cm]
 & \le &
 C
 \min\{1 , t^{-1/2} \}
 \Big[ 
 2 +4\dfrac{{\sqrt{2 K t}}}{\sqrt{\overline{\epsilon}}}
 \Big] e^{-\overline{\epsilon} K /  8}
\\[0.2cm]
 & \le &
C
 \Big[ 
 2 +4\dfrac{{\sqrt{2 K }}}{\sqrt{\overline{\epsilon}}}
 \Big] e^{-\overline{\epsilon} K /  8} .
  \end{array}
\end{equation}
This can be made arbitrarily small by choosing $K$ to be sufficiently large.
\end{proof}


\begin{proof}[Proof of Lemma~\ref{lemma:lde:decay_sum_Ml_times_ln}]
Writing
$x_l^t = l/t + a \cdot \mu$
and taking $K$ as in 
Lemma~\ref{lemma:lde:l_1_noncompact_interval},
this result shows that it suffices to find a constant $C > 0$ so that
\begin{equation*}
\label{eq:lde:est:on:core:times:x}
    \sum_{|x_l^t| \le K } |M_l(t)||x_l^t|\leq Ct^{-\frac{1}{2}}
\end{equation*}
holds for all $t > 0$.
Writing
$\nu= \frac{\overline{\epsilon}}{8K}$ and applying the bound
\eqref{eq:est:core:m:l}
we find
\begin{equation}
\begin{array}{lcl}
    \sum_{|x_l^t| \le K } |M_l(t)||x_l^t| &\leq &
    C
 \min\{1 , t^{-1/2} \}
 \sum_{|x_l^t| \le K}
 |x_l^t| e^{- \nu t  (x_l^t)^2}.
\\[0.2cm]
\end{array}
\end{equation}
The desired bound \eqref{eq:lde:est:on:core:times:x} now follows from an application of \eqref{eqn:lde:sum_s} with $n=1$.
\end{proof}

\section{Phase approximation strategies}\label{sec:theta}

Throughout this paper, various scalar LDEs of the form $\dot{\theta} = \Theta(\theta)$ are considered, which can all be seen as approximations to the (asymptotic) evolution of the 
phase $\gamma(t)$ defined in 
\eqref{eqn:main:def_of_gamma}. Our main purpose
here is to explore the relationship between the various points of view and to obtain
several key decay rates.

We proceed by introducing the standard shift operator $S:\ell^\infty(\Z) \mapsto \ell^\infty(\Z)$ that acts as
\begin{equation*}
    [S\theta]_l = \theta_{l+1}. 
\end{equation*}
This allows us to represent 
the (k)-th discrete derivative~\eqref{eqn:lde:nth_derivative} 
in the convenient form
\begin{equation}\label{eqn:analysis_theta:nth_derivative_S}
    \begin{aligned}
        \partial^{(k)}\theta  &= (S - I)^k\theta = (S^{k-1} + \cdots +S + I)(S-I)\theta \\
        & = (S^{k-1} + \cdots +S + I)\partial \theta  .
    \end{aligned}
\end{equation}
Recalling the shifts $\sigma_{\nu}$ introduced in~\eqref{eq:mr:def:shifts:sigma},
we also define the first-difference operators
\begin{equation}
\label{eq:th:def:pi:nu}
\begin{array}{lcll}
\pi^\diamond_{\nu} \theta  &=& 
(S^{\sigma_\nu} - I)\theta,
&\nu \in \left\{1, 2, 3, 4\right\},
\\[0.3cm]
\pi^\diamond_{\nu\oplus\nu'} \theta  &= &(S^{\sigma_\nu + \sigma_{\nu'}} - I)\theta, 
&\nu, \nu' \in \left\{1, 2, 3, 4\right\},
\end{array}
\end{equation}
together with their second-difference counterparts
\begin{equation}
\label{eq:th:def:pi:nu:nup}
    \pi^{\diamond \diamond} _{\nu\nu'} \theta  = \pi^\diamond_{\nu'}  \pi^\diamond_{\nu} \theta   
    =  (S^{\sigma_{\nu'}} - I)(S^{\sigma_\nu} - I)\theta,
    \qquad \qquad \nu,\nu' \in \left\{1, 2, 3, 4\right\}.
\end{equation}
These can be expanded as first differences by means of the useful identity
\begin{equation}\label{eqn:super_sub_sol:2nd_diff_via_1st_diff}
     \pi^{\diamond \diamond} _{\nu\nu'} \theta = \pi^\diamond_{\nu\oplus\nu'} \theta - \pi^\diamond_{\nu} \theta - \pi^\diamond_{\nu'} \theta .
\end{equation}
For convenience, we also introduce
the shorthands
\begin{equation}
    \pi^\diamond_{l;\nu}  \theta
    = [ \pi^\diamond_{\nu} \theta]_l,
    \qquad
    \pi^\diamond_{l;\nu\oplus\nu'}  \theta
    = [ \pi^\diamond_{\nu\oplus\nu'} \theta]_l,
    \qquad
    \pi^{\diamond\diamond}_{l;\nu\nu'}  \theta
    = [ \pi^{\diamond\diamond}_{\nu\nu'} \theta]_l
\end{equation}
for $\nu, \nu' \in \{1, 2, 3, 4\}$.

All the nonlinearities that we consider share a common linearization, which 
using \eqref{eqn:super_sub_sol:2nd_diff_via_1st_diff}
and the definitions \eqref{eqn:main:a_k}
can be represented in the equivalent forms
\begin{equation}
\label{eq:th:def:h:lin}
\begin{array}{lcl}
    \mathcal{H}_{\mathrm{lin}}[h]
    &= & \sum_{\nu=1}^4 \alpha^\diamond_{p;\nu} \pi^\diamond_{\nu}h +
    \sum_{\nu,\nu'=1}^4 \alpha_{p;\nu\nu'}^{\diamond \diamond} \pi^{\diamond \diamond}_{\nu\nu'} h
\\[0.2cm]
& = & 
\sum_{\nu=1}^4 \alpha^\diamond_{p;\nu}
\alpha^\diamond_{p;\nu} \pi^\diamond_{\nu}h 
+ \sum_{\nu,\nu'=1}^4
\alpha_{p;\nu\nu'}^{\diamond\diamond}
\big( \pi_{\nu \oplus \nu'} - \pi^\diamond_{\nu} - \pi^\diamond_{\nu'}  \big) h
\\[0.2cm]
& = &
\sum_{k=-N}^N a_k (S^k - I) h.
\end{array}
\end{equation}
It is important to observe
that the assumptions (HS$)_1$ and (HS$)_2$ 
guarantee that condition (h$\alpha$) in
{\S}\ref{sec:lde} is satisfied. In particular,
we will be able to exploit all the linear results obtained in that section.

\paragraph{Summation convention}
To make our notation more concise, we will use the Einstein summation convention whenever this is not likely to lead to ambiguities. This means that
any Greek index that appears only on the right hand side of an equation
is automatically summed.
For example, the first line of
\eqref{eq:th:def:h:lin} can be simplified as
\begin{equation}
\begin{array}{lcl}
    \mathcal{H}_{\mathrm{lin}}[h]
    &= &  \alpha^\diamond_{p;\nu} \pi^\diamond_{\nu}h 
    +
     \alpha_{p;\nu\nu'}^{\diamond\diamond} \pi^{\diamond \diamond}_{\nu\nu'} h .
\end{array}
\end{equation}

\paragraph{`Cole-Hopf' nonlinearity
$\Theta_{\mathrm{ch}}$}
We start by discussing the
nonlinearity $\Theta_{\mathrm{ch}}$ defined by~\eqref{eqn:main:eqn_for_alpha_diamonds},
which for $d \neq 0$ is given by
\begin{equation}
\label{eq:th:def:theta:ch}
[\Theta_{\mathrm{ch}}(\theta)]_l
= \dfrac{1}{d} \sum_{k=-N}^N a_k \left(e^{d\big(\theta_{l+k}(t) - \theta_l(t)\big)} - 1 \right)   + c_*.
\end{equation}
The key feature is that any solution
to 
\begin{equation}
\label{eq:th:dot:theta:ch}
    \dot{\theta} = \Theta_{\mathrm{ch}}(\theta)
\end{equation} 
can be used to construct a solution to the linear problem
\begin{equation}
\label{eq:th:dot:h:ch}
    \dot{h}(t) = \mathcal{H}_{\mathrm{lin}} [h(t)]
\end{equation}
by applying the Cole-Hopf transform
\begin{equation}
\label{eq:th:def:cole:hopf:h}
    h(t) = e^{d\big(\theta(t) - c_* t\big)}.
\end{equation}
This can be verified by a straightforward computation. Vice versa, any nonnegative solution to the linear LDE \eqref{eq:th:def:cole:hopf:h}
yields a solution to
\eqref{eq:th:dot:theta:ch}
by writing
\begin{equation}
\label{eq:th:rec:theta:from:h}
    \theta(t) = \dfrac{\log h(t) }{d} + c_* t.
\end{equation}

Our first main result uses this correspondence to establish bounds on the discrete derivatives of solutions to \eqref{eq:th:dot:theta:ch}.
In order to exploit the 
fact that this LDE
is invariant under spatially homogeneous perturbations, we introduce the deviation seminorm
\begin{equation}\label{eqn:analysis_of_theta:semiorm}
    [\theta]_{\mathrm{dev}} = \norm{\theta  - \theta_0}_{\ell^\infty}
\end{equation}
for sequences $\theta\in \ell^\infty(\Z)$.
In view of \eqref{eq:th:rec:theta:from:h}, it is essential to ensure
that $h$ remains positive. This is where 
Proposition \ref{prop:lde:comparison_principle} comes into play, which requires us to impose a
flatness condition on the initial condition $\theta(0)$. This was not needed for the corresponding result \cite[Cor. 6.2]{jukic2019dynamics},
where the comparison principle could be exploited.

\begin{proposition}\label{prop:lde:analysis_theta:differences} 
Assume that (H$g$), (H$\Phi$),   (HS$)_1$ and (HS$)_2$ 
all hold and fix a positive constant $R>0$. Then there exist constants $M$ and $\delta$ such that for any $\theta \in C^1\big([0,\infty);\ell^\infty(\Z)\big)$
that satisfies the LDE~\eqref{eqn:main:main_eqn_for_theta}
with 
$[\theta(0)]_{\mathrm{dev}} < R$ and $\norm{\partial \theta(0)}_{\ell^\infty} \leq \delta$, 
we have the estimates
\begin{align}
 	\norm{\partial^{(k)} \theta(t)}_{\ell^\infty} &\leq M \min\left\{\norm{\partial \theta^0}_{\ell^\infty}, t^{-\frac{k}{2}}\right\}, \qquad k= 1, 2, 3. \label{eqn:analysis_theta:differences}
 \end{align}
Moreover, for any pair $(m,n)\in \Z^2$ there exists a constant $C= C(m,n, R)$ such that 
\begin{align}
    \norm{
    n(S^m - I)\theta(t) - m (S^n - I)\theta(t)
    }_{\ell^\infty} &\leq   C \min\left\{\norm{\partial \theta^0}_{\ell^\infty}, t^{-1}\right\} , \label{eqn:analysis_theta:differences_mn} 
 	\\
    \norm{
     n \partial(S^m - I)\theta(t) - m \partial (S^n - I)\theta(t) 
    }_{\ell^\infty} &\leq  C \min\left\{\norm{\partial \theta^0}_{\ell^\infty}, t^{-\frac{3}{2}}\right\} . \label{eqn:analysis_theta:differences_mn2} 
 \end{align}
\end{proposition}

\paragraph{`Comparison' nonlinearity $\Theta_{\mathrm{cmp}}$}
Upon introducing the quadratic expression
\begin{equation}
\label{eq:th:def:q:cmp}
    [\mathcal{Q}_{\mathrm{cmp}}(\theta)]_l
    = 
    \alpha_{q;\nu\nu'}^{\diamond\diamond} [\pi^\diamond_{l;\nu}\theta] [\pi^\diamond_{l;\nu'} \theta],
\end{equation}
we are ready to define
a new nonlinear function $  \Theta_{\mathrm{cmp}}:\ell^\infty(\Z)\to \ell^\infty(\Z)$
that acts as
\begin{align}
  \Theta_{\mathrm{cmp}}(\theta)
    &=
    \mathcal{H}_{\mathrm{lin}}[ \theta]
    + \mathcal{Q}_{\mathrm{cmp}}(\theta)
    + c_*.
    \label{eqn:main:main_eqn_for_theta_approx}
\end{align}
This function plays an important role in {\S}\ref{sec:sub:sup} where we construct sub- and super solutions for~\eqref{eqn:main:AC_equation} in order to exploit the comparison principle. Indeed,
our choice \eqref{eqn:intro:super_sol_ansatz} will generate terms in the super-solution residual
that contain the factor 
\begin{equation}
\label{eq:th:def:r:theta}
\begin{aligned}
        \mathcal{R}_{\theta}(t)
   :& =  \dot{\theta}(t)- \Theta_{\mathrm{cmp}}\big(\theta(t)\big) 
   \\[0.3cm]
   &= \alpha^\diamond_{p;\nu} \pi^\diamond_{\nu} \theta(t)
    +
     \alpha_{p;\nu\nu'}^{\diamond \diamond} \pi^{\diamond \diamond}_{\nu\nu'} \theta(t) +  \alpha_{q;\nu\nu'}^{\diamond\diamond} [\pi^\diamond_{l;\nu}\theta(t)] [\pi^\diamond_{l;\nu'} \theta(t)] + c_*.
\end{aligned}
\end{equation}
Since this difference does not have a sign that we can exploit, we need to 
absorb it into our remainder terms. This requires a decay rate of
$\mathcal{R}_{\theta}(t) \sim t^{-3/2}$ or faster.

Obviously, we can achieve $\mathcal{R}_{\theta} = 0$ by 
choosing $\theta$ appropriately. However, 
the resulting LDE $\dot{\theta} = \Theta_{\mathrm{cmp}}(\theta)$
is surprisingly hard to analyze due to the presence of the problematic quadratic terms,
which precludes us from obtaining the desired $\partial \theta \sim t^{-1/2}$ decay rates in $\ell^\infty$ (rather than $\ell^2$, which is much easier). 

This problem is circumvented by our choice to use \eqref{eq:th:dot:theta:ch}
as the evolution for $\theta$. Our second main result provides the necessary bounds on $\mathcal{R}_{\theta}(t)$ and two other related expressions.
The main challenge here
is to compare the quadratic
terms in $\Theta_{\mathrm{cmp}}$
and $\Theta_{\mathrm{ch}}$.
In fact, our choice
\eqref{eqn:main:d}
for the parameter $d$ is
based on the necessity
to  neutralize the dangerous components that lead to $O(t^{-1})$ behaviour.

\begin{proposition}\label{prop:analysis_theta:theta_approx_quadratic}
Consider the setting of Proposition~\ref{prop:lde:analysis_theta:differences}. There exist constants $M$ and $\delta$ such that for any $\theta \in C^1\big([0,\infty);\ell^\infty(\Z)\big)$
that satisfies the LDE~\eqref{eq:th:dot:theta:ch}
with 
$[\theta(0)]_{\mathrm{dev}} < R$ and $\norm{\partial \theta(0)}_{\ell^\infty} \leq \delta$, we have the bound
\begin{equation}        
\label{eqn:super_and_sub_sols_Rtheta_est}
 \norm{\mathcal{R}_\theta(t)}_{\ell^\infty}
   \le M \min\left\{\norm{\partial \theta(0)}_{\ell^\infty}, t^{-\frac{3}{2}}\right\}
\end{equation}
for all $t > 0$. In addition,
for any $\nu \in \{1, 2,3,4\}$
and $t > 0$
we have 
\begin{equation}        
\label{eqn:super_and_sub_sols_Rtheta_est:deriv:i}
  \norm{
\pi^\diamond_{\nu}\dot{\theta}(t) - \sum_{\nu'=1}^4 \alpha^\diamond_{p;\nu'} \pi^{\diamond \diamond}_{\nu'\nu}\theta(t)
}_{\ell^\infty}
  \le M \min\left\{\norm{\partial \theta(0)}_{\ell^\infty}, t^{-\frac{3}{2}}\right\},
\end{equation}
while for any pair 
$\nu,\nu' \in \{1, 2,3,4\}$
and $t > 0$ we have
\begin{equation}        
\label{eqn:super_and_sub_sols_Rtheta_est:deriv:ii}
  \norm{
          \pi^{\diamond \diamond}_{\nu \nu'} \dot{\theta}(t)
        }_{\ell^\infty}
  \le M \min\left\{\norm{\partial \theta(0)}_{\ell^\infty}, t^{-\frac{3}{2}}\right\} .
\end{equation}
\end{proposition}

\paragraph{`Discrete mean curvature' nonlinearity
$\Theta_{\mathrm{dmc}}$}
Recalling the sequences
$(A_k)$ and $(B_k)$ together with the
functions $\beta_\theta$ and $\Delta_\theta$ defined
in \eqref{eq:int:lap:beta:gamma}, we generalize the definition of $\overline{c}_{\theta}$
from \eqref{eq:int:c:gamma}
slightly by
writing
\begin{equation}
\label{eq:analysis_theta:dc:lap:gamma}
    [\widetilde{c}_\theta]_l
    =\sum_{0<|k|\leq N}  C_k c_{\varphi_{l;k}(\theta)}
\end{equation}
for a sequence $(C_k)$ that must satisfy
\begin{equation}
    \label{eq:th:nrm:C}
\sum_{0 < |k| \le N} C_k = 1,
\qquad \qquad
\sum_{0 < |k| \le N} k C_k = 0.
\end{equation}
The corresponding
generalization
of the definition
\eqref{eqn:main:Theta_dmc}
for $\Theta_{\mathrm{dmc}}$
is now given by
\begin{equation}\label{eqn:th:dmc:wt}
\widetilde{\Theta}_{\mathrm{dmc}}(\theta) =  \kappa_H \frac{{\Delta}_\theta}{\beta_\theta^2} + \beta_\theta \widetilde{c_\theta} ,
\end{equation}
which reduces to
$\Theta_{\mathrm{dmc}}$
in the special case $C_k = 1/(2N)$. 

Our task here is to 
establish a slight
generalization of Proposition~\ref{prop:main:mcf}
by analyzing the difference of $\widetilde{\Theta}_{\mathrm{dmc}}$ with $\Theta_{\mathrm{ch}}$. We achieve this by expanding the
 direction-dependent 
 wavespeeds $c_{\varphi}$ introduced in Lemma \ref{lem:mr:dir:dep:waves}
 in terms of the angle $\varphi$. In particular,
 we provide proofs for the explicit
 expressions 
 stated in Lemma
 \ref{lemma:main:expressions_c_lambda}. This allows us to make
 the link with the
 identities
 \eqref{eq:int:cond:kappa:d}
  for the parameters $\kappa_H$ and $d$.

\subsection{Coefficient identities  }

Our results in this section strongly depend
on the equivalence
of the representations
\eqref{eqn:main:d} for the parameter $d$. Our goal is
to establish this equivalence by providing the proofs of Lemma~\ref{lemma:main_results:coeff_diamond} and Lemma~\ref{lemma:main:expressions_c_lambda}.
To set the stage,
we recall
the set of shifts
\begin{equation*}
    (\tau_1, \tau_2, \tau_3, \tau_4) = (\sigma_h, \sigma_v, -\sigma_h, -\sigma_v)
\end{equation*}
and their corresponding translation operators 
\begin{equation*}
    [T_\nu h](\xi) = h(\xi + \tau_\nu), \qquad \nu \in \left\{1, 2, 3, 4\right\}
\end{equation*}
that were defined in \eqref{eq:mr:def:T:nu}.
This allows us
to recast the
direction-dependent
travelling-wave MFDE \eqref{eqn:main:perturbed:mfde}  in the convenient form
\begin{equation}
\label{eq:th:dir:dep:mfde}
    \begin{aligned}
        -c_\varphi \Phi_\varphi'(\xi) &= \Phi(\xi + \tau_\nu \cos{\varphi} + \sigma_\nu \sin{\varphi})
        -4 \Phi_\varphi(\xi) + g\big(\Phi_\varphi(\xi) \big),
    \end{aligned}
\end{equation}
which after linearization around $\Phi_{\varphi}$ gives rise to the linear operators
\begin{equation}
    [\mathcal{L}^{\varphi} v](\xi) = c_\varphi v'(\xi) + v(\xi + \tau_\nu \cos{\varphi} + \sigma_\nu \sin{\varphi}) - 4 v(\xi) + g'\big(\Phi_\varphi(\xi)\big) v(\xi).
\end{equation}
These should not be confused with their counterparts $\mathcal{L}_{\omega}$ defined
in \eqref{eq:mr:def:l:omega}, agreeing only when $\varphi = \omega = 0$.
\begin{proof}[Proof of Lemma~\ref{lemma:main_results:coeff_diamond}]
In view of the definition~\eqref{eqn:main:eqn_for_alpha_diamonds} and the identity $\langle \Phi_*', \psi_* \rangle = 1$, we have
\begin{equation*}
    \left \langle T_\nu \Phi_*', \psi_* \right\rangle - \alpha_{p;\nu}^\diamond  \left \langle \Phi_*', \psi_* \right\rangle =  \alpha_{p;\nu}^\diamond  -  \alpha_{p;\nu}^\diamond = 0,
\end{equation*}
for each fixed $\nu \in \{1, 2, 3, 4\}$,
which implies that
$T_\nu \Phi_*' - \alpha_{p;\nu}^\diamond  \Phi_*' \in \mathcal{R}(\mathcal{L}_0)$
by Lemma~\ref{lemma:main_results:lambda_omega}.
In particular, we can find
a bounded $C^1$-smooth function $\overline{p}_\nu^\diamond$
for which 
$$\mathcal{L}_0
\big[ \overline{p}_\nu^\diamond + b\Phi_*' \big] = T_\nu \Phi_*' - \alpha_{p;\nu}^\diamond  \Phi_*' $$ 
holds for any $b\in \R$.
Setting $b = - \langle \overline{p}_\nu^\diamond , \psi_* \rangle $ 
we can construct our desired
function ${p}_\nu^\diamond$ by writing
${p}_\nu^\diamond=\overline{p}_\nu^\diamond + b\Phi_*'$.
The remaining  functions ${p}_{\nu \nu'}^{\diamond \diamond}$ and ${q}_{\nu \nu'}^{\diamond \diamond}$
can be constructed analogously.
\end{proof}

\begin{proof}[Proof of Lemma~\ref{lemma:main:expressions_c_lambda}]
To establish item~(\textit{\ref{item:analysis_theta:c}}),
we introduce the
function $\chi(\xi):= \xi \Phi_*'(\xi)$ and use
the MFDE \eqref{eq:th:dir:dep:mfde}
at $\varphi = 0$ to compute
\begin{align*}
   [ \mathcal{L}_0 \chi](\xi) &= c_*\Phi_*'(\xi) + c_*\xi \Phi_*''(\xi) + (\xi+ \tau_\nu) T_\nu \Phi_*'(\xi)- 4\xi \Phi_*'(\xi) + \xi g'\big(\Phi_*(\xi)\big)\Phi_*'(\xi) 
   \\
   & = c_*\Phi_*'(\xi) + T_\nu \Phi_*'(\xi) + \xi \frac{d}{d\xi}\Big( c_* \Phi_*'(\xi) + \tau_{\nu} T_\nu \Phi_*(\xi) - 4 \Phi_*(\xi) + g\big(\Phi_*(\xi)\big)\Big)
   \\
   & = c_*\Phi_*'(\xi) + \tau_\nu T_\nu \Phi_*'(\xi).
\end{align*}
We integrate this expression against the kernel element $\psi_*$ and recall the definition of $\alpha_{p;\nu}^\diamond$ from Lemma~\ref{lemma:main_results:coeff_diamond} to obtain $c_* = -\tau_\nu \alpha_{p;\nu}^\diamond$, as claimed.

Turning to the other items, 
  we differentiate the equation~\eqref{eq:th:dir:dep:mfde}
with respect to $\varphi$. This yields
\begin{equation}\label{eqn:analysis_theta:partial_c}
    \begin{aligned}
    -[\partial_\varphi c_\varphi] \Phi_\varphi'(\xi)  &= [\mathcal{L}^\varphi \partial_\varphi \Phi_\varphi](\xi)
  +\Phi_\varphi'(\xi + \tau_\nu \cos{\varphi} + \sigma_\nu \sin{\varphi})(-\tau_\nu\sin{\varphi} + \sigma_\nu \cos{\varphi}),
\end{aligned}
\end{equation}
where we emphasize that
differentiations with respect to the angle $\varphi$
will always be denoted by $\partial_\varphi$.
Evaluating~\eqref{eqn:analysis_theta:partial_c}  in $\varphi = 0$, we obtain
\begin{equation*}
    -[\partial_\varphi c_{\varphi}]_{\varphi= 0}\Phi_*'(\xi) =  [\mathcal{L}_0 [\partial_\varphi \Phi_\varphi]_{\varphi = 0}] (\xi)+ \sigma_\nu T_\nu \Phi_*'(\xi).
\end{equation*}
Integrating
against  the adjoint kernel element $\psi_*$,
we
may use the characterization~\eqref{eqn:main:range_L0}
in 
combination with Lemma~\ref{lemma:main_results:coeff_diamond}
to arrive at the explicit expression
\begin{equation}
   - [\partial_\varphi c_\varphi]_{\varphi = 0}  = 
   \sigma_\nu \langle T_\nu \Phi'_*, \psi_* \rangle
   = \sigma_\nu \alpha^\diamond_{p;\nu}
\end{equation}
stated in  \textit{(ii)}.
Applying Lemma~\ref{lemma:main_results:coeff_diamond} once more,
we subsequently obtain
\begin{equation*}
  \mathcal{L}_0 [\partial_\varphi \Phi_\varphi]_{\varphi = 0}  = \sigma_\nu \alpha_{p;\nu}^\diamond \Phi_*' - \sigma_\nu [T_{\nu} \Phi_*']
  = -\mathcal{L}_0 (\sigma_\nu p_{\nu}^\diamond).
\end{equation*}
The Fredholm properties
formulated in 
Lemma \ref{lem:mr:fred:props:l:0} hence imply
\begin{equation*}
    [\partial_\varphi \Phi_\varphi]_{\varphi = 0} = -\sigma_\nu p_{\nu}^\diamond + b\Phi_*', 
\end{equation*}
where the coefficient $b$ is given by
\begin{equation*}
    b=\langle [\partial_\varphi \Phi_\varphi]_{\varphi = 0}, \psi_* \rangle + \sigma_\nu \langle p_{\nu}^\diamond, \psi_* \rangle .
\end{equation*}
This vanishes on account
of the normalization choices
in Lemmas 
\ref{lemma:main_results:lambda_omega} and~\ref{lemma:main_results:coeff_diamond},
establishing~(\textit{\ref{item:analysis_theta:partial_phi}}).

A further differentiation
of \eqref{eqn:analysis_theta:partial_c} with respect to $\varphi$ yields
\begin{equation}\label{eqn:analysis_theta:partial2_c}
    \begin{aligned}
        -[\partial_\varphi^2 c_\varphi] \Phi_\varphi'(\xi)   &= 
        2[\partial_\varphi c_\varphi] \partial_{\varphi}\Phi_\varphi'(\xi)  + [\mathcal{L}^{\varphi} \partial^2_\varphi \Phi_\varphi] (\xi) 
    \\
    &\qquad +  2 \partial_\varphi \Phi_\varphi'(\xi + \tau_\nu \cos{\varphi} + \sigma_\nu \sin{\varphi})(-\tau_\nu\sin{\varphi} + \sigma_\nu \cos{\varphi})   \\
        &\qquad + \Phi_\varphi'(\xi + \tau_\nu \cos{\varphi} + \sigma_\nu \sin{\varphi})(-\tau_\nu \cos{\varphi} - \sigma_\nu \sin{\varphi})
        \\
        &\qquad
        + \Phi_\varphi''(\xi + \tau_\nu \cos{\varphi} + \sigma_\nu \sin{\varphi})(-\tau_\nu\sin{\varphi} + \sigma_\nu \cos{\varphi})^2
       \\
       &\qquad 
       + 
        g''\big(\Phi_\varphi(\xi)\big)[\partial_\varphi\Phi_\varphi(\xi)]^2 .
    \\
    \end{aligned}
\end{equation}
Evaluating this 
in $\varphi = 0$ and integrating
against the adjoint kernel element $\psi_*$, we 
obtain
\begin{equation*}
    \begin{aligned}
        -[\partial_\varphi^2 c_\varphi]_{\varphi = 0}  &= 
         2[\partial_\varphi c_\varphi]_{\varphi=0} \left\langle [\partial_{\varphi}\Phi_\varphi']_{\varphi = 0 }, \psi_*\right\rangle  
         +  2 \sigma_\nu \langle T_\nu[\partial_\varphi \Phi_\varphi']_{\varphi = 0}, \psi_*\rangle    \\
        & \qquad -\tau_\nu \langle T_\nu \Phi', \psi_* \rangle
        + \langle T_\nu\Phi_*'', \psi_* \rangle \sigma_\nu^2 
        +
       \langle g''(\Phi_*(\xi))[\partial_\varphi\Phi_\varphi]_{\varphi = 0}^2 , \psi_*\rangle.
    \end{aligned}
\end{equation*}
Substituting the expressions
from items~(\textit{\ref{item:analysis_theta:c}}),  (\textit{\ref{item:analysis_theta:partial_c}}) and (\textit{\ref{item:analysis_theta:partial_phi}}), we arrive at
\begin{equation*}
    \begin{aligned}
        -[\partial_\varphi^2 c_\varphi]_{\varphi = 0}  &= 
        -2\sigma_\nu\alpha_{p;\nu}^\diamond \langle   [\partial_\varphi \Phi_\varphi']_{\varphi = 0}, \psi_*\rangle
               + 2 \sigma_\nu  \langle T_\nu [\partial_\varphi \Phi_\varphi']_{\varphi = 0}, \psi_*\rangle    
        + \langle T_\nu \Phi_\varphi'', \psi_* \rangle \sigma_\nu^2 
        \\
       & \qquad
        -\tau_\nu \alpha_{p;\nu}^\diamond 
              + \langle  g''(\Phi_*)[\partial_\varphi\Phi_\varphi]_{\varphi = 0}^2 , \psi_*\rangle 
       \\
       &= 
       2\sigma_\nu\alpha_{p;\nu}^\diamond \langle \sigma_{\nu'} \frac{d}{d\xi} p_{\nu'}^\diamond , \psi_*\rangle 
       - 2 \sigma_\nu \sigma_{\nu'}\langle T_\nu \frac{d}{d\xi}p_{\nu'}^\diamond, \psi_*\rangle    
        + \langle T_\nu \Phi_*'', \psi_* \rangle \sigma_\nu^2 
        \\
        &\qquad
        +  c_* +  \sigma_\nu \sigma_{\nu'} \langle g''(\Phi_*) p_{\nu}^\diamond p_{\nu'}^\diamond , \psi_*\rangle 
       \\
       & =  c_* - 2\sigma_\nu \sigma_{\nu'} \alpha_{q;\nu\nu'}^{\diamond \diamond},
\end{aligned}
\end{equation*}
which establishes~\textit{(\ref{item:analysis_theta:c_2nd_derivative})}. 
Finally,
items~\textit{(\ref{item:analyis_theta:partial_lambda})}
and~\textit{(\ref{item:analyis_theta:partial_lambdax2})}
follow directly from Lemma~5.6 in \cite{hoffman2017entire} and the definition~\eqref{eqn:main:a_k} for the coefficients $(a_k)$.
\end{proof}

\subsection{Quadratic comparisons}

In order to establish the main results of this section
we need to carefully examine the structure 
of the quadratic terms in our nonlinearities. As a preparation, we first confirm
that the difference operators
\eqref{eq:th:def:pi:nu}
and \eqref{eq:th:def:pi:nu:nup} can be appropriately bounded by the corresponding discrete derivatives.

\begin{lemma}\label{lem:th:bnd:pi:fncs}
   There exist a constant $M > 0$ so that for any 
   $\theta\in\ell^\infty(\Z)$
   and any $\nu, \nu', \nu''\in \{1, 2,3, 4\}$
we have the estimates
\begin{align}
    \norm{ \pi_{ \nu}^{\diamond} \theta ]}_{\ell^\infty}&\leq M
    \norm{ \partial \theta}_{\ell^\infty},
    \label{eqn:analysis_theta:first_diff_nu}
    \\
    \norm{ \pi_{ \nu\nu'}^{\diamond \diamond} \theta ]}_{\ell^\infty}&\leq M\norm{ \partial^{(2)} \theta}_{\ell^\infty}, \label{eqn:analysis_theta:second_diff_nu}
    \\
     \norm{\pi^\diamond_{\nu} [\pi_{ \nu'\nu''}^{\diamond \diamond} \theta(t) ]}_{\ell^\infty}&\leq M\norm{\partial^{(3)} \theta}_{\ell^\infty} \label{eqn:analysis_theta:first_diff_second_diff_nu},
    \\
     \norm{\pi^\diamond_{\nu} \big[
     [\pi_{ \nu'}^{\diamond } \theta][ \pi_{ \nu''}^{\diamond } \theta] \big] }_{\ell^\infty}&\leq M\norm{\partial^{(2)} \theta}_{\ell^\infty}
     \norm{\partial \theta}_{\ell^\infty}. \label{eqn:analysis_theta:first_diff_squares_nu}
\end{align}
\end{lemma}
\begin{proof}
The first three bounds
follow directly 
from the fact that the difference operators $\pi_{; \nu}^{\diamond} $ can all be represented in the form
\begin{equation}
    \pi_{;{\nu}} = S^{- \max\{ |\sigma_h|, |\sigma_v|} P_{\nu}(S)(S-I)
\end{equation}
for appropriate polynomials $P_{\nu}$. The final
bound follows from the product
rule
\begin{equation}
    (S^{n} -I ) [\theta_1 \theta_2]
    = [S^n \theta_1](S^n  -I) \theta_2
    + [(S^n - I) \theta_1 ]  \theta_2
\end{equation}
which holds for all $\theta_1, \theta_2 \in \ell^\infty(\Z)$.
\end{proof}

We proceed
with our analysis by 
providing an explicit formula for the operators $S^m-I$,
which isolates the terms
for which only a single discrete derivative $\partial$ can be factored out.
This leads naturally to the crucial bounds
\eqref{eq:th:bnd:n:m:theta},
which will allow us to extract
additional decay from suitably
combined first-difference operators.

\begin{lemma}\label{lemma:analysis_theta:Sm-I}
For any integer $m \ge 1$ we have the identities
\begin{equation}\label{eqn:lde:nonlinear:induction}
\begin{array}{lcl}
    (S^m - I) &= &(S-I)^2\sum_{k=0}^{m-2} (m-k-1) S^k  + m (S-I),  \\[0.3cm]
     (S^{-m} - I) &= &-(S-I)^2\sum_{k=0}^{m-2} (m-k-1) S^{k-m} - m S^{-m}(S-I).
\end{array}
\end{equation}
\end{lemma}
\begin{proof}
We only consider the first identity, noting that the second one follows readily from the computation
\begin{equation*}
    S^{-m} - I = -S^{-m}(S^m - I). 
\end{equation*}
For $m =1$ the claim follows trivially, while for $m=2$ we  have $ (S^2 - I) = (S^2 - 2S + I) + 2(S-I) $. Assuming that \eqref{eqn:lde:nonlinear:induction} holds for all $k$ up to some $m\geq  2$, we compute
\begin{align*}
   (S^{m+1} - I) &= 
   S(S^{m} - I) + (S-I) 
   \\
   &=(S-I)^2\sum_{k=0}^{m-2} (m-k-1) S^{k+1}  + m S(S-I) + (S-I) 
 \\
 & = (S-I)^2\sum_{k=1}^{m-1} (m-k) S^{k}  + m S(S-I) + (S-I) .
\end{align*}
Adding and subtracting $m(S-I)^2$ results in the desired identity
\begin{align*}
 (S^{m+1} - I)  
 & = (S-I)^2\sum_{k=0}^{m-1} (m-k) S^{k}  + (m + 1)(S-I).
\end{align*}
\end{proof}

\begin{cor}
Pick a pair $(m,n) \in \Z^2$.
Then there exists a constant $C = C(m,n) > 0$ so that for any $\theta \in \ell^\infty(\Z)$ we have the bounds
\begin{equation}
\label{eq:th:bnd:n:m:theta}
\begin{array}{lcl}
    \norm{n(S^m - I)\theta - m (S^n - I) \theta }_{\ell^\infty} 
    & \le & C \norm{\partial^{(2)} \theta }_{\ell^\infty},
\\[0.2cm]
    \norm{n^2[(S^m-I)\theta]^2  - m^2  [(S^n-I)\theta]^2 }_{\ell^\infty} & \le & C \norm{\partial^{(2)} \theta }_{\ell^\infty} \norm{\partial \theta }_{\ell^\infty},
\end{array}
\end{equation}
where the squares are evaluated in a pointwise fashion.
\end{cor}
\begin{proof}
To establish the first bound,
we assume without loss that $m >0$ and set out to exploit the identities \eqref{eqn:lde:nonlinear:induction}. The key observation is that all terms featuring an $(S-I)^2$ factor can be absorbed by the stated bound.
If also $n > 0$, then the remaining terms involving $(S-I)$ factors cancel.
If $n  < 0$, then we compute
\begin{equation}
    nm(S-I) - mn S^{n} (S-I)
    = nm (I - S^n)(S-I),
\end{equation}
which can be written as a sum of (shifted) second-differences. 
The second bound now follows directly from the standard factorization
$a^2 - b^2 = (a+b)(a-b)$.
\end{proof}

The bounds above can be used
to reduce the mixed products
appearing in
the definition \eqref{eq:th:def:q:cmp} for $\mathcal{Q}_{\mathrm{cmp}}$ as a sum of pure squares.
Inspired by \eqref{eqn:super_sub_sol:2nd_diff_via_1st_diff} and the identity
$2ab = (a+b)^2 - a^2 -b^2$,
we
introduce the expression
\begin{equation}
\label{eq:th:def:q:i}
    \mathcal{Q}_{\mathrm{cmp};I}(\theta)
    = \frac{1}{2}\alpha_{q;\nu\nu'}^{\diamond\diamond} 
    \Big(
      [\pi_{\nu \oplus \nu'} \theta]^2
      - [\pi_{\nu}\theta ]^2
       - [\pi_{\nu'} \theta]^2
      \Big).
\end{equation}

\begin{lemma}\label{lemma:analysis_theta:Q-QI}
Consider the setting of Proposition~\ref{prop:lde:analysis_theta:differences}.
Then there exists $C > 0$
so that for any $\theta \in \ell^\infty(\Z)$ we have the bound
\begin{equation}
\label{eq:bnd:on:q:vs:q:i}
    \norm{\mathcal{Q}_{\mathrm{cmp}}(\theta) - \mathcal{Q}_{\mathrm{cmp};I}(\theta)}_{\ell^\infty} \le 
    C \norm{\partial \theta}_{\ell^\infty}
    \norm{\partial^{(2)} \theta}_{\ell^\infty} .
\end{equation}
\end{lemma}
\begin{proof}

We fix the pair $(\nu, \nu')$ and  use the relation~\eqref{eqn:super_sub_sol:2nd_diff_via_1st_diff} to write
\begin{equation*}
    [\pi_{\nu} \theta ][\pi_{\nu'} \theta]
    = \frac{1}{2}\big[
       (\pi_{\nu} \theta + \pi_{\nu'} \theta)^2
       - [\pi_{\nu}\theta ]^2
       - [\pi_{\nu'} \theta]^2
    \big]
    = \frac{1}{2}
    \big[
      (\pi_{\nu \oplus \nu'} \theta - \pi_{\nu \nu'} \theta)^2
      - [\pi_{\nu}\theta ]^2
       - [\pi_{\nu'} \theta]^2
    \big] .
\end{equation*}
In particular, we see that
\begin{align*}
    \mathcal{Q}_{\mathrm{cmp}}(\theta)
    -\mathcal{Q}_{\mathrm{cmp};I}(\theta)
    &=
      \frac{1}{2}\alpha^{\diamond \diamond}_{q;\nu \nu'}
       [\pi_{\nu \nu'} \theta] \big(\pi_{\nu \nu'} \theta - 2 \pi_{\nu \oplus \nu'} \theta 
      \big),
\end{align*}
from which the bound is immediate by Lemma~\ref{lem:th:bnd:pi:fncs}.
\end{proof}

Turning to $\Theta_{\mathrm{ch}}$, we introduce the quadratic expression
\begin{equation*}
[\mathcal{Q}_{\mathrm{ch}}(\theta)]_l
= \frac{1}{2} \sum_{k=-N}^N a_k (\theta_{l+k} - \theta_l)^2
\end{equation*}
and note that $d \mathcal{Q}_{\mathrm{ch}}(\theta)$
is the second-order term in the Taylor expansion of 
\eqref{eq:th:def:theta:ch}.
Inserting the definitions~\eqref{eqn:main:a_k} for the coefficients $(a_k)$, we see that
\begin{equation}
\label{eq:th:def:q:ii}
    \mathcal{Q}_{\mathrm{ch}}(\theta)
    =
    \frac{1}{2}
    \Big(
     \alpha^{\diamond \diamond}_{p; \nu \nu'}
     \big( [\pi_{\nu \oplus \nu'} \theta]^2
     - [\pi_{\nu} \theta]^2
     - [\pi_{\nu'} \theta]^2
     \big)
     + \alpha^\diamond_{p;\nu} [\pi_{\nu} \theta]^2
    \Big),
\end{equation}
which closely resembles the structure of
\eqref{eq:th:def:q:i}. Indeed,
in both cases the slowly-decaying terms can be isolated in a transparant fashion, using
the coefficients
\begin{equation}
    \beta_{\mathrm{cmp}} = 
     \sigma_{\nu} \sigma_{\nu'} \alpha^{\diamond \diamond}_{q ; \nu \nu'},
    \qquad \qquad
    \beta_{\mathrm{ch}}
    =  \sigma_{\nu} \sigma_{\nu'} \alpha^{\diamond \diamond}_{p ; \nu \nu'}
     +\frac{1}{2} \sigma_{\nu}^2 \alpha^{\diamond}_{p;\nu}.
\end{equation}

\begin{lemma}
Consider the setting of Proposition~\ref{prop:lde:analysis_theta:differences}.
Then there exists $C > 0$
so that for any $\theta \in \ell^\infty(\Z)$ we have the bound
\begin{align}
    \norm{ \mathcal{Q}_{\mathrm{cmp};I}(\theta)
     - \beta_{\mathrm{cmp}}
     (\partial \theta)^2
     }_{\ell^\infty} 
     + 
     \norm{ \mathcal{Q}_{\mathrm{ch}}(\theta)
     - \beta_{\mathrm{ch}}
     (\partial \theta)^2
     }_{\ell^\infty} 
      &\le  C \norm{ \partial \theta}_{\ell^\infty}
     \norm{ \partial^{(2)} \theta}_{\ell^\infty} . \label{eqn:analysis_theta:Q_I-beta_I}
\end{align}
\end{lemma}
\begin{proof}
In view of \eqref{eq:th:bnd:n:m:theta}
we have the bound
\begin{equation}
\norm{[\pi_{\nu \oplus \nu'} \theta]^2     - (\sigma_{\nu} + \sigma_{\nu'})^2 [\partial \theta]^2}_{\ell^\infty}
+ 
\norm{[\pi_{\nu} \theta]^2     - \sigma_{\nu}^2 [\partial \theta]^2}_{\ell^\infty}
\le C \norm{ \partial \theta}_{\ell^\infty}
     \norm{ \partial^{(2)} \theta}_{\ell^\infty},
\end{equation}
which can be directly
applied to the definitions
\eqref{eq:th:def:q:i}
and \eqref{eq:th:def:q:ii} to obtain the desired estimate.
\end{proof}

 Lemma~\ref{lemma:main:expressions_c_lambda} shows that the ratio 
$$\frac{\beta_{\mathrm{cmp}}}{\beta_{\mathrm{ch}}} = \dfrac{2 \sum_{\nu,\nu'=1}^4\sigma_\nu \sigma_{\nu'} \alpha_{q;\nu \nu'}^{\diamond \diamond}}{\sum_{\nu=1}^4\sigma_\nu^2 \alpha^\diamond_{p;\nu} + 2 \sum_{\nu,\nu'=1}^4 \sigma_\nu\sigma_{\nu'}} $$
is exactly the value of the coefficient $d$ defined in~\eqref{eqn:main:d}. In particular,
combining 
\eqref{eq:bnd:on:q:vs:q:i}
and \eqref{eqn:analysis_theta:Q_I-beta_I}
we see that
\begin{equation}
\label{eq:th:diff:q:cmp:ch}
    \norm{ \mathcal{Q}_{\mathrm{cmp}}(\theta) - d\mathcal{Q}_{\mathrm{ch}}(\theta) }_{\ell^\infty}
    \le 3 C \norm{ \partial \theta}_{\ell^\infty}
     \norm{ \partial^{(2)} \theta}_{\ell^\infty} ,
\end{equation}
which allows us to establish the following crucial bound.

\begin{cor}
Consider the setting of Proposition~\ref{prop:lde:analysis_theta:differences}.
Then there exists $C > 0$
so that for any $\theta \in \ell^\infty(\Z)$ we have the bound
\begin{equation}
\label{eq:th:bnd:ch:vs:cmp}
\norm{    \Theta_{\mathrm{ch}}(\theta)
    - \Theta_{\mathrm{cmp}}(\theta) }_{\ell^\infty} 
\le C e^{2 N |d|  \norm{\partial \theta}_{\ell^\infty} }
\norm{\partial \theta}_{\ell^\infty}^3
+  C \norm{ \partial \theta}_{\ell^\infty}
     \norm{ \partial^{(2)} \theta}_{\ell^\infty}.
\end{equation}
\end{cor}
\begin{proof}
For $d = 0$ we simply have
\begin{equation*}
    \Theta_{\mathrm{ch}}(\theta) - \Theta_{\mathrm{cmp}}(\theta) = 0.
\end{equation*}
For $d\neq 0$, we note that a Taylor expansion up to third order implies
\begin{equation}\label{eqn:analysis_theta:Theta_ch_Taylor}
    \norm{\Theta_{\mathrm{ch}}(\theta) - \mathcal{H}_{\mathrm{lin}}[\theta]
    - d\mathcal{Q}_{\mathrm{ch}}(\theta) - c_*}_{\ell^\infty}
    \le 
    C e^{ 2 N|d|  \norm{\partial \theta}_{\ell^\infty} }
\norm{\partial \theta}_{\ell^\infty}^3.
\end{equation}
In view of~\eqref{eq:th:diff:q:cmp:ch}, the desired bound
now follows directly 
from the identity
\begin{equation}
\begin{array}{lcl}
    \Theta_{\mathrm{cmp}}(\theta) - \mathcal{H}_{\mathrm{lin}}[\theta]
    - d\mathcal{Q}_{\mathrm{ch}}(\theta)-  c_* &= &  \mathcal{Q}_{\mathrm{cmp}}(\theta) -   d\mathcal{Q}_{\mathrm{ch}}(\theta) .
\end{array}
\end{equation}
\end{proof}

We now turn to our final nonlinearity 
$\widetilde{\Theta}_{\mathrm{dmc}}$
and show that it can be expanded as
\begin{equation}
\label{eq:th:exp:theta:dmc}
\begin{aligned}
\widetilde{\Theta}_{\mathrm{dmc};I}(\theta) &= 
 \sum_{0< |k|\leq N} \left(\frac{2\kappa_H B_k}{k^2} - C_k\dfrac{[\partial_\varphi c_\varphi]_{\varphi = 0}}{k}\right)(\theta_{l+k} - \theta_l) \\
           &\indent 
            +\sum_{0< |k|\leq N}  \dfrac{\left(A_k c_* + C_k [\partial^2_\varphi c_\varphi]_{\varphi = 0} \right)}{2k^2}(\theta_{l+k} - \theta_l)^2 + c_*   ,
\end{aligned}
\end{equation}
up to third order in $\theta$.
This 
is more than sufficient
to establish
Proposition~\ref{prop:main:mcf},
but also allows the relation between
the coefficients to be fully explored
by the interested reader.
For example, in the setting
where $[\partial_\varphi c_{\varphi}]_{\varphi =0} \neq 0$,
it is also possible to prescribe $(B_k)$
and read-off the accompanying
values for $(A_k, C_k)$.
In any case, the conclusions
of Proposition~\ref{prop:main:mcf}
are valid for any sequence
$(C_k)$ that satisfies
\eqref{eq:th:nrm:C}.

\begin{lemma}\label{lemma:analysis_theta:expansion_theta_tilde}
For any sequence $(A_k, B_k,C_k)_{0 < |k| \le N}$,
there exists a constant $K > 0$ so that we have the bound
\begin{equation}\label{eqn:main:Theta_dmc_expanded}
    \begin{aligned}
           \norm{ \widetilde{\Theta}_{\mathrm{dmc}}(\theta)
           - \widetilde{\Theta}_{\mathrm{dmc};I}(\theta)}_{\ell^\infty}
           & \le K \norm{\partial \theta}^3_{\ell^\infty}.
    \end{aligned}
\end{equation}
\end{lemma}
\begin{proof}
Recalling the definitions
\eqref{eq:int:lap:beta:gamma},
we first expand the terms $\beta_\theta $ and $\Delta_{\theta}/\beta_\theta^2 $
as
\begin{equation}\label{eqn:analysis_theta:tilde_theta_exp}
\begin{aligned}
\left[\beta_\theta\right]_l &= 1+ \sum_{0<|k|\leq N} \dfrac{A_k}{2k^2}(\theta_{l+k} - \theta_l)^2 + O(\norm{\partial \theta}^3_{\ell^\infty}), \\  \frac{[{\Delta}_\theta]_l}{[\beta_\theta^2]_l}  &=  \sum_{0 < |k| \le N} \frac{2 B_k}{ k^2}  (\theta_{l+k} - \theta_l) + O(\norm{\partial \theta}^3_{\ell^\infty}). 
    \end{aligned}
\end{equation}
To find a corresponding representation for  $\widetilde{c}_\theta$ we 
first expand each individual term $c_{\varphi_{l;k} (\theta)}$ as
\begin{equation}\label{eqn:main:expansion_for_ck}
   c_{\varphi_{l;k} (\theta)} = c_* + [\partial_\varphi c_\varphi]_{\varphi = 0} \varphi_{l;k} (\theta) + \dfrac{1}{2}[\partial^2_\varphi c_\varphi]_{\varphi = 0} \big(\varphi_{l;k} (\theta)\big)^2 + O(\varphi_{l;k}(\theta)^3).
\end{equation}
Referring to Figure~\ref{fig:intro:zoom},
we use the explicit formula 
\begin{equation*}
    \tan{\varphi_{l;k}(\theta)} = -\dfrac{\theta_{l+k} - \theta_l }{k}
\end{equation*}
together with  the expansion $\tan{\varphi_{l;k}(\theta)} = \varphi_{l;k}(\theta) + O(\varphi_{l;k}(\theta)^3) $ to obtain
\begin{equation*}
   \left[ \widetilde{c}_\theta\right]_l =  c_* - [\partial_\varphi c_\varphi]_{\varphi = 0} \sum_{0<|k|\leq N} \dfrac{\theta_{l+k} - \theta_l }{k}+ \dfrac{1}{2}[\partial^2_\varphi c_\varphi]_{\varphi = 0} \dfrac{(\theta_{l+k} - \theta_l)^2 }{k^2} +O(\norm{\partial \theta}^3_{\ell^\infty}),
\end{equation*}
which yields the desired statement.
%
\end{proof}

\begin{proof}[Proof of Proposition~\ref{prop:main:mcf}]
Recalling that our
assumptions imply that $c_*\neq 0$ and $\kappa_H  \neq 0$,
we may write
\begin{equation}\label{eqn:main:mcf:equations_coefficients_1}
    \begin{array}{clc}
          A_k := & \dfrac{d a_k k^2}{c_*} -  \dfrac{C_k [\partial^2_\varphi c_\varphi]_{\varphi = 0}}{ c_*}, 
       \\[0.4cm]
        B_k:=& \dfrac{a_k k^2}{2\kappa_H} + \dfrac{k C_k [\partial_\varphi c_\varphi]_{\varphi = 0} }{2 \kappa_H}
    \end{array}
\end{equation}
for  $0<|k|\leq N$
and use \eqref{eq:th:exp:theta:dmc}
to conclude
\begin{align*}
    \widetilde{\Theta}_\mathrm{dmc;I}(\theta) &= \sum_{k=-N}^N a_k(\theta_{l+k} - \theta_l) + \sum_{k=-N}^N \dfrac{d}{2} a_k(\theta_{l+k} - \theta_l)^2  + c_* 
    = 
    \mathcal{H}_{\mathrm{lin}}[\theta]
    + d \mathcal{Q}_{\mathrm{ch}}(\theta) + c_*.
    \end{align*}
In particular, the desired bound
follows 
from~\eqref{eqn:analysis_theta:Theta_ch_Taylor} and~\eqref{eqn:main:Theta_dmc_expanded}.

It hence remains to check 
that our coefficients
\eqref{eqn:main:mcf:equations_coefficients_1} satisfy the restrictions
~\eqref{eq:int:nrm:A:C},
which we will achieve
under the general conditions
\eqref{eq:th:nrm:C}.
Employing item \textit{(\ref{item:analyis_theta:partial_lambdax2})} of Lemma~\ref{lemma:main:expressions_c_lambda} we compute
\begin{align*}
     \sum_{0<|k|\leq N} A_k & =  \dfrac{-d[\partial_{\omega}^2 \lambda_\omega]_{\omega = 0} - [\partial^2_\varphi c_\varphi]_{\varphi = 0}}{c_*}.
\end{align*}
This sum is equal to one if and only if the parameter $d$ is chosen as
in~\eqref{eqn:main:d}, which is by straightforward computation equivalent to the definition~\eqref{eq:int:cond:kappa:d}. In a similar fashion, we may
use items~\textit{(\ref{item:analysis_theta:partial_c})}
and ~\textit{(\ref{item:analyis_theta:partial_lambdax2})}
of Lemma~\ref{lemma:main:expressions_c_lambda} to compute
\begin{align*}
     \sum_{0<|k|\leq N} B_k & =  
     -\dfrac{[\partial_{\omega}^2 \lambda_\omega]_{\omega = 0}}{2 \kappa_H}, 
     \qquad 
     \qquad \sum_{0<|k|\leq N} B_k/k = \left(\sum_{0<|k|\leq N} \dfrac{a_k k}{2 \kappa_H} \right)+  \dfrac{ [\partial_\varphi c_\varphi]_{\varphi = 0} }{ 2 \kappa_H} = 0.
\end{align*}
Setting the first sum equal to one leads to the choice~\eqref{eq:int:cond:kappa:d}
for $\kappa_H$.
\end{proof}

\subsection{Cole-Hopf transformation}

We have now collected all the ingredients we need to 
exploit the Cole-Hopf transformation and establish Proposition~\ref{prop:lde:analysis_theta:differences}. The main challenge is to pass difference operators through the relation \eqref{eq:th:rec:theta:from:h}. Proposition~\ref{prop:analysis_theta:theta_approx_quadratic} subsequently follows in a relatively straightforward fashion from the bound
\eqref{eq:th:bnd:ch:vs:cmp}.

\begin{proof}[Proof of Proposition~\ref{prop:lde:analysis_theta:differences}]
Since the function $\tilde{\theta}_l(t) = \theta_l(t) - \theta_0(0)$ also satisfies the first line of~\eqref{eqn:main:main_eqn_for_theta} and this spatially homogeneous shift
is not seen by the difference operators in \eqref{eqn:analysis_theta:differences}-\eqref{eqn:analysis_theta:differences_mn2},
we may assume without loss that $\theta_0(0) = 0$ and consequently $[\theta(0)]_{\mathrm{dev}} = \norm{\theta(0)}_{\ell^\infty} \le R$.
For $d = 0$, we can immediately apply Theorem~\ref{thm:lde:n-th-diff-theta} to function $h(t): = \theta(t) - c_* t $.

For $d\neq 0$, the initial condition
$h(0) = e^{d\theta(0)}$ for
the transformed
system \eqref{eq:th:dot:h:ch}
satisfies the bounds
\begin{equation*}
e^{-|d| R} \leq     \inf_{l\in \Z} h_l(0) \leq \sup_{l\in \Z} h_l(0) \leq e^{|d|R} ,
\qquad \qquad 
\norm{\partial h(0)}_{\infty}
\le e^{|d|R} \delta .
\end{equation*}

We can no longer use the comparison principle to extend these bounds to all $t > 0$ as in \cite{jukic2019dynamics}. Instead, we 
employ Proposition~\ref{prop:lde:comparison_principle}
with the choice
$\varepsilon =   {e^{-2|d|R}}/2$
to find $T = T(R)$ so that 
\begin{align*}
\inf_{l\in \Z} h_l(t) \geq  e^{-|d|R} - \varepsilon e^{|d|R} =   \dfrac{1}{2}e^{-|d|R}
\end{align*}
holds for all $t\geq T$.
On the other hand, using
the constant $C(T, \varepsilon)$ from 
Proposition~\ref{prop:lde:comparison_principle} to write
\begin{equation}
    \delta = \frac{e^{-2 |d|R}}{2C(T,\varepsilon)}
    = \epsilon/C(T,\epsilon),
\end{equation}
we see that \eqref{eqn:lde:cp:flat_diff}
implies that also
\begin{equation}
\inf_{l\in \Z} h_l(t) \geq
e^{-|d|R }
 -  C(T,\varepsilon) e^{|d|R} \delta
 = 
\dfrac{1}{2} e^{-|d|R} 
\end{equation}
for all $0 \le t \leq T$. This provides a uniform strictly positive lower bound for $h$ that is essential for our estimates below.

Turning to \eqref{eqn:analysis_theta:differences}, we pick $l \in \Z$ and use the 
intermediate value theorem
to write
\begin{equation}\label{eqn:lde:bounds:cole_hopf}
    [\partial \theta(t)]_l = \dfrac{1}{d}\dfrac{ [\partial h(t)]_l}{h^a_{l}(t)}, \qquad [\partial^{(2)}\theta(t)]_l = \dfrac{1}{d}\left(\dfrac{[\partial^{(2)} h(t)]_l}{h_{l}(t)}   -\dfrac{[(S-I) h(t)]_{l}^2}{h^b_{l}(t)^2}-\dfrac{[(S^2-I)h(t)]_{l}^2}{2 h^c_{l}(t)^2}\right),
\end{equation}
together with
\begin{align*}
    [\partial^{(3)} \theta(t)]_l &= \dfrac{1}{d} \left(\dfrac{[\partial^{(3)} h(t)]_l }{h_l(t)} + \dfrac{
    [(S^3 - I)h(t)]_l^2 - 3 [(S^2 - I)h(t)]_l^2 + 3[(S - I)h(t)]_l^2 }{2h_l(t)^2}\right) \\
     &+ \quad \dfrac{1}{d} \left(\dfrac{[(S^3 - I)h(t)]_l^3 }{6 h_l^d(t)^3} + \dfrac{[(S^2 - I)h(t)]_l^3 }{6 h_l^e(t)^3} + \dfrac{[(S - I)h(t)]_l^3 }{6 h_l^f(t)^3}\right)
\end{align*}
where we have the inclusions
\begin{equation}
\frac{1}{2} e^{-|d|R}
\le 
\min_{n=0,1,2,3}\{ h_{l+n}(t) \}
\le h_l^a(t),
h_l^b(t),
h_l^c(t),
h_l^d(t),
h_l^e(t),
h_l^f(t)
\le 
\max_{n=0,1,2,3}\{ h_{l+n}(t) \}.
\end{equation}
Applying Theorem~\ref{thm:lde:n-th-diff-theta},
the desired bounds for $k=1,2$ follow directly,
while for $k=3$ it suffices to show that the term
\begin{equation}\label{eqn:function_theta:3rd_diff_aux}
    \tilde{h}(t) =  [(S^3 - I)h(t)]^2 - 3 [(S^2 - I)h(t)]^2 + 3[(S - I)h(t)]^2
\end{equation}
satisfies $||\tilde{h}(t)||_{\ell^\infty} \le M t^{-3/2}$. In view of the decomposition
\begin{align*}
     \tilde{h}(t) &=  [(S^3 - I)h(t)]^2 - 9 [(S-I)h(t)]^2 - 3\big([(S^2 - I)h(t)]^2 - 4[(S-I)h(t)]^2\big) 
\end{align*}
this follows from \eqref{eq:th:bnd:n:m:theta}
and Theorem~\ref{thm:lde:n-th-diff-theta}.
The remaining 
estimates~\eqref{eqn:analysis_theta:differences_mn} and \eqref{eqn:analysis_theta:differences_mn2} now follow directly from 
\eqref{eq:th:bnd:n:m:theta}.
\end{proof}

\begin{proof}[Proof of Proposition~\ref{prop:analysis_theta:theta_approx_quadratic}]
The first bound~\eqref{eqn:super_and_sub_sols_Rtheta_est}
follows directly by combining
\eqref{eq:th:bnd:ch:vs:cmp}
with
Proposition~\ref{prop:lde:analysis_theta:differences}.
To establish~\eqref{eqn:super_and_sub_sols_Rtheta_est:deriv:i} we fix $\nu\in \{1, 2, 3, 4\}$ and exploit the definition~\eqref{eq:th:def:r:theta} to compute
\begin{equation*}
\begin{array}{lcl}
   \pi^\diamond_{\nu}\dot{\theta}(t) - \sum_{\nu'=1}^4 \alpha^\diamond_{p;\nu'} \pi^{\diamond \diamond}_{\nu'\nu}\theta(t)
    &= &
      \pi^\diamond_{\nu} \mathcal{R}_\theta(t)
      + 
      \sum_{\nu',\nu''=1}^4
    \alpha_{p;\nu'\nu''}^{\diamond \diamond} \pi^\diamond_{\nu} [\pi_{ \nu'\nu''}^{\diamond \diamond} \theta(t) ] 
    \\[0.2cm]
    & & \qquad
    + \sum_{\nu',\nu'' = 1}^4 \alpha_{q;\nu'\nu''}^{\diamond \diamond} \pi^\diamond_{\nu} [\pi_{ \nu'}^{\diamond } \theta(t) \pi_{ \nu''}^{\diamond } \theta(t) ] .  
\end{array}
\end{equation*}
The desired bound now follows from
Lemma~\ref{lem:th:bnd:pi:fncs}
in combination with
Proposition~\ref{prop:lde:analysis_theta:differences}.
In a similar fashion,
the final bound~\eqref{eqn:super_and_sub_sols_Rtheta_est:deriv:i}
follows from the observation
\begin{equation*}
\begin{array}{lcl}
    \pi^{\diamond \diamond}_{\nu \nu'} \dot{\theta}(t)
    &= & 
    \pi^{\diamond \diamond}_{\nu \nu'}  \mathcal{R}_\theta(t)
    + \sum_{\nu''=1}^4
    \alpha_{p;\nu''}^{\diamond } \pi^\diamond_{\nu \nu'} [\pi_{ \nu''}^{\diamond } \theta(t) ] 
    \\[0.2cm]
    & & \qquad
    + \sum_{\nu'',\nu'''=1}^4 \alpha_{p;\nu''\nu'''}^{\diamond \diamond} \pi^{\diamond \diamond}_{\nu \nu'} [\pi^{\diamond \diamond}_{\nu'' \nu'''}\theta(t)]  
    \\[0.2cm]
    & &
    \qquad
    + \sum_{\nu'',\nu'''=1}^4 \alpha_{q;\nu''\nu'''}^{\diamond \diamond} \pi^{\diamond \diamond}_{\nu \nu'} [\pi_{ \nu''}^{\diamond } \theta(t) \pi_{l; \nu'''}^{\diamond } \theta(t) ] .
\end{array}
\end{equation*}
\end{proof}

\section{Construction of super- and sub-solutions}
\label{sec:sub:sup}
The main aim of this section is to construct explicit super- and sub-solutions 
for the discrete Allen-Cahn equation~\eqref{eqn:main:discrete_AC_new},
using the function $\theta$ introduced in Theorem~\ref{thm:main_result:gamma_mcf}. 
To be more precise, 
for any $u\in C^1([0, \infty), \ell^\infty(\Zs)\big)$
we define the residual
\begin{equation*}
  \mathcal{J}[u](t) = \dot{u}(t) - [\Delta^\times u(t)] - g\big(u(t)\big)
\end{equation*}
and say that $u$ is a super- respectively sub-solution for~\eqref{eqn:main:discrete_AC_new} if the inequality $\mathcal{J}[u]_{n,l}(t)\geq 0$, respectively $\mathcal{J}[u]_{n,l}(t)\leq 0$ holds for all $(n,l)\in \Zs$ and $t\geq 0$. Our construction utilizes the functions 
introduced in Lemma \ref{lemma:main_results:coeff_diamond} together with the difference operators defined in
\eqref{eq:th:def:pi:nu} and \eqref{eq:th:def:pi:nu:nup}. The main difference compared to our earlier work \cite{jukic2019dynamics} and the PDE results in \cite{NARAMATANO2011} is that a significant number of additional terms are needed to control the anisotropic effects caused by the misalignment of our wave with the underlying lattice.

\begin{proposition}
\label{prop:super_and_sub_sol:supersolution_proposition} 
Fix $R > 0$  and suppose that
the assumptions (H$g$), (H$\Phi_*$),  (HS$)_1$
and (HS$)_2$
all hold. Then for any $\epsilon > 0$, there exist  constants $\delta>0$, $\nu > 0$ and  $C^1$-smooth functions
\begin{equation}
    z: [0, \infty) \to \R,
    \qquad
    Z: [0, \infty) \to \R
\end{equation}
so that for any $\theta^0 \in \ell^\infty(\Z)$ with
\begin{equation}
[\theta^0]_{\mathrm{dev}} < R,
\qquad \qquad
\norm{\partial \theta^0 }_{\ell^\infty}  < \delta
\end{equation}
the following holds true.

\begin{enumerate}[(i)]
\item\label{item:super_sub_sols:explicit_formulas}
Writing  $\theta: [0, \infty) \to \ell^\infty(\Z)$ for the solution to
\eqref{eqn:main:main_eqn_for_theta}
with the initial condition $\theta(0) = \theta^0$,
the function $u^+$ defined by
	\begin{equation}\label{eqn:super_and_sub_sol:def_u_plus}
	\begin{aligned}
    u^+_{n,l}(t) &= \Phi_*\big(n-\theta_l(t) + Z(t)\big) + \pi^\diamond_{l;\nu}\theta(t) p^\diamond_\nu\big(n-\theta_l(t) + Z(t)\big) + \pi^{\diamond \diamond}_{l;\nu\nu'}\theta(t) p^{\diamond \diamond}_{\nu \nu'}\big(n-\theta_l(t) + Z(t)\big) \\
    &\qquad \qquad  +  \pi^\diamond_{l;\nu}\theta(t) \pi^\diamond_{l;\nu'}\theta(t)q^{\diamond \diamond}_{\nu \nu'}\big(n-\theta_l(t) + Z(t)\big)  + z(t)
	\end{aligned}
\end{equation}
is a super-solution of 
\eqref{eqn:main:AC_equation},
while
the function $u^-$ defined by
\begin{equation}\label{eqn:super_and_sub_sol:def_u_minus}
	\begin{aligned}
    u^-_{n,l}(t) &= \Phi_*\big(n-\theta_l(t) - Z(t)\big) + \pi^\diamond_{l;\nu}\theta(t) p^\diamond_\nu\big(n-\theta_l(t) - Z(t)\big) + \pi^{\diamond \diamond}_{l;\nu\nu'}\theta(t) p^{\diamond \diamond}_{\nu \nu'}\big(n-\theta_l(t) - Z(t)\big) \\
    &\qquad \qquad +  \pi^\diamond_{l;\nu}\theta(t) \pi^\diamond_{l;\nu'}\theta(t)q^{\diamond \diamond}_{\nu \nu'}\big(n-\theta_l(t) - Z(t)\big)  - z(t)
	\end{aligned}
\end{equation}
is a sub-solution of \eqref{eqn:main:AC_equation}. %
	\item\label{item:super_sub_sols:Z0} 
	We have $Z(0) = 0$ together with the bound $0\leq Z(t) \leq \epsilon$ for all $t \ge 0$.
	\item\label{item:super_sub_sols:z0_nu} We have the bound $0\leq z(t) \leq \epsilon $ for all $t \ge 0$,
	together with the initial inequalities
		\begin{equation}
	  \label{eq:int:sub:sup:bnd:init:p:zero}
   z(0) - \delta \norm{p_\nu^\diamond}_{L^\infty} - 2\delta \norm{p_{\nu \nu'}^{\diamond\diamond}}_{L^\infty} - \delta^2 \norm{q_{\nu \nu'}^{\diamond\diamond}}_{L^\infty} > \nu > 0 .
	\end{equation}
	\item\label{item:super_sub_sols:asymp} 
	The asymptotic behaviour
	$z (t) = O ( t^{-\frac{3}{2}} ) $
	holds for $t \to \infty$.
\end{enumerate}
In addition, the constants $\nu = \nu(\epsilon)$
satisfy $\lim_{\epsilon\downarrow 0} \nu(\epsilon) = 0$.
\end{proposition}

\subsection{Preliminaries}
Our proof of Proposition~\ref{prop:super_and_sub_sol:supersolution_proposition} focuses on the analysis of
the super-solution residual
$\mathcal{J}[u^{+}]$,
since the sub-solution 
$\mathcal{J}[u^{-}]$
can be analyzed completely analogously.  We start by splitting the residual
into the five components
\begin{align}\label{eqn:super_and_sub_sols:initial_resdual}
 \mathcal{J}[u^+]  = \mathcal{J}_{\Phi} + \mathcal{J}_{p_\nu^\diamond} + \mathcal{J}_{p_{\nu\nu'}^{\diamond \diamond}} + \mathcal{J}_{q_{\nu\nu'}^{\diamond \diamond}} + \mathcal{J}_{\text{glb}}.
 \end{align}
The first four are closely related
to the functions
$\Phi_*$, $p_\nu^\diamond$, $p_{\nu \nu'}^{\diamond \diamond}$ and $q_{\nu \nu'}^{\diamond \diamond}$, respectively, depending
on $Z$ only through the variable
\begin{equation*}
    \xi_{n,l}(t) = n - \theta_l(t) + Z(t).
\end{equation*}
Indeed, we use the definitions

\begin{equation}\label{eqn:super_sub_sols:all_Js}
    \begin{aligned}
    \left[\mathcal{J}_{\Phi}(t)\right]_{n,l} 
    &= -\Phi_*'\big(\xi_{n,l}(t)\big) \dot{\theta_l}(t) -  \left[\Delta^\times \Phi_*\big(\xi(t)\big)\right]_{n,l} ,   \\[0.3cm]
    [\mathcal{J}_{p_\nu^\diamond}(t)]_{n,l} 
    &
    =  [\pi^\diamond_{l;\nu}\dot{\theta}(t)] p^\diamond_\nu(\xi_{n,l}(t)) - [\pi^\diamond_{l;\nu}\theta(t)] \dfrac{d}{d\xi}p^\diamond_\nu\big(\xi_{n,l}(t)\big)\dot{\theta_l}(t)  - \left[  \Delta^\times[\pi^\diamond_{\nu}\theta(t)] p^\diamond_\nu\big(\xi(t)\big)\right]_{n,l}, 
    \\[0.3cm]
    [\mathcal{J}_{p_{\nu\nu'}^{\diamond \diamond}}(t)]_{n,l} 
    &=
    [\pi^{\diamond\diamond}_{l;\nu\nu'}\dot{\theta}(t)] p^{\diamond\diamond}_{\nu\nu'}\big(\xi_{n,l}(t)\big) -  \dot{\theta_l}(t) [\pi^{\diamond\diamond}_{l;\nu \nu'}\theta(t)], \dfrac{d}{d\xi}p^{\diamond\diamond}_{\nu \nu'}\big(\xi_{n,l}(t)\big) \\
    &\indent  - \left[\Delta^\times [\pi^{\diamond \diamond}_{\nu\nu'}\theta(t)] p^{\diamond \diamond}_{\nu\nu'}\big(\xi(t)\big)\right]_{n,l}, 
    \\[0.3cm]
    [\mathcal{J}_{q_{\nu\nu'}^{\diamond \diamond}}(t)]_{n,l} 
    &
    = [\pi^{\diamond}_{l;\nu}\dot{\theta}(t)][\pi^{\diamond}_{l;\nu'}\theta(t)] q^{\diamond\diamond}_{\nu\nu'}\big(\xi_{n,l}(t)\big) +  [\pi^{\diamond}_{l;\nu}\theta(t)][\pi^{\diamond}_{l;\nu'}\dot{\theta}(t)] q^{\diamond\diamond}_{\nu\nu'}\big(\xi_{n,l}(t)\big) \\
    &\indent  -
    [\pi^{\diamond}_{l;\nu}\theta(t)][\pi^{\diamond}_{l;\nu'}\theta(t)] \dfrac{d}{d\xi}q^{\diamond\diamond}_{\nu \nu'}\big(\xi_{n,l}(t)\big)\dot{\theta_l}(t)     
     - \left[\Delta^\times [\pi^{\diamond }_{\nu}\theta(t)][ \pi^{\diamond}_{\nu'}\theta(t)] q^{\diamond\diamond}_{\nu\nu'}\big(\xi(t)\big)\right]_{n,l}. 
    \\
\end{aligned}
\end{equation}
On the other hand, 
we group the terms related to
the nonlinearity $g$
and the dynamics of the functions $Z$ and $z$ 
into the global term
\begin{equation}\label{eqn:super_sub_sols:J_glb}
    [\mathcal{J}_{\text{glb}}(t)]_{n,l}  = \dot{Z}(t)\Big( \Phi_*'\big(\xi_{n,l}(t)\big) + B_{n,l}(t) \Big)
     + \dot{z}(t) - g\big(u^+_{n,l}(t)\big),
\end{equation}
where  $B_{n,l}(t)$ denotes the bounded sequence
\begin{equation}
    B_{n,l}(t) :=  [\pi^\diamond_{l;\nu}\theta(t)] \dfrac{d}{d\xi}p^\diamond_\nu\big(\xi_{n,l}(t)\big) + [\pi^{\diamond\diamond}_{l;\nu \nu'}\theta(t)] \dfrac{d}{d\xi}p^{\diamond\diamond}_{\nu \nu'}\big(\xi_{n,l}(t)\big) + 
    [\pi^{\diamond}_{l;\nu}\theta(t)][\pi^{\diamond}_{l;\nu'}\theta(t)] \dfrac{d}{d\xi}q^{\diamond\diamond}_{\nu \nu'}\big(\xi_{n,l}(t)\big).
\end{equation}
In the first phase of our analysis we split each of these terms into
a useful approximation
together with a residual that contains terms that behave asymptotically as $O(t^{-\frac{3}{2}})$.
In the second phase we combine these approximations, allowing us to isolate an additional set of higher order terms and extract our final approximation.

\subsection{Analysis of \texorpdfstring{$\mathcal{J}_{\Phi}$}{Lg}} \label{subsec:super_sub_sols:J_phi}
Setting out to analyze the term $\mathcal{J}_{\Phi}$,
we introduce the 
approximation
\begin{equation}\label{eqn:super_sub_sol:J_phi_apx}
    \begin{aligned}
    \left[\mathcal{J}_{\Phi;\text{apx}}(t)\right]_{n,l} &:= \Phi_*'\big(\xi_{n,l}(t)\big)(-\dot{\theta}_l(t) + c_*) + g\Big(\Phi_*\big(\xi_{n,l}(t)\big)\Big)  + [\pi^\diamond_{l;\nu}\theta(t)] [T_\nu \Phi_*']\big(\xi_{n,l}(t)\big)  \\
    &\indent  - \dfrac{1}{2} [\pi^\diamond_{l;\nu}\theta(t)]^2 [T_\nu \Phi_*'']\big(\xi_{n,l}(t)\big) \\[0.3cm]
\end{aligned}
\end{equation}
and implicitly define the residual $\mathcal{R}_{\Phi}$ via the splitting
\begin{equation}\label{eqn:super_sub_sol:J_phi=J_phi_apx+R}
    \mathcal{J}_{\Phi} = \mathcal{J}_{\Phi;\text{apx}} + \mathcal{R}_{\Phi}.
\end{equation}
The result below confirms that the expression $\mathcal{R}_{\Phi}$ contains only higher order terms,
allowing us to focus on
$\mathcal{J}_{\Phi;\text{apx}}$
for our further computations.


\begin{lemma}\label{lemma:super_and_sub_sols:RPhi}
Assume the setting of Proposition~\ref{prop:super_and_sub_sol:supersolution_proposition}. 
There exists constants $M > 0$, $\delta>0$
so that for any $\theta \in C^1([0,\infty);\ell^\infty)$
that satisfies the LDE~\eqref{eqn:main:main_eqn_for_theta}
with 
$[\theta(0)]_{\mathrm{dev}} < R$ and $\norm{\partial \theta(0)}_{\ell^\infty}\leq \delta $
and any pair of  functions 
$z, Z \in C([0,\infty);\R)$,
we have the estimate
\begin{equation}\label{eqn:super_and_sub_sols_RPhi_est}
    \norm{\mathcal{R}_{\Phi}(t)}_{\ell^\infty} \le M \min\left\{\norm{\partial \theta(0)}_{\ell^\infty}, t^{-\frac{3}{2}}\right\},
    \qquad 
    \qquad
    t > 0.
\end{equation}
\end{lemma}
\begin{proof}
Proceeding from the definition 
\begin{equation} \label{eqn:super_sub_sols:laplacian_phi}
    [\Delta^\times \Phi_*\left(\xi\left(t\right)\right)]_{n,l}  = 
    [T_\nu \Phi_*] \big(\xi_{n,l+\sigma_\nu}(t)\big) - 4\Phi_*\big(\xi_{n,l}(t)\big),
\end{equation}    
 we expand $[T_\nu \Phi_*] \big(\xi_{n,l+\sigma_\nu}(t)\big)$ around $[T_\nu \Phi] \big(\xi_{n,l}(t)\big)$
 to find
 \begin{align*}
      [\Delta^\times \Phi_*(\xi(t))]_{n,l}  &= \Big([\tau_\nu \Phi_*] \big(\xi_{n,l}(t)\big)  - 4 \Phi_* \big(\xi_{n,l}(t)\big) \Big) - [\pi^\diamond_{l;\nu}\theta(t)]  [T_\nu \Phi_*']\big(\xi_{n,l}(t)\big)  
     \\
     & \qquad + \dfrac{1}{2}[\pi^\diamond_{l;\nu}\theta(t)]^2  [T_\nu \Phi_*'']\big(\xi_{n,l}(t)\big)
    + \dfrac{1}{2}\int_{0}^{-\pi_{l;\nu}^\diamond \theta(t)} [T_\nu \Phi_*'''](\xi_{n, l+\sigma_\nu}(t) - s) s^2\, ds.
 \end{align*}
 Inserting this expression into the definition~\eqref{eqn:super_sub_sols:all_Js} for $\mathcal{J}_{\Phi}$, we arrive at
 \begin{align*}
     [\mathcal{J}_\Phi]_{n,l}(t) - [\mathcal{J}_{\Phi; \text{apx}}]_{n,l}(t) &=
     -\Big[  c_* \Phi_*'\big(\xi_{n,l}(t)\big) + [T_\nu \Phi_*] \big(\xi_{n,l}(t)\big)  - 4 \Phi_* \big(\xi_{n,l}(t)\big) + g\Big(\Phi_*\big(\xi_{n,l}(t)\big)\Big) \Big]\\
    &\indent   - \dfrac{1}{2}\int_{0}^{-\pi_{l;\nu}^\diamond \theta(t)} [T_\nu \Phi_*''']\big(\xi_{n, l+\sigma_\nu}(t) - s\big) s^2\, ds .
 \end{align*}
The first row vanishes due to the MFDE~\eqref{eqn:main:mdfe}, while the second row satisfies the desired bound
~\eqref{eqn:super_and_sub_sols_RPhi_est} by combination of 
the estimate~\eqref{eqn:analysis_theta:first_diff_nu} with Proposition
\ref{prop:lde:analysis_theta:differences}.  
\end{proof}

\subsection{Analysis of \texorpdfstring{$\mathcal{J}_{p_\nu^\diamond}$}{Lg}}
In this subsection we fix $\nu \in \{1, 2, 3, 4\}$ and analyze the function $\mathcal{J}_{p_\nu^\diamond}$. 
In particular, we introduce the expression

\begin{align}
    [\mathcal{J}_{p_{\nu}^{\diamond};\text{apx}}]_{n,l}(t)
    &:= 
    [\pi^\diamond_{l;\nu}\theta(t)] \Big[- [\mathcal{L}_0 p_\nu^\diamond]\big(\xi_{n,l}(t)\big) +  g'\Big(\Phi_*\big(\xi_{n,l}(t)\big)\Big)
    p_{\nu}^\diamond\big(\xi_{n,l}(t)\big)\Big]
    \\
    &\indent  + [ \pi^{\diamond \diamond}_{l;\nu \nu' }\theta(t)] \Big( \alpha_{p;\nu}^\diamond p_{\nu'}^\diamond \big(\xi_{n,l}(t)\big) - [T_{\nu'} p_{\nu}^\diamond]\big(\xi_{n,l}(t)\big)\Big) \\
   &\indent 
    + [ \pi^\diamond_{l;\nu}\theta(t)] [ \pi^\diamond_{l;\nu'}\theta(t)] \left(      -    \alpha_{p;\nu'}^{\diamond} \dfrac{d}{d\xi}p_{\nu}^{\diamond}\big(\xi_{n,l}(t)\big)  +
 \dfrac{d}{d\xi} [T_{\nu'} p_{\nu}^\diamond] \big(\xi_{n,l}(t)\big) 
    \right)
\label{eq:sub:sup:def:j:p:nu:apx}
\end{align}
and set out to obtain bounds
for the residual
\begin{equation}\label{eqn:super_sub_sol:R_p_nu}
\mathcal{R}_{p_\nu^\diamond}:= \mathcal{J}_{p_{\nu}^{\diamond}} - \mathcal{J}_{p_{\nu}^{\diamond};\text{apx}}.
\end{equation}
\begin{lemma}\label{lemma:super_and_sub_sols:Rp_nu_diamond}
Consider the setting of Proposition~\ref{prop:super_and_sub_sol:supersolution_proposition}.
Then there exist  constants $\delta>0$ and $M > 0$
so that for any $\theta \in C^1([0,\infty);\ell^\infty)$
that satisfies the LDE~\eqref{eqn:main:main_eqn_for_theta}
with 
$[\theta(0)]_{\mathrm{dev}} < R$ and $\norm{\partial \theta(0)}_{\ell^\infty}\leq \delta $
and any pair of  functions 
$z, Z \in C([0,\infty);\R)$,
we have the estimate
\begin{equation}\label{eqn:super_and_sub_sols_Rp_nu_est}
    \norm{\mathcal{R}_{p_\nu^\diamond}(t)}_{\ell^\infty} \le M \min\left\{\norm{\partial \theta(0)}_{\ell^\infty}, t^{-\frac{3}{2}}\right\}, \qquad
    t > 0.
\end{equation}
\end{lemma}

\begin{proof}
For fixed $\nu\in \{1, 2, 3, 4\}$, we rewrite the discrete Laplacian as
\begin{align}\label{eqn:super_sub_sol:laplacian_p_nu}
\Delta^\times\left[ [\pi^\diamond_{\nu}\theta(t)] p^\diamond_\nu\big(\xi(t)\big)\right]_{n,l} &
=  [\pi_{l;\nu+ {\nu'}} \theta(t)] [T_{\nu'} p_{\nu}^\diamond] \big((\xi_{n,l+\sigma_{\nu'}}(t)\big) -  4[\pi^\diamond_{l;\nu} \theta(t)] p^\diamond_\nu\big(\xi_{n,l}(t)\big),
\end{align}
where the first term is summed
over the indices $\nu'\in \{1, 2, 3, 4\}$. Using~\eqref{eqn:super_sub_sol:2nd_diff_via_1st_diff} we rephrase this term as
\begin{equation}\label{eqn:super_sub_sols:J_p_nu:laplacian}
\begin{aligned}
     \left[\pi_{l;\nu+\nu'}^{\diamond} \theta(t)\right] [T_{\nu'} p_{\nu}^\diamond] \big((\xi_{n,l+\sigma_{\nu'}}(t)\big) 
     & 
     =  [\pi_{l;\nu}^{\diamond} \theta(t)] [T_{\nu'} p_{\nu}^\diamond] \left(\xi_{n,l+\sigma_{\nu'}}(t)\right) \\
     &\indent  + [\pi_{l;\nu \nu'}^{\diamond \diamond} \theta(t)] [T_{\nu'} p_{\nu}^\diamond] \big(\xi_{n,l+\sigma_{\nu'}}(t)\big) .
\end{aligned}
\end{equation}
Furthermore, for a fixed pair $(\nu, \nu')$ we expand $[T_{\nu'} p_{\nu}^\diamond] \left(\xi_{n,l+\sigma_{\nu'}}(t)\right) $ around $[T_{\nu'} p_{\nu}^\diamond] \big(\xi_{n,l}(t)\big)$ to find
\begin{align*}
    [T_{\nu'} p_{\nu}^\diamond] \big(\xi_{n,l+\sigma_{\nu'}}(t)\big) & = [T_{\nu'} p_{\nu}^\diamond] \big(\xi_{n,l}(t)\big) - [\pi_{l;\nu'}^{\diamond}\theta(t)]  \dfrac{d}{d\xi} [T_{\nu'} p_{\nu}^\diamond] \big(\xi_{n,l}(t)\big) \\
    &
    \quad + \int_{0}^{-\pi_{l;\nu'}^\diamond\theta(t)} \dfrac{d^2}{d \xi ^2}[T_{\nu'} p_{\nu}^\diamond] \left(\xi_{n,l+\sigma_{\nu'}}(t) - s \right) s\, ds.
\end{align*}
Inserting this expression back into~\eqref{eqn:super_sub_sols:J_p_nu:laplacian}, we obtain
\begin{align*}
      \left[\pi_{l;\nu+\nu'}^{\diamond} \theta(t)\right] [T_{\nu'} p_{\nu}^\diamond] \big((\xi_{n,l+\sigma_{\nu'}}(t)\big) 
     &
     =  [\pi^\diamond_{l;\nu} \theta(t)] [T_{\nu'} p_{\nu}^\diamond] \big((\xi_{n,l}(t)\big) 
     \\
     & \quad - [\pi^\diamond_{l;\nu} \theta(t)] [\pi^\diamond_{l;\nu'} \theta(t)] \dfrac{d}{d\xi} [T_{\nu'} p_{\nu}^\diamond] \big(\xi_{n,l}(t)\big) 
     \\
     & \quad + [\pi_{l;\nu \nu'}^{\diamond \diamond} \theta(t)] [T_{\nu'} p_{\nu}^\diamond] \big(\xi_{n,l}(t)\big) +  [\tilde{\mathcal{R}}_{p_{\nu}^\diamond}(t)]_{n,l}
\end{align*}
where $
{\tilde{\mathcal{R}}}_{p_{\nu}^\diamond}$ is defined as
\begin{align*}
    [\tilde{\mathcal{R}}_{p_{\nu}^\diamond}(t)]_{n,l} &= \left([\pi_{l;\nu}^\diamond \theta(t)] + [\pi_{l;\nu \nu'}^{\diamond\diamond} \theta(t)]\right) \int_{0}^{-\pi_{l;\nu'}^\diamond\theta(t)} \dfrac{d^2}{d\xi^2}[T_{\nu'} p_{\nu}^\diamond] \left(\xi_{n,l+\sigma_{\nu'}}(t) - u \right) u \ du \\
     & \qquad 
     + [\pi_{l;\nu \nu'}^{\diamond \diamond} \theta(t)] [\pi_{l;\nu'}^{\diamond}\theta(t)]  \dfrac{d}{d\xi} [T_{\nu'} p_{\nu}^\diamond]  \big(\xi_{n,l}(t)\big).
\end{align*}
Note that this term satisfies the bound~\eqref{eqn:super_and_sub_sols_Rp_nu_est}
in view of 
Proposition 
\ref{prop:lde:analysis_theta:differences} and Lemma~\ref{lem:th:bnd:pi:fncs}.

Comparing
\eqref{eqn:super_sub_sols:all_Js}
and \eqref{eq:sub:sup:def:j:p:nu:apx},
the definition \eqref{eqn:super_sub_sol:R_p_nu}
hence yields

\begin{align*}
     [\mathcal{R}_{p_{\nu}^\diamond}(t)]_{n,l} 
     &= \Big([\pi^\diamond_{l;\nu}\dot{\theta}(t)] - [ \pi^{\diamond \diamond}_{l;\nu' \nu }\theta(t)]  \alpha_{p;\nu'}^\diamond \Big)p^\diamond_\nu\big(\xi_{n,l}(t)\big) 
     \\
     &\indent  - [\pi^\diamond_{l;\nu}\theta(t)]  \Big(\dot{\theta_l}(t)   -   \alpha_{p;\nu'}^{\diamond}[ \pi^\diamond_{l;\nu'}\theta(t)] \Big) \dfrac{d}{d\xi}p^\diamond_\nu\big(\xi_{n,l}(t)\big)
     \\
      & \indent - [\pi^\diamond_{l;\nu} \theta(t)] \Big[ [T_{\nu'} p_{\nu}^\diamond]  \big(\xi_{n,l}(t) \big) - 4p_{\nu}^\diamond  \big(\xi_{n,l}(t) \big) + g'\Big(\Phi_*\big(\xi_{n,l}(t)\big)\Big)
      p_{\nu}^\diamond \big(\xi_{n,l}(t)\big) -  [\mathcal{L}_0 p_\nu^\diamond]\big(\xi_{n,l}(t)\big) \Big] 
     \\
      &\indent   - [\tilde{\mathcal{R}}_{p_{\nu}^\diamond}(t)]_{n,l}.
\end{align*}
The first row satisfies our estimate~\eqref{eqn:super_and_sub_sols_Rp_nu_est} due to the bound~\eqref{eqn:super_and_sub_sols_Rtheta_est:deriv:i} of Proposition~\ref{prop:analysis_theta:theta_approx_quadratic}.
Similarly, for the second row we can apply
\eqref{eqn:super_and_sub_sols_Rtheta_est}
in combination with
Proposition 
\ref{prop:lde:analysis_theta:differences} and Lemma~\ref{lem:th:bnd:pi:fncs}.
The third row vanishes in view of
the definition \eqref{eq:mr:def:l:omega} for
the operator $\mathcal{L}_0$,
which completes the proof. 
\end{proof}

 \subsection{Analysis of \texorpdfstring{$\mathcal{J}_{p_{\nu \nu'}^{\diamond \diamond}}$}{Lg}
 and \texorpdfstring{$\mathcal{J}_{q_{\nu\nu'}^{\diamond \diamond}}$ }{Lg}
 }\label{subsec:super_sub_sols:J_qnux2}
 
Throughout this subsection,
we fix a  pair $(\nu, \nu')\in \{1, 2, 3, 4\}^2$ and study the approximants
\begin{equation}\label{eqn:super_sub_sols:J_p_nu_nu'}
     \ [\mathcal{J}_{p_{\nu \nu'}^{\diamond \diamond};\text{apx}}(t) ]_{n,l}:= [\pi_{\nu \nu'}^{\diamond \diamond}\theta(t)] \Big( -[\mathcal{L}_0 p^{\diamond \diamond}_{\nu\nu'}]\big(\xi_{n,l}(t)\big)  + g'\big(\Phi_*\big(\xi_{n,l} (t)\big)\big) p^{\diamond \diamond}_{\nu\nu'}\big(\xi_{n,l}(t)\big) \Big)
\end{equation}
together with
\begin{align*}
    [\mathcal{J}_{q_{\nu \nu'}^{\diamond \diamond};\text{apx}}]_{n,l}(t):=
  [\pi^{\diamond}_{l;\nu}\theta(t)][\pi^{\diamond}_{l;\nu'}\theta(t)] \Big(-[\mathcal{L}_0 q^{\diamond \diamond}_{\nu\nu'}]\big(\xi_{n,l}(t)\big)  + g'\big(\Phi_*\big(\xi_{n,l} (t)\big)\big) q^{\diamond \diamond}_{\nu\nu'}\big(\xi_{n,l}(t)\big)\Big) .
\end{align*}
In particular, we obtain
bounds on the residuals
\begin{equation}\label{eqn:super_sub_sols:R_q_nu_nu'}
     \mathcal{R}_{p_{\nu \nu'}^{\diamond \diamond}}:=\mathcal{J}_{p_{\nu \nu'}^{\diamond \diamond}}- \mathcal{J}_{p_{\nu \nu'}^{\diamond \diamond};\text{apx}},
     \qquad \qquad
        \mathcal{R}_{q_{\nu \nu'}^{\diamond \diamond}}:=\mathcal{J}_{q_{\nu \nu'}^{\diamond \diamond}}- \mathcal{J}_{q_{\nu \nu'}^{\diamond \diamond};\text{apx}}.
\end{equation}

\begin{lemma}\label{lemma:super_and_sub_sols:Rp_nu_diamondx2}
Consider the setting of Proposition~\ref{prop:super_and_sub_sol:supersolution_proposition}.
Then there exist constants $\delta>0$ and $M > 0$
so that for any $\theta \in C^1([0,\infty);\ell^\infty(\Z))$
that satisfies the LDE~\eqref{eqn:main:main_eqn_for_theta}
with 
$[\theta(0)]_{\mathrm{dev}} < R$ and $\norm{\partial \theta(0)}_{\ell^\infty}\leq \delta $
and any pair of  functions 
$z, Z \in C([0,\infty);\R)$,
we have the estimate
\begin{equation}\label{eqn:super_and_sub_sols_Rnux2_est}
    \norm{\mathcal{R}_{p_{\nu \nu'}^{\diamond \diamond}}(t)}_{\ell^\infty} \le M \min\left\{\norm{\partial \theta(0)}_{\ell^\infty}, t^{-\frac{3}{2}}\right\}, \qquad
    t > 0.
\end{equation}
\end{lemma}
\begin{proof}
For a fixed pair $(\nu, \nu') \in \left\{1, 2, 3, 4\right\}^2$, we write the discrete Laplacian in the form
\begin{align*}
    \Delta^\times\Big[ [ \pi^{\diamond \diamond}_{\nu\nu'}\theta(t) ] p^{\diamond \diamond}_{\nu\nu'}\big(\xi(t)\big)\Big]_{n,l}
    & = \pi^{\diamond \diamond}_{l+\sigma_{\nu''};\nu\nu'}\theta(t) [\tau_{\nu''}p^{\diamond \diamond}_{\nu\nu'}]\big(\xi_{n,l+\sigma_{\nu''}} (t)\big)     - 4 \pi^{\diamond \diamond}_{l;\nu\nu'}\theta(t) p^{\diamond \diamond}_{\nu\nu'}\big(\xi_{n,l}(t)\big),
\end{align*}
summing the first term over the indices $\nu''\in \{1, 2 , 3, 4\}$.
Adding and subtracting  $[\pi^{\diamond \diamond}_{l;\nu\nu'}\theta(t)] [\tau_{\nu''}p^{\diamond \diamond}_{\nu\nu'}]\big(\xi_{n,l+\sigma_{\nu''}} (t)\big) $ 
while also  expanding $[\tau_{\nu''}p^{\diamond \diamond}_{\nu\nu'}]\big(\xi_{n,l+\sigma_{\nu''}} (t)\big)$ around $[\tau_{\nu''}p^{\diamond \diamond}_{\nu\nu'}]\big(\xi_{n,l} (t)\big)$,
we obtain
\begin{equation}\label{eqn:super_and_sub_sols:laplace_p_nunu'}
    \begin{aligned}
    \Delta^\times\Big[ [ \pi^{\diamond \diamond}_{\nu\nu'}\theta(t) ] p^{\diamond \diamond}_{\nu\nu'}\big(\xi(t)\big)\Big]_{n,l}
      &= [\pi^{\diamond \diamond}_{l;\nu\nu'}\theta(t)]\Big( [T_{\nu''} p^{\diamond \diamond}_{\nu\nu'}]\big(\xi_{n,l+\sigma_{\nu''}} (t)\big) - 4  p^{\diamond \diamond}_{\nu\nu'}\big(\xi_{n,l}(t)\big)\Big)
      \\ 
    &\indent  + [\pi^{\diamond}_{l;\nu''} \pi^{\diamond \diamond}_{l;\nu\nu'}\theta(t) ] [T_{\nu''} p^{\diamond \diamond}_{\nu\nu'}]\big(\xi_{n,l+\sigma_{\nu''}} (t)\big) \\
    &=[\pi^{\diamond \diamond}_{l;\nu\nu'}\theta(t)]\Big( [T_{\nu''}  [p^{\diamond \diamond}_{\nu\nu'}]\big(\xi_{n,l}(t)\big)  - 4 p^{\diamond \diamond}_{\nu\nu'}\big(\xi_{n,l}(t)\big)
    \Big)
    +  [\tilde{\mathcal{R}}_{p_{\nu \nu'}^{\diamond \diamond}}(t)]_{n,l}.
\end{aligned}
\end{equation}
Here $ \tilde{\mathcal{R}}_{p_{\nu \nu'}^{\diamond \diamond}}(t)$ is equal to
\begin{align*}
   [ \tilde{\mathcal{R}}_{p_{\nu \nu'}^{\diamond \diamond}} (t)]_{n,l} &= [\pi^{\diamond}_{l;\nu''} \pi^{\diamond \diamond}_{l;\nu\nu'}\theta(t) ] [T_{\nu''} p^{\diamond \diamond}_{\nu\nu'}]\big(\xi_{n,l+\sigma_{\nu''}} (t)\big)  \\
   &\indent  + \pi^{\diamond \diamond}_{l;\nu\nu'}\theta(t) \int_0^{-\pi_{l;\nu'}\theta(t)} \frac{d}{d\xi} [T_{\nu''} p_{\nu \nu'}^{\diamond \diamond}](\xi_{n,l+\sigma_{\nu'}}(t) - u)\, du,
\end{align*}
which satisfies the bound~\eqref{eqn:super_and_sub_sols_Rnux2_est}
on account of
Proposition 
\ref{prop:lde:analysis_theta:differences} and Lemma~\ref{lem:th:bnd:pi:fncs}.

Combining~\eqref{eqn:super_and_sub_sols:laplace_p_nunu'} with the definitions~$\eqref{eqn:super_sub_sols:all_Js}_3$ and \eqref{eqn:super_sub_sols:J_p_nu_nu'} yields
\begin{align*}
    [{\mathcal{R}}_{p_{\nu \nu'}^{\diamond \diamond}}]_{n,l}(t) & = [\pi^{\diamond\diamond}_{l;\nu\nu'}\dot{\theta}(t)] p^{\diamond\diamond}_{\nu\nu'}\big(\xi_{n,l}(t)\big) - [\pi^{\diamond\diamond}_{l;\nu \nu'}\theta(t)] \dot{\theta_l}(t)   \dfrac{d}{d\xi}p^{\diamond\diamond}_{\nu \nu'}\big(\xi_{n,l}(t)\big)
    \\
    &\qquad  + [\pi^{\diamond\diamond}_{l;\nu \nu'}\theta(t)] \Big(  -[T_{\nu''}  [p^{\diamond \diamond}_{\nu\nu'}]\big(\xi_{n,l}(t)\big)  + 4 p^{\diamond \diamond}_{\nu\nu'}\big(\xi_{n,l}(t)\big) + [\mathcal{L}_0 p^{\diamond \diamond}_{\nu\nu'}]\big(\xi_{n,l}(t)\big) \Big) \\
    &\qquad  -[\pi^{\diamond\diamond}_{l;\nu \nu'}\theta(t)] g'\big(\Phi_*\big(\xi_{n,l} (t)\big)\big) p^{\diamond \diamond}_{\nu\nu'}\big(\xi_{n,l}(t)\big) 
    \\
    & \qquad - [\tilde{\mathcal{R}}_{p_{\nu \nu'}^{\diamond \diamond}}]_{n,l}(t).
\end{align*}
Exploiting
the definition \eqref{eq:mr:def:l:omega} for
the operator $\mathcal{L}_0$ once more,
the second and third row sum to  $  c_* [\pi^{\diamond\diamond}_{l;\nu \nu'}\theta(t)]\dfrac{d}{d\xi} p_{\nu \nu'}^{\diamond\diamond}\big(\xi_{n,l}(t)\big)$.
This allows us to write
\begin{align*}
    [{\mathcal{R}}_{p_{\nu \nu'}^{\diamond \diamond}}]_{n,l}(t) & = [\pi^{\diamond\diamond}_{l;\nu\nu'}\dot{\theta}(t)] p^{\diamond\diamond}_{\nu\nu'}\big(\xi_{n,l}(t)\big) - [\pi^{\diamond\diamond}_{l;\nu \nu'}\theta(t)] \Big(\dot{\theta_l}(t)  - c_*\Big)   \dfrac{d}{d\xi}p^{\diamond\diamond}_{\nu \nu'}\big(\xi_{n,l}(t)\big)
    - [\tilde{\mathcal{R}}_{p_{\nu \nu'}^{\diamond \diamond}}]_{n,l}(t).
\end{align*}
The  estimate~\eqref{eqn:super_and_sub_sols_Rnux2_est}  now follows from Propositions~\ref{prop:lde:analysis_theta:differences} and
\ref{prop:analysis_theta:theta_approx_quadratic} in combination with
Lemma~\ref{lem:th:bnd:pi:fncs}.
\end{proof}

\begin{lemma}\label{lemma:super_and_sub_sols:Rq_nu_diamondx2}
Consider the setting of Proposition~\ref{prop:super_and_sub_sol:supersolution_proposition}.
Then there exist constants $\delta>0$, $M > 0$
so that for any $\theta \in C^1\big([0,\infty);\ell^\infty(\Z)\big)$
that satisfies the LDE~\eqref{eqn:main:main_eqn_for_theta}
with 
$[\theta(0)]_{\mathrm{dev}} < R$ and $\norm{\partial \theta(0)}_{\ell^\infty}\leq \delta$
and any pair of  functions 
$z, Z \in C([0,\infty);\R)$,
we have the estimate
\begin{equation}\label{eqn:super_and_sub_sols:Rqnux2_est}
    \norm{\mathcal{R}_{q_{\nu \nu'}^{\diamond \diamond}}(t)}_{\ell^\infty} \le M \min\left\{\norm{\partial \theta(0)}_{\ell^\infty}, t^{-\frac{3}{2}}\right\}, \qquad
    t > 0.
\end{equation}
\end{lemma}
\begin{proof}
Proceeding as in Lemma \ref{lemma:super_and_sub_sols:Rp_nu_diamondx2}, we first fix a pair
$(\nu, \nu') \in \left\{1, 2, 3, 4\right\}^2$
and write
\begin{align*}
    \Delta^\times\left[ [\pi^{\diamond }_{\nu}\theta(t)][ \pi^{\diamond}_{\nu'}\theta(t)] q^{\diamond \diamond}_{\nu\nu'}\big(\xi(t)\big)
    \right]_{n,l}  
    & = [\pi^{\diamond }_{l;\nu}\theta(t)][ \pi^{\diamond}_{l;\nu'}\theta(t)]  \Big([T_{\nu''} q^{\diamond \diamond}_{\nu\nu'}]\big(\xi_{n,l}(t)\big) - 4q^{\diamond \diamond}_{\nu\nu'} \big(\xi_{n,l}(t)\big)\Big)
     \\
    &\indent  + [\tilde{\mathcal{R}}_{q_{\nu \nu'}^{\diamond \diamond}}]_{n,l}(t),
\end{align*}
Here we sum over $\nu''\in \{1, 2, 3, 4\}$ and use the expression
\begin{align*}
   [ \tilde{\mathcal{R}}_{q_{\nu \nu'}^{\diamond \diamond}} (t)]_{n,l}&= \pi^{\diamond }_{l;\nu''}\left[[\pi^{\diamond }_{l;\nu}\theta(t)][ \pi^{\diamond}_{l;\nu'}\theta(t)] \right]q^{\diamond \diamond}_{\nu\nu'}(\xi_{n,l}(t)]   \\
   &\indent  + [\pi^{\diamond }_{l;\nu}\theta(t)][ \pi^{\diamond}_{l;\nu'}\theta(t)]  \int_0^{-\pi_{l;\nu'}\theta(t)} \frac{d}{d\xi} [T_{\nu''} q_{\nu \nu'}^{\diamond \diamond}](\xi_{n,l+\sigma_{\nu'}} - u)\, du,
\end{align*}
which satisfies the bound~\eqref{eqn:super_and_sub_sols:Rqnux2_est} on account of
Proposition 
\ref{prop:lde:analysis_theta:differences} and Lemma~\ref{lem:th:bnd:pi:fncs}.

 Combining this with the definitions~$\eqref{eqn:super_sub_sols:all_Js}_4$ and~\eqref{eqn:super_sub_sols:R_q_nu_nu'}, we obtain
\begin{align*}
   [\mathcal{R}_{q_{\nu\nu'}^{\diamond \diamond}} (t)]_{n,l}
    &
    = \Big([\pi^{\diamond}_{l;\nu}\dot{\theta}(t)][\pi^{\diamond}_{l;\nu'}\theta(t)] + [\pi^{\diamond}_{l;\nu}\theta(t)][\pi^{\diamond}_{l;\nu'}\dot{\theta}(t)]\Big)q^{\diamond\diamond}_{\nu\nu'}\big(\xi_{n,l}(t)\big)   \\
     &\qquad  -
    [\pi^{\diamond}_{l;\nu}\theta(t)][\pi^{\diamond}_{l;\nu'}\theta(t)] \dot{\theta_l}(t)  \dfrac{d}{d\xi}q^{\diamond\diamond}_{\nu \nu'}\big(\xi_{n,l}(t)\big)    
    \\
    &\qquad  + [\pi^{\diamond }_{l;\nu}\theta(t)][ \pi^{\diamond}_{l;\nu'}\theta(t)]  \Big( [\mathcal{L}_0 q^{\diamond \diamond}_{\nu\nu'}]\big(\xi_{n,l}(t)\big) - [T_{\nu''} q^{\diamond \diamond}_{\nu\nu'}]\big(\xi_{n,l}(t)\big) + 4q^{\diamond \diamond}_{\nu\nu'} \big(\xi_{n,l}(t)\big)\Big)
     \\
      &
    \qquad - [\pi^{\diamond}_{l;\nu}\theta(t)][\pi^{\diamond}_{l;\nu'}\theta(t)]  g'\big(\Phi_*\big(\xi_{n,l} (t)\big)\big) q^{\diamond \diamond}_{\nu\nu'}\big(\xi_{n,l}(t)\big) \\
      &\qquad  - [\tilde{\mathcal{R}}_{q_{\nu \nu'}^{\diamond \diamond}}]_{n,l}(t) .
\end{align*}
Applying the
definition \eqref{eq:mr:def:l:omega} for
the operator $\mathcal{L}_0$
to simplify the third and fourth row, we arrive at
\begin{align*}
   [\mathcal{R}_{q_{\nu\nu'}^{\diamond \diamond}} (t)]_{n,l}
    &
    = \Big([\pi^{\diamond}_{l;\nu}\dot{\theta}(t)][\pi^{\diamond}_{l;\nu'}\theta(t)] + [\pi^{\diamond}_{l;\nu}\theta(t)][\pi^{\diamond}_{l;\nu'}\dot{\theta}(t)]\Big)q^{\diamond\diamond}_{\nu\nu'}\big(\xi_{n,l}(t)\big)   \\
     &\qquad  -
    [\pi^{\diamond}_{l;\nu}\theta(t)][\pi^{\diamond}_{l;\nu'}\theta(t)] \Big(\dot{\theta_l}(t)  - c_*\Big)\dfrac{d}{d\xi}q^{\diamond\diamond}_{\nu \nu'}\big(\xi_{n,l}(t)\big)    \\
      &\qquad  - [\tilde{\mathcal{R}}_{q_{\nu \nu'}^{\diamond \diamond}}(t)]_{n,l},
\end{align*}
The statement now follows from 
Propositions~\ref{prop:lde:analysis_theta:differences} and
\ref{prop:analysis_theta:theta_approx_quadratic} in combination with
Lemma~\ref{lem:th:bnd:pi:fncs}.
\end{proof}

\subsection{Final splitting}

Defining the aggregate quantities 
\begin{align}
    \mathcal{J}_{\text{apx}} &= \mathcal{J}_{\Phi;\text{apx}} + \mathcal{J}_{p_\nu^\diamond;\text{apx}} + \mathcal{J}_{p_{\nu\nu'}^{\diamond\diamond};\text{apx}} + \mathcal{J}_{q_{\nu\nu'}^{\diamond\diamond};\text{apx}}, \\
     \mathcal{R} &= \mathcal{R}_{\Phi} + \mathcal{R}_{p_\nu^\diamond} + \mathcal{R}_{p_{\nu\nu'}^{\diamond\diamond}} + \mathcal{R}_{q_{\nu\nu'}^{\diamond\diamond}}, \label{eqn:super_and_sub_sols:residual_R}
\end{align}
the results in \S\ref{subsec:super_sub_sols:J_phi}-\S\ref{subsec:super_sub_sols:J_qnux2} provide 
the decomposition
\begin{equation}\label{eqn:super_sub_sols:J(u)}
    \mathcal{J}[u^+] =  \mathcal{J}_{\text{apx}} + \mathcal{J}_{\text{glb}} + \mathcal{R},
\end{equation}
together with the explicit expression 
\begin{align*}
     [\mathcal{J}_{\text{apx}}(t)]_{n,l}  &= \Phi_*'\big(\xi_{n,l}(t)\big)\Big(-\dot{\theta}_l(t) + c_*\Big) \\
      &\indent   + [\pi^\diamond_{l;\nu}\theta(t)] \Big(- [\mathcal{L}_0 p_\nu^\diamond]\big(\xi_{n,l}(t)\big)  + [\tau_\nu \Phi_*']\big(\xi_{n,l}(t)\big)  \Big)\\
    &\qquad  + [ \pi^{\diamond \diamond}_{l;\nu \nu' }\theta(t)] \Big( \alpha_{p;\nu}^\diamond p_{\nu'}^\diamond \big(\xi_{n,l}(t)\big) - [T_{\nu'} p_{\nu}^\diamond]\big(\xi_{n,l}(t)\big)-[\mathcal{L}_0 p^{\diamond \diamond}_{\nu\nu'}]\big(\xi_{n,l}(t)\big)   \Big) \\
   &\qquad
    + [ \pi^\diamond_{l;\nu}\theta(t)] [ \pi^\diamond_{l;\nu'}\theta(t)] \left(      -    \alpha_{p;\nu'}^{\diamond} \dfrac{d}{d\xi}p_{\nu}^{\diamond}\big(\xi_{n,l}(t)\big)  +
 \dfrac{d}{d\xi} [T_{\nu'} p_{\nu}^\diamond] \big(\xi_{n,l}(t)\big) 
      \right) 
      \\
    & \qquad
    +[ \pi^\diamond_{l;\nu}\theta(t)] [ \pi^\diamond_{l;\nu'}\theta(t)] \left(- \dfrac{1}{2} \textbf{1}_{\nu = \nu'}[T_{\nu}  \Phi_*'']\big(\xi_{n,l}(t)\big) -[\mathcal{L}_0 q^{\diamond \diamond}_{\nu\nu'}]\big(\xi_{n,l}(t)\big) \right)
    \\
     &\qquad  + \Big( u^+_{n,l}(t) - \Phi_*\big(\xi_{n,l}(t)\big) - z(t) \Big)g'\big(\Phi_*\big(\xi_{n,l} (t)\big)\big)  + g\big(\Phi_*\big(\xi_{n,l}(t)\big)\big).
       \end{align*}
Recalling the MFDEs~\eqref{eqn:main:eqns_for_p}
we can reduce $\mathcal{J}_\mathrm{apx}$ to
\begin{equation}\label{eqn:super_sub_sols:J_apx_splitting1}
\begin{aligned}
     \ [\mathcal{J}_{\text{apx}}(t) ]_{n,l}  &= \Phi_*'\big(\xi_{n,l}(t)\big)\Big(-\dot{\theta}_l(t) +   \alpha_{p;\nu}^{\diamond } \pi^{\diamond }_{l;\nu} \theta(t) + \alpha_{p;\nu\nu'}^{\diamond \diamond} \pi^{\diamond \diamond}_{l;\nu\nu'} \theta(t) + \alpha_{q;\nu\nu'}^{\diamond\diamond} [\pi^\diamond_{l;\nu}\theta(t)] [\pi^\diamond_{l;\nu'} \theta(t)] + c_*\Big) 
     \\
    & \qquad
    +\dfrac{1}{2} [ \pi^\diamond_{l;\nu}\theta(t)] [ \pi^\diamond_{l;\nu'}\theta(t)] g''\big(\Phi_*\big(\xi_{n,l}(t)\big)\big)p_{\nu}^\diamond\big(\xi_{n,l}(t)\big)p_{\nu'}^\diamond\big(\xi_{n,l}(t)\big)
    \\
    &\indent  + \Big(u^+_{n,l}(t) - \Phi_*\big(\xi_{n,l}(t)\big) - z(t) \Big)
        g'\big(\Phi_*\big(\xi_{n,l} (t)\big)\big) + g\big(\Phi_*\big(\xi_{n,l}(t)\big)\big) . 
\end{aligned}
\end{equation}
Recalling \eqref{eq:th:def:r:theta}, 
the first row of~\eqref{eqn:super_sub_sols:J_apx_splitting1} can be recognized as the expression $-\Phi_*'\big(\xi_{n,l}(t)\big)[\mathcal{R}_\theta(t)]_{n,l}$. Grouping the terms related to the function $g$ in $\mathcal{J}_{\mathrm{apx}}$ and $\mathcal{J}_{\mathrm{glb}}$, we introduce
the new function
\begin{align*}
    [\mathcal{J}_g(t)]_{n,l} &=  - g\big(u^+_{n,l}(t)\big)  +  \Big(u^+_{n,l}(t) - \Phi_*\big(\xi_{n,l}(t)\big) - z(t) \Big)
        g'\big(\Phi_*\big(\xi_{n,l} (t)\big)\big) + g\big(\Phi_*\big(\xi_{n,l}(t)\big)\big) 
        \\[0.3cm]
         & \qquad +\dfrac{1}{2} [ \pi^\diamond_{l;\nu}\theta(t)] [ \pi^\diamond_{l;\nu'}\theta(t)] g''\big(\Phi_*\big(\xi_{n,l}(t)\big)\big)p_{\nu}^\diamond\big(\xi_{n,l}(t)\big)p_{\nu'}^\diamond(\xi_{n,l}\big(t)\big) \\[0.3cm]
        & = - g\big(u^+_{n,l}(t)\big) + [\mathcal{J}_{\text{apx}}(t) ]_{n,l} + \Phi_*'\big(\xi_{n,l}(t)\big) [\mathcal{R}_\theta (t)]_{n,l}. \\
\end{align*}
Together with the residual
\begin{align*}
    [\mathcal{R}_\mathrm{rest}(t)]_{n,l}
    &= [\mathcal{J}_{\mathrm{glb}}(t)]_{n,l} + g\big(u^+_{n,l}(t)\big) -\Phi_*'\big(\xi_{n,l}(t)\big) [\mathcal{R}_\theta (t)]_{n,l} \\[0.3cm]
     &= \dot{Z}(t) \Big( \Phi_*'\big(\xi_{n,l}(t)\big)+ B_{n,l}(t) \Big)
         + \dot{z}(t) -\Phi_*'\big(\xi_{n,l}(t)\big) [\mathcal{R}_\theta (t)]_{n,l}
         \\
\end{align*}
this leads to the decomposition
\begin{equation}\label{eqn:super_and_sub_sols:J_apx_new_splitting}
    \mathcal{J}_{\mathrm{apx}} + \mathcal{J}_{\mathrm{glb}} = \mathcal{J}_g +  \mathcal{R}_\mathrm{rest}. \\
\end{equation}

Expanding $g\big(u^+_{n,l}(t)\big)$ around $g\big(\Phi_*\big(\xi_{n,l}(t)\big)\big)$ up to third order, we obtain the further reduction
\begin{equation}
    [\mathcal{J}_g(t)]_{n,l} =  [G^a(t)]_{n,l} + [G^b(t)]_{n,l}  - z(t) 
        g'\big(\Phi_*\big(\xi_{n,l} (t)\big)\big) ,
\end{equation}
where we have introduced the expressions
\begin{align*}
    G_{n,l}^a(t) &= \dfrac{1}{2} [ \pi^\diamond_{l;\nu}\theta(t)] [ \pi^\diamond_{l;\nu'}\theta(t)] g''\big(\Phi_*\big(\xi_{n,l}(t)\big)\big)p_{\nu}^\diamond\big(\xi_{n,l}(t)\big)p_{\nu'}^\diamond\big(\xi_{n,l}(t)\big) 
        \\
        & \qquad -\dfrac{1}{2} g''\big(\Phi_*\big(\xi_{n,l}(t)\big)\big) \Big(u^+_{n,l}(t) - \Phi^*\big(\xi_{n,l}(t)\big)\Big)^2, 
        \\
    G^b_{n,l}(t) &=  - \dfrac{1}{6} g'''\big( s_{n,l}(t)\big) \Big(u^+_{n,l}(t) - \Phi_*\big(\xi_{n,l}(t)\big)\Big)^3 
\end{align*} 
for an appropriate function $s_{n,l}(t) \in [\Phi_*\big(\xi_{n,l}(t)\big), u^+_{n,l}(t)]$.

In the following lemma, we  formulate an appropriate factorization for these new sequences $G^a(t)$ and $G^b(t)$. 
\begin{lemma}\label{lemma:super_sub_sols:splitting_for_G}
Consider the setting of Proposition~\ref{prop:super_and_sub_sol:supersolution_proposition}.
Then there exist constants $\delta>0$, $M > 0$ 
so that for any $\theta \in C^1([0,\infty);\ell^\infty)$
that satisfies the LDE~\eqref{eqn:main:main_eqn_for_theta}
with 
$[\theta(0)]_{\mathrm{dev}} < R$ and $\norm{\partial \theta (0)}_{\ell^{\infty}} \leq \delta $
and any pair of  functions 
$z, Z \in C([0,\infty);\R)$, with $\norm{z}_{L^\infty}\leq 1$, the following holds true. 
\begin{enumerate}[(i)]
    \item\label{item:super_sub_sols:splittings_G} For any $t > 0$ there exist sequences $H^a(t)$, $H^b(t)$, $\mathcal{R}^a(t)$ and $\mathcal{R}^b(t)$ in $\ell^\infty(\Z^2_\times)$ such that 
    the identities
\begin{align}
       G^a_{n,l}(t)  &= z(t) H^a_{n,l}(t) + \mathcal{R}^a_{n,l}(t), \label{eqn:super_sub_sols:splitting_G1}\\[0.3cm]
       G^b_{n,l}(t)  &= z(t) H^b_{n,l}(t) + \mathcal{R}^b_{n,l}(t) \label{eqn:super_sub_sols:splitting_G2}
\end{align}
    hold for all $(n,l) \in \Z^2_\times$.

    \item\label{item:super_sub_sols:splittings_R_estimates}
    For any $t > 0$ we have the estimate
    \begin{equation}\label{eqn:super_sub_sols:almost_final_splitting}
    \max\left\{\norm{\mathcal{R}^a(t)}_{\ell^\infty(\Zs)}, \norm{\mathcal{R}^b(t)}_{\ell^\infty(\Zs)}\right\} \le M \min\left\{\norm{\partial \theta(0)}_{\ell^\infty}, t^{-\frac{3}{2}}\right\}.
\end{equation}
\item \label{item:super_sub_sols:bounds_H}  For any $t > 0$ the sequences $H^a(t)$ and $H^b(t)$ satisfy the bound
\begin{align*}
    \max\left\{\norm{H^a(t)}_{\ell^\infty(\Zs)}, \norm{H^b(t)}_{\ell^\infty(\Zs)}\right\} &\leq M(\norm{z}_{L^\infty} + \delta).
    \end{align*}
\end{enumerate}
\end{lemma}
\begin{proof} 
For convenience, we introduce the shorthand
\begin{align*}
    K_{n,l}(t) :&= u^+_{n,l}(t) - \Phi_*\big(\xi_{n,l}(t)\big) - z(t) 
    \\
    &=[\pi^\diamond_{l;\nu}\theta(t)] p^\diamond_\nu\big(\xi_{n,l}(t)\big) + [\pi^{\diamond\diamond}_{l;\nu \nu'}\theta(t)] p^{\diamond\diamond}_{\nu \nu'}\big(\xi_{n,l}(t)\big) + 
    [\pi^{\diamond}_{l;\nu}\theta(t)][\pi^{\diamond}_{l;\nu'}\theta(t)] q^{\diamond\diamond}_{\nu \nu'}\big(\xi_{n,l}(t)\big),
\end{align*}
which allows us to rewrite $G^a_{n,l}(t)$ as
\begin{align*}
      G_{n,l}^a(t) &= \dfrac{1}{2}g''\big(\Phi_*\big(\xi_{n,l}(t)\big)\big)\Big( [ \pi^\diamond_{l;\nu}\theta(t)] [ \pi^\diamond_{l;\nu'}\theta(t)] p_{\nu}^\diamond\big(\xi_{n,l}(t)\big)p_{\nu'}^\diamond\big(\xi_{n,l}(t)\big) 
       - \big(K_{n,l}(t) + z(t)\big)^2 \Big) .
\end{align*}
The expression
\begin{equation*}
    \tilde{\mathcal{R}}^a_{n,l}(t) :=  [\pi^\diamond_{l;\nu}\theta(t)] [\pi^\diamond_{l;\nu'}\theta(t)]  p^\diamond_\nu\big(\xi_{n,l}(t)\big) p^\diamond_{\nu'}\big(\xi_{n,l}(t)\big)  - \big(K_{n,l}(t)\big)^2 
\end{equation*}
satisfies the estimate~\eqref{eqn:super_sub_sols:almost_final_splitting} by
Proposition 
\ref{prop:lde:analysis_theta:differences} and Lemma~\ref{lem:th:bnd:pi:fncs},
which in turn gives the splitting~\eqref{eqn:super_sub_sols:splitting_G1} upon defining
\begin{equation}\label{eqn:super_sub_sols:def_H1_R1}
    \begin{aligned}
      H_{n,l}^a(t) &= -\dfrac{1}{2}g''\big(\Phi_*\big(\xi_{n,l}(t)\big)\big)\Big(  z(t) + 2K_{n,l}(t) \Big)  ,
      \\
      \mathcal{R}_{n,l}^a(t) &= \dfrac{1}{2}g''\big(\Phi_*\big(\xi_{n,l}(t)\big)\big) \tilde{\mathcal{R}}^a_{n,l}(t) .
\end{aligned}
\end{equation}

To obtain the splitting~\eqref{eqn:super_sub_sols:splitting_G2}, we first notice that  $\big(K_{n,l}(t)\big)^3$ already satisfies the estimate~\eqref{eqn:super_sub_sols:almost_final_splitting} by
Proposition 
\ref{prop:lde:analysis_theta:differences} and Lemma~\ref{lem:th:bnd:pi:fncs}.
In order to establish items (i) and (ii), it therefore suffices to write
\begin{equation}\label{eqn:super_sub_sols:def_H2_R2}
\begin{aligned}
      H_{n,l}^b(t) &= \dfrac{1}{6} g'''\big( s_{n,l}(t)\big)\Big(  z^2(t) + 3K_{n,l}(t) z(t) + 3\big(K_{n,l}(t)\big)^2 \Big),
      \\
      \mathcal{R}_{n,l}^b(t) &= \dfrac{1}{6} g'''\big( s_{n,l}(t)\big) \big(K_{n,l}(t)\big)^3 .
\end{aligned}
\end{equation}
 Item~\textit{(\ref{item:super_sub_sols:bounds_H})} finally follows from the definitions of $H^a$ and $H^b$ and the fact that the functions $g''$ and $g'''$ are bounded on  compact intervals. 
\end{proof}
We are now finally ready to define our final splitting. 
Setting 
\begin{equation}\label{eqn:super_sub_sols:formula_H}
    H(t)  = H^a(t) + H^b(t)
\end{equation} 
we write
\begin{equation}\label{eqn:super_sub_sols:final_splitting}
    \mathcal{J}[u^+] = \mathcal{J}_{\mathrm{apx;fin}} + \mathcal{R}_\mathrm{fin},
\end{equation}
where the quantities $\mathcal{J}_{\mathrm{apx;fin}} $ and $\mathcal{R}_\mathrm{fin}$ are defined by
\begin{align}
    [\mathcal{J}_{\mathrm{apx;fin}}]_{n,l}(t) &=  \dot{Z}(t) \Big( \Phi_*'\big(\xi_{n,l}(t)\big) + B_{n,l}(t) \Big) + z(t) \Big(-g'\big(\Phi_*\big(\xi_{n,l}(t)\big)\big) + H_{n,l} (t)\Big)+ \dot{z}(t), \label{eqn:super_and_sub_sol:final_split_J_apx}
    \\[0.2cm]
    \mathcal{R}_\mathrm{fin}(t) &= \mathcal{J}(t) - \mathcal{J}_{\mathrm{apx;fin}}(t).
    \label{eqn:super_and_sub_sol:final_split_R}
\end{align}

\begin{lemma}\label{lemma:super_and_sub_sols:R_final}
Consider the setting of Proposition~\ref{prop:super_and_sub_sol:supersolution_proposition}.
Then there exist constants $\delta>0$, $M > 0$
so that for any $\theta \in C^1\big([0,\infty);\ell^\infty(\Z)\big)$
that satisfies the LDE~\eqref{eqn:main:main_eqn_for_theta}
with 
$[\theta(0)]_{\mathrm{dev}} < R$ and $\norm{\partial \theta(0)}_{\ell^\infty} < R$
and any pair of  functions 
$z, Z \in C([0,\infty);\R)$ with $\norm{z}_{L^\infty} \leq 1 $,
we have the estimate
\begin{equation}
    \norm{\mathcal{R}_\mathrm{fin}(t)}_{\ell^\infty(\Zs)} \le M \min\left\{\norm{\partial \theta(0)}_{\ell^\infty}, t^{-\frac{3}{2}}\right\},
    \qquad 
    t > 0.
\end{equation}
\end{lemma}
\begin{proof}
Comparing equations \eqref{eqn:super_sub_sols:J(u)}, \eqref{eqn:super_and_sub_sols:J_apx_new_splitting} and ~\eqref{eqn:super_and_sub_sol:final_split_J_apx} we can explicitly identify $\mathcal{R}_\mathrm{fin}(t)$ as
\begin{equation*}
    [\mathcal{R}_\mathrm{fin}(t) ]_{n,l}= -\Phi_*'\big(\xi_{n,l}(t)\big) [\mathcal{R}_\theta (t)]_{n,l} + \mathcal{R}_{n,l}(t) + \mathcal{R}_{n,l}^a(t) + \mathcal{R}_{n,l}^b(t).
\end{equation*}
The statement now follows from Proposition~\ref{prop:analysis_theta:theta_approx_quadratic} 
in combination with
the definition~\eqref{eqn:super_and_sub_sols:residual_R}
and Lemmas~\ref{lemma:super_and_sub_sols:RPhi}, ~\ref{lemma:super_and_sub_sols:Rp_nu_diamond}, ~\ref{lemma:super_and_sub_sols:Rp_nu_diamondx2}, \ref{lemma:super_and_sub_sols:Rq_nu_diamondx2} and \ref{lemma:super_sub_sols:splitting_for_G}.
\end{proof}

\subsection{Proof of Proposition~\ref{prop:super_and_sub_sol:supersolution_proposition}}
We are now finally ready to prove Proposition~\ref{prop:super_and_sub_sol:supersolution_proposition}. As a first step, we show how to pick all the constants and functions appearing in the statement.  
Without loss of generality, we assume
that the constant $M$
from Lemma~\ref{lemma:super_and_sub_sols:R_final}
satisfies 
\begin{equation}\label{eqn:super_and-sub_sols:maxM}
    M\geq\max\{1, 52 D, \sup_{s\in [0,1]} |g''(s)|, \sup_{s\in [0,1]} |g'''(s)| 
    \},
\end{equation}
where the constant $D$ is defined by
\begin{equation*}
    D = \max\{\norm{p_\nu ^\diamond}_{L^\infty(\R)}, \norm{p_{\nu \nu'} ^{\diamond \diamond}}_{L^\infty(\R)}, \norm{q_{\nu \nu'} ^{\diamond \diamond}}_{L^\infty(\R)}, \norm{ [p_\nu ^\diamond]'}_{L^\infty(\R)}, \norm{  [p_{\nu \nu'} ^{\diamond \diamond}]'}_{L^\infty(\R)}, \norm{ [q_{\nu \nu'} ^{\diamond \diamond}]'}_{L^\infty(\R)} \}.
\end{equation*}
We pick a constant $m \in ( 3\epsilon, \frac{1}{2}] $ in such a way that
\begin{equation}\label{eqn:super_and_sub_sols:bound_g_prime}
    -g'(s) \geq 2m > 0, \text{ for } s\in [-\epsilon, \epsilon] \cup [1-\epsilon, 1+\epsilon],
\end{equation}
reducing $\epsilon$ if needed. Next, we define the positive constants
$$C_{\epsilon} = \max\{1, \dfrac{2m + 
M
}{\min_{\Phi_*\in [\epsilon, 1-\epsilon]}\Phi_*'}\}, \quad \qquad \delta_{\epsilon} = \dfrac{\epsilon^3m^3}{6^3 M^3 C_{\epsilon}^3},
\qquad \qquad
\nu_{\epsilon} =
\dfrac{\epsilon^3 m^2 }{3 \cdot 6^3 M^2C_{\epsilon}^3} = \dfrac{M\delta_{\epsilon}}{3m},
$$
together with the positive function
\begin{equation}
    \begin{array}{lcl}
     K_{\epsilon}:[0,\infty)\to \R,
& & t \mapsto  M \min\left\{\delta_{\epsilon}, t^{-\frac{3}{2}} \right\}.
    \end{array}
\end{equation}
We now choose a function $z\in C^\infty\big( [0, \infty) ;\mathbb{R}\big)$ that satisfies
$$K_{\epsilon}(t)\leq m z(t) \leq 2K_{\epsilon}(t), \qquad \qquad m|\dot{z}(t)|\leq 2\tilde{K}_{\epsilon}(t),$$ where $\tilde{K}_{\epsilon}$ is defined by
\begin{equation}
    \tilde{K}_{\epsilon}(t) = \begin{cases}
    0, \quad &t\leq \delta_{\epsilon}^{-\frac{2}{3}},
    \\
    \frac{3}{2}M t^{-\frac{5}{2}},  &t> \delta_{\epsilon}^{-\frac{2}{3}},
    \end{cases}
\end{equation}
which we recognize as the absolute value of the weak derivative of the function $K_{\epsilon}$. 
In addition, we define the function $Z\in C^\infty\left[0, \infty\right)$ by
$$ Z(t) = C_{\epsilon}\int_0^t z(s)\, ds. $$ 
\begin{proof}[Proof of Proposition~\ref{prop:super_and_sub_sol:supersolution_proposition}]

The functions $z$ and $Z$ are clearly nonnegative, with
\begin{align*}
z(0) - \norm{p_\nu^\diamond}_{L^\infty} \delta_{\epsilon} - 2\norm{p_{\nu \nu'}^{\diamond\diamond}}_{L^\infty} \delta_{\epsilon} - \norm{q_{\nu \nu'}^{\diamond\diamond}}_{L^\infty} \delta_{\epsilon}^2 \geq \dfrac{M\delta_{\epsilon}}{m} - 52D\delta_{\epsilon} 
\geq \dfrac{M\delta_{\epsilon}}{m}\left(1 - m \right) 
\geq \dfrac{M\delta_{\epsilon}}{2m} > \nu_{\epsilon}.
\end{align*}
Furthermore, we have $z(t) \leq \dfrac{2M\delta_{\epsilon}}{m}\leq \epsilon$,
together with
$$Z(t)\leq \dfrac{2C_{\epsilon}}{m}\int_0^\infty K_{\epsilon}(s) ds \leq \dfrac{6C_{\epsilon}}{m} M \delta_{\epsilon}^{\frac{1}{3}} = \epsilon.$$ In particular, items (ii)-(iv) are satisfied.
In addition,
using the fact that $z(t)\leq \frac{2M\delta_\epsilon}{m}$ in combination with item \textit{(\ref{item:super_sub_sols:bounds_H})} of Lemma~\ref{lemma:super_sub_sols:splitting_for_G},  
we obtain the crude a-priori bound
\begin{equation}
\label{eq:sub:sup:a:priori:bnd}
     \norm{H(t)}_{\ell^\infty(\Z^2_\times)}  \le \epsilon, \quad \text{for all } t\geq 0.  
\end{equation}
Turning to (i), Lemma~\ref{lemma:super_and_sub_sols:R_final}
implies that it suffices
to show that the 
residual \eqref{eqn:super_and_sub_sol:final_split_J_apx}
satisfies
$\mathcal{J}_{\mathrm{apx;fin}}(t) \geq K_{\epsilon}(t)$.
Introducing the notation
\begin{equation*}
   \mathcal{I}_A(t) = \frac{\dot{Z}(t)}{z(t)} \Phi_*'\big(\xi(t)\big),
   \quad 
    \mathcal{I}_B(t) = \frac{\dot{Z}(t)}{z(t)} B(t),
   \quad  
    \mathcal{I}_C(t) = H(t) ,  
    \quad 
    \mathcal{I}_D(t) = -g'\big(\Phi_*(\xi(t))\big),
    \quad 
    \mathcal{I}_E(t) = \frac{\dot{z}(t)}{z(t)} 
\end{equation*}
we see that
\begin{equation*}
    \mathcal{J}_{\mathrm{apx;fin}}
      =  
     z\big( \mathcal{I}_A 
      + \mathcal{I}_B + \mathcal{I}_C + \mathcal{I}_D + \mathcal{I}_E \big).
\end{equation*}
Using the observation 
\begin{equation}
    \frac{|\dot{z}(t)|}{z(t)} 
    \le \begin{cases} 0, \quad &t\leq \delta_{\epsilon}^{-\frac{2}{3}}, \\
    3t^{-1}, &t>\delta_{\epsilon}^{-\frac{2}{3}},
    \end{cases}
  \end{equation}
we obtain the global bounds
\begin{equation}
     \begin{array}{lclcl}
   |\mathcal{I}_B(t)|
       & \leq & C_{\epsilon}  M \delta_{\epsilon} \leq 
       & \leq & \dfrac{m}{3} ,
     \\[0.2cm]
     |\mathcal{I}_C(t)|  &\leq & \epsilon& \leq & \dfrac{m}{3} ,
     \\[0.2cm]
     \big|\mathcal{I}_E (t)\big|
	    & \le & 3 \delta_{\epsilon} ^{2/3}
       & \le & \dfrac{m}{3}.
     \end{array}
 \end{equation}

 When  $\Phi_*(\xi) \in (0,\epsilon] \cup[1-\epsilon, 1) $, we
 may use \eqref{eqn:super_and_sub_sols:bound_g_prime} to obtain the lower bound
 \begin{equation}
    \mathcal{I}_D
     \ge  2m.
 \end{equation}
 Together with $\mathcal{I}_A \ge 0$,
 this allows us
 to conclude
 \begin{equation}
 \label{eq:sub:sup:fin:bnd:j:apx}
     \mathcal{J}_{\mathrm{apx}}  \ge
    m z(t) \geq K_{\epsilon}(t).
 \end{equation}
 On the other hand,
 when 
 $\Phi_*(\xi) \in [ \epsilon, 1-\epsilon]$, we 
 have 
 \begin{equation}
     | \mathcal{I}_A | 
     \ge
     C_{\epsilon} \dfrac{2m + M}{C_{\epsilon}}
     \ge 2m + M,
     \qquad 
     | \mathcal{I}_D | 
     \le M,
 \end{equation}
 which again yields \eqref{eq:sub:sup:fin:bnd:j:apx}.
In a similar manner one can show that $\mathcal{J}[u^-] \leq 0$.     
\end{proof}

\section{Phase approximation and stability results}
\label{sec:asymp}
In this section we show that $\gamma$ can be well-approximated by $\theta$ after allowing sufficient time
for the interface to `flatten'.
This is achieved using
the sub- and super-solutions
constructed in {\S}\ref{sec:sub:sup}
and allows us to establish
Theorem \ref{thm:main_result:gamma_mcf}
and Theorem \ref{thm:main_results:gamma_approx_u}. In view of the preparatory work in {\S}\ref{sec:theta}-\ref{sec:sub:sup}
which accounts for the transition from horizontal to general rational propagation directions,
we can here simply appeal to the corresponding results in 
\cite[{\S}8-9]{jukic2019dynamics} to a large extent.

The main idea for our proof of
Theorem \ref{thm:main_result:gamma_mcf}
is to
compare the information on $\gamma$
resulting from the asymptotic description \eqref{eqn:main:gamma_approx_u}
with the phase information that can be derived
from \eqref{eqn:super_and_sub_sol:def_u_plus}-\eqref{eqn:super_and_sub_sol:def_u_minus}. In particular,
we capture the solution $u$
between the sub- and super-solutions
constructed in {\S}\ref{sec:sub:sup}
and exploit the monotonicity properties of $\Phi_*$.

\begin{lemma}
\label{lem:phase:theta:vs:gamma}
Assume that (H$g$), (H$\Phi$), (H$0$), (HS$)_1$ and (HS$)_2$ all hold and let $u$ be a solution of \eqref{eqn:main:discrete_AC_new} with the initial condition \eqref{eqn:main:initial_condition_new}. Then for every $\epsilon >0$, there exists a constant $\tau_\epsilon > 0$ so that 
for any $\tau \ge \tau_{\epsilon}$ the solution $\theta$ of the
LDE \eqref{eqn:main:main_eqn_for_theta}
with the initial value 
$\theta(0) = \gamma(\tau)$
satisfies 
\begin{equation}
\label{eqn:asymp:phi:bnds}
    |\Phi_*\big(n-\gamma_l(t) \big)  
    - \Phi_*\big(n -\theta_l(t-\tau)
    \big)| \le \epsilon \\
\end{equation}    
for all $(n,l) \in \Zs$ and $t \ge \tau$.
\end{lemma}
\begin{proof}
The proof is adapted from~\cite[Lemma 8.2]{jukic2019dynamics}. We restrict our attention to the upper bound
$\Phi_*\big(n-\gamma_l(t) \big)  
    \le \Phi_*\big(n -\theta_l(t-\tau)
    \big) + \epsilon$,
    noting that the lower bound follows in the same way. 
    
    Without loss of generality, we assume that $0 < \epsilon < 1$. Recalling the constant $\nu_{\epsilon}$
from Proposition~\ref{prop:super_and_sub_sol:supersolution_proposition}, 
 Theorem~\ref{thm:main_results:gamma_approx_u}
 and Lemma 
\ref{lemma:phase_gamma:gamma-ct_bounded}
 allow us to find 
 $\tau_\epsilon>0$ and $R >0$ for which
 the bounds
\begin{equation}\label{eqn:asymp:comp:init:bnd:u:i:j}
    \left| u_{n,l}(t) - \Phi_*\big(n-\gamma_l(t)\big) \right| \leq \frac{1}{2}\nu_{\epsilon}, 
    \qquad \qquad
    [\gamma(t)]_{\mathrm{dev}} \le R
\end{equation}
hold for all $(n,l)\in \Zs$
and $t \ge \tau_{\epsilon}$. 
We now recall the constant $\delta>0$ and the functions $z$ and $Z$ 
that arise by
applying Proposition~\ref{prop:super_and_sub_sol:supersolution_proposition} with our pair $(\epsilon, R)$. Decreasing $\delta$ if necessary, we may assume that $\epsilon > \delta$. After possibly increasing
$\tau_{\epsilon}$,
we may use Proposition~\ref{prop:large_time_behaviour:l_differences_gamma} to obtain 
\begin{equation}
\norm{\partial \gamma (\tau)}_{\ell^\infty}\leq \delta, \qquad \tau\geq \tau_\epsilon.
\end{equation}

We now recall the super-solution
$u^+$ defined in \eqref{eqn:super_and_sub_sol:def_u_plus}.
Our choice for $\theta(0)$ together with 
the
bounds \eqref{eq:int:sub:sup:bnd:init:p:zero}
and \eqref{eqn:asymp:comp:init:bnd:u:i:j}
imply that
\begin{equation}
\begin{aligned}
    u_{n,l}(\tau) &\leq
     \Phi_*\big(n-\gamma_l(\tau)\big) + p_\nu^\diamond\big(n-\gamma_l(\tau)\big)
      [\pi_{l;\nu}^\diamond \gamma(\tau)] +
      p_{\nu \nu'}^{\diamond \diamond}\big(n-\gamma_l(\tau)\big)
      [\pi_{l;\nu\nu'}^{\diamond \diamond} \gamma(\tau)] \\
      &\qquad +q_{\nu \nu'}^{\diamond \diamond}\big(n-\gamma_l(\tau)\big)
      [\pi_{l;\nu}^\diamond \gamma(\tau)] [\pi_{l;\nu'}^\diamond \gamma(\tau)] + z(0)
     \\
     &= u^+_{n,l}(0) 
     .
     \end{aligned}
\end{equation}    
In particular,
the comparison principle
for the
LDE~\eqref{eqn:main:discrete_AC_new} together with
the bound \eqref{eqn:asymp:comp:init:bnd:u:i:j}
implies that
\begin{equation}
 \Phi_*\big(n - \gamma_l(t) \big)
 \le 
 u_{n,l}(t) + \frac{1}{2} \nu(\epsilon) \leq
    u^+_{n,l}(t - \tau)
    + \frac{1}{2} \nu_{\epsilon},
    \qquad 
    \qquad
    t \ge \tau.
\end{equation}
On the other hand, Corollary
\ref{prop:lde:analysis_theta:differences} in combination with~\eqref{eqn:super_and_sub_sol:def_u_plus} allows us to obtain a constant $C>0$ for which we have
\begin{equation}
    u^+_{n,l}(t)
    - \Phi_*\big( n - \theta_l(t) \big)
    \le C \epsilon,
    \qquad
    \qquad
    t \ge 0.
\end{equation}
In particular, we see that
\begin{equation}
    \Phi_*\big( n  - \gamma_l(t) \big)
     \le \Phi_*\big( n - \theta_l(t - \tau) \big)
     + \frac{1}{2} \nu_{\epsilon}
     + C \epsilon,
     \qquad \qquad
     t \ge \tau,
\end{equation}
from which the statement can
readily be obtained.
\end{proof}

\begin{proof}[Proof of Theorem 
\ref{thm:main_result:gamma_mcf}
]
The result can be obtained by
following the proof of Proposition 8.1 in \cite{jukic2019dynamics}.
\end{proof}

\begin{proof} [Proof of Theorem~\ref{thm:mr:periodicity+decay:stability}]
The proof can be copied almost verbatim from \cite[\S 9]{jukic2019dynamics} up to the notational changes that we exhibited in the proof of Lemma~\ref{lem:phase:theta:vs:gamma}.

\end{proof}

\bibliographystyle{klunumHJ}
\bibliography{refs}
\end{document}